\documentclass[11pt,reqno]{amsart}

\usepackage{amsrefs}
\usepackage{amsmath}
\usepackage{amsthm}
\usepackage{graphicx}
\usepackage{subcaption}
\usepackage{tikz}
\usetikzlibrary{trees,snakes,external,backgrounds,patterns}
\usepackage{amsfonts}
\usepackage{amssymb}
\usepackage{color,xcolor}
\usepackage{bm}
\usepackage{hyperref}
\usepackage{enumerate}
\usepackage[margin=1.1in]{geometry}
\usepackage{mathtools}

\usepackage[mathscr]{euscript}

\usepackage{comment}
\usepackage{bbm}
\usepackage{subcaption}
\usepackage{enumitem}

\numberwithin{equation}{section}

\newtheorem{theorem}{Theorem}[section]
\newtheorem{lemma}[theorem]{Lemma}
\newtheorem{proposition}[theorem]{Proposition}
\newtheorem{corollary}[theorem]{Corollary}

\theoremstyle{definition} 

\newtheorem{definition}[theorem]{Definition}
\newtheorem{remark}[theorem]{Remark}

\def\Integ{{\mathcal I}}
\def\newdom{{\mathcal D}}

\def\Jac{{\mathscr J}}
\def\Curve{{\mathscr C}}
\def\cl{{\rm cl}}
\def\cldom{{\rm cl}(\bDa)}
\def\mstar{{m^*}}
\def\trans{\Gamma}

\def\E{{\mathbb E}}

\def\R{{\mathbb R}}
\def\SS{{\mathbb S}}
\def\N{{\mathbb N}}
\def\PP{{\mathbb P}}

\def\calU{{\mathscr{U}}}

\newcommand{\lan}{\langle}
\newcommand{\ran}{\rangle}

\def\mfn{F}

\def\cnew{\upsilon}
\def\gn{\bar{g}^n_\theta}
\def\gnt{\mathfrak{g}^n_\theta}

\def\newpsi{f}
\def\newdom{{\mathcal D}}

\def\Hess{{\mathcal H}}

\def\den{\xi}
\def\numer{\kappa}
\def\bV{\bar{V}}
\def\hV{\widehat{V}}
\def\empn{L^n_{\theta}}
\def\Snp{\bar{S}^{(n,p)}}
\def\Sn{\bar{S}^{n}}
\def\bDpa{\bar{D}_{p,a}}
\def\bDa{\bar{D}_a}
\def\bln{\bar{l}^n}

\theoremstyle{remark}

\usepackage{etoolbox}
\AtEndEnvironment{theorem}{\setcounter{claim}{0}}

\newcommand{\norm}[1]{\left\|#1\right\|}
\newcommand{\abs}[1]{\left\lvert#1\right\rvert}

\DeclareMathOperator{\hess}{Hess}

\newcommand{\Wp}{W^{(n,p)}}
\newcommand{\pnormal}{\gamma_p}
\newcommand{\normal}{\gamma_2}
\newcommand{\xstar}{a^*}
\newcommand{\sphere}{\mathbb{S}}
\newcommand{\dens}{f}
\newcommand{\newdens}{\widetilde{f}}
\newcommand{\smeas}{\sigma}

\def\Wpb{\mathscr{W}^{(n,p)}}
\def\Snpb{\bar{\mathscr{S}}^{(n,p)}}
\def\bDpab{\bar{\mathscr{D}}_{p,a}}
\def\bDab{\bar{\mathscr{D}}_a}
\def\Snb{\bar{\mathscr{S}}^{n}}

\begin{document}

\title{Geometric sharp large deviations for random projections of $\ell_p^n$ spheres and balls} 
	 	\author[Liao]{Yin-Ting Liao}  
	 	\author[Ramanan]{Kavita Ramanan}
                \thanks{The first author was supported by NSF Grant DMS-1954351 and a GSAA fellowship from the Taiwan Government. The second author was supported by the National Science Foundation under grant DMS-1713032 and by the Office of Naval Research under the Vannevar Bush Faculty Fellowship N000142112887.} 
	 	\address{Division of Applied Mathematics, Brown University, 182 George Street, Providence, RI 02912} 
	 	 \email{yin-ting\_liao@brown.edu, kavita\_ramanan@brown.edu}

\begin{abstract}  
  Accurate estimation of tail probabilities of projections of high-dimensional probability measures is of relevance in high-dimensional statistics and asymptotic geometric analysis.  Whereas large deviation principles identify the asymptotic exponential decay rate of probabilities, sharp large deviation estimates also provide the "prefactor" in front of the exponentially decaying term.
  For fixed  $p \in (1,\infty)$,  consider independent sequences
  $(X^{(n,p)})_{n \in \N}$ and $(\Theta^n)_{n \in \N}$  of random vectors with
$\Theta^n$ distributed according to the normalized cone measure on the unit $\ell_2^n$ sphere, and 
    $X^{(n,p)}$  distributed according to the normalized cone measure on
    the unit  $\ell_p^n$ sphere. 
    For almost every realization $(\theta^n)_{n\in\N}$ of $(\Theta^n)_{n\in\N}$,
    (quenched) sharp large deviation estimates are established for suitably normalized (scalar) projections of $X^{(n,p)}$ onto $\theta^n$, that are 
    asymptotically exact (as the dimension $n$ tends to infinity).   
    Furthermore,  the case when $(X^{(n,p)})_{n \in \N}$ is replaced with $(\mathscr{X}^{(n,p)})_{n \in \N}$,  where 
  $\mathscr{X}^{(n,p)}$ is  distributed according to the uniform (or normalized volume) measure on the unit 
   $\ell_p^n$ ball, is also considered.
 In  both cases, in contrast to the (quenched) large deviation rate function, the prefactor exhibits
    a  dependence on  the projection directions $(\theta^n)_{n \in\N}$ that encodes additional geometric information
    that enables one to  distinguish between projections of balls and spheres. 
    Moreover, comparison with numerical estimates obtained by direct computation and importance sampling shows that the obtained analytical expressions for tail probabilities provide good approximations even for moderate values of $n$.
    The results on the one hand provide more accurate quantitative estimates of tail probabilities of random projections of $\ell_p^n$ spheres than logarithmic asymptotics, and on the other hand, generalize classical sharp large deviation estimates in the spirit of Bahadur and Ranga Rao to a geometric setting. The proofs combine Fourier analytic and probabilistic techniques. Along the way, several results of independent interest are obtained including  a simpler representation for the quenched  large deviation rate function that shows that it is strictly convex, a central limit theorem for random projections  under a certain family of tilted
 measures, and  multi-dimensional generalized Laplace asymptotics. 
 \end{abstract}

\setcounter{tocdepth}{1}
\maketitle

\tableofcontents 

\section{Introduction}
\label{sec-intro}

\subsection{Motivation and context}

The study of high-dimensional norms,   the convex bodies that describe their level sets, 
 and other high-dimensional geometric structures 
 are  central themes in geometric functional analysis \cite{MilSchGAFA04},
 and the
burgeoning field of asymptotic geometric analysis \cite{ArtGiaMil15}. 
 Several results in these fields have
shown that the presence of high dimensions  often  
imposes a certain regularity that has a probabilistic flavor.
A significant result of this type is the 
central limit theorem (CLT) for convex sets \cite{Kla07}   which, roughly speaking, 
says that if $X^n$ is  a high-dimensional random vector 
uniformly distributed on an isotropic convex body (namely, a compact convex set with non-empty interior whose normalized volume measure has zero mean and identity covariance matrix),   its one-dimensional scalar
projections $\langle X^n, \theta^n\rangle$ along most directions $\theta^n$ on
the unit $(n-1)$-dimensional sphere $\sphere^{n-1}$ in $\R^n$ have
Gaussian fluctuations.
In fact, this result holds for the larger class of isotropic logconcave measures
as well as more general high-dimensional measures that satisfy a certain concentration estimates called the thin shell condition (see, e.g.\ \cites{Sud78,vonWei97,Mec12b}). 
Of particular interest is the geometry of $\ell^n_p$ spaces, which has been classically studied using laws of large numbers, CLTs and concentration results \cites{Bob03,GueMil11,Schechtman90,Schmuck01}.
 These constitute  beautiful universality results that suggest that random projections 
 of the uniform measure on a convex body behave in some aspects like sums of
 independent  random variables. 
 On the other hand, they also imply the somewhat negative
 conclusion that typical fluctuations of lower-dimensional random projections  do not 
 yield much information about high-dimensional measures. 
 It is therefore natural  to ask  whether such random projections
 also satisfy  other properties  exhibited by sums of independent
 random variables, in particular those that 
 capture non-universal
 features that would yield useful information about the corresponding 
 high-dimensional measures.

 With this objective, large deviation principles (LDP) were
 established for suitably normalized  one-dimensional random projections of  $\ell_p^n$ balls  in  \cites{GanKimRam16,GanKimRam17}.  These  works
 established both quenched LDPs,  conditioned
 on the sequence of projection
 directions, as well as annealed LDPs, which average over the randomness of the projection directions.
Subsequently,  quenched LDPs 
  for multi-dimensional projections were obtained in \cite{skim-thesis}, and 
  annealed large deviation results for norms of $\ell_p^n$ balls and their multidimensional random  projections  were established 
 in \cites{Alonso18,Kabluchko17,KabProTha19b,KimLiaRam19}, with \cite{KabProTha19b} also considering
 moderate deviations (see also \cite{ProThaTur19} for a recent survey). 
 Going beyond the setting of $\ell_p^n$ balls (and measures with a similar
 representation),  annealed LDPs  were obtained for norms of 
 multidimensional projections 
 of more general sequences of high-dimensional random vectors
 $(X^n)_{n \in \N}$ that satisfy a so-called asymptotic thin shell condition
 in \cite{skim-thesis,KimLiaRam19}.  
  All these  LDPs are indeed  non-universal, in that the associated speeds (or exponential decay rates) and rate functions (that also captures the exponent) both
   encode properties of the high-dimensional measures. 
   However, although LDPs (in contrast to concentration results or large deviation upper bounds)
   identify the precise asymptotic exponential
   decay rate and allow for the identification of conditional limit laws \cite{KimRam18}, 
   they have the  drawback  that in general they only provide approximate estimates of the
   probabilities,   characterizing only the limit of the logarithms of the 
   deviation probabilities,  as the dimension $n$ goes to infinity.
 Thus, existing LDPs for random projections cannot be applied  directly to 
  provide accurate estimates of  tail probabilities or  develop efficient algorithms
 that  distinguish between two given high-dimensional  measures, 
 tasks that are  of importance in statistics, data analysis and computer science
 \cite{Diaconis84}.

 \subsection{Discussion of results}
 
 Our broad goal is to establish sharp (quenched) large deviation results of high-dimensional  measures
 that not only capture the precise asymptotic exponential decay rate
 of tail probabilities of random projections, but also their
``prefactors" (or the terms in front of the exponential),  so as to  provide  more accurate quantitative estimates  
 in  finite dimensions, much in the spirit of the local theory of Banach spaces. 
In addition, we  aim to identify additional geometric information that sharp large deviation 
estimates provide over LDPs. 
In this article, we focus on one-dimensional projections  of
 $\ell_p^n$ spheres and balls and obtain estimates of deviation probabilities that are asymptotically exact as the dimension goes to infinity.

 It is worthwhile to mention that 
 for the Euclidean norm of a random vector distributed on an isotropic convex body, 
 sharp large deviation upper bounds were obtained in several works
 (see, for example,  \cite{Kla07,FleAIHP10,Pao10,GueMil11} and references therein). 
 While these estimates have the very nice feature that they are
 universal (in that they apply for all isotropic convex bodies or, more generally, 
 logconcave measures), that very feature also makes them not tight for
 many specific sub-classes of convex bodies. 
 As a consequence, our proof techniques are different from those used in the latter works,
 and may be of independent interest.  
 In addition, we develop and analyze importance sampling algorithms
 to compute geometric quantities such as the volume fraction of small $\ell_p^n$ spherical caps in a certain direction, which would be infeasible to compute with reasonable accuracy using standard Monte Carlo estimation
 since the quantities are vanishingly small. 
 We expect that such computational approaches based on  large deviations
 may be useful more generally in the study of
  high-dimensional geometric structures. Indeed, the first version of this article has already spurred further work in this direction. For example, Kaufmann \cite{Kau21} studied annealed (i.e., averaged over the randomness of $\Theta$)  sharp large deviation estimates for $q$-norms of random vectors uniformly distributed on $\ell^n_p$ balls, and the paper \cite{KimRam21} establishes  
      quenched large deviation estimates for multi-dimensional projections
      of $\ell^n_p$ balls and their norms.

We now describe some of the challenges in obtaining such sharp estimates and comment on our proof technique. Our results can be viewed as 
 a geometric generalization of classical sharp large deviation 
 estimates in the spirit of   Bahadur and Ranga Rao \cite{BahRao93},
 which we now briefly recall.  
 Given a  sequence  of independent and identically distributed
 (i.i.d.) random variables $(X_i)_{i \in \N}$,  for each $n \in \N$, 
 let $S^n$ denote the corresponding empirical mean: 
 \begin{equation}
   \label{rep-sn}
   S^n :=  \frac{1}{n} \sum_{i=1}^n X_i^n = \frac{1}{\sqrt{n}} \left\langle X^n, \mathfrak{I}^n \right\rangle, 
 \end{equation}
where $X^n := (X_1, \ldots, X_n)$ and $\mathfrak{I}^n:=\frac{1}{\sqrt{n}}(1,1,\cdots,1)\in \sphere^{n-1}$.
  Under suitable assumptions on the (marginal) distribution of $X_1$ it was shown in 
 \cite{BahRao93} that 
 \begin{equation}\label{BRR}
\mathbb{P}\left(S^n \geq a\right)=\frac{e^{-n\mathbb{I}(a)}}{\bar{\sigma}_a\tau_a\sqrt{2\pi n}}\left(1+o(1)\right),
 \end{equation}   
where $\mathbb{I}$ is the Legendre transform of $\Lambda$, the logarithmic moment generating function of $X_1$, $\tau_a>0$ and $\bar{\sigma}_a>0$ are suitable constants specified below and $o(1)$ indicates a term $\varepsilon_n$ that satisfies $\varepsilon_n \to 0$ as $n\to\infty$.
  Key ingredients of the proof in \cite{BahRao93} include first identifying
   a ``tilted'' measure under which the  rare event
 on the left-hand side of \eqref{BRR} becomes typical, and  second, establishing  a quantitative CLT for the sequence $(S^n)_{n \in \N}$ under the tilted measure.  
Specifically, 
this tilted measure is also another product measure of the form $\otimes \widetilde{\mathbb{P}}_a,$ where $\widetilde{\mathbb{P}}_a$ is a measure absolutely continuous with respect to $\mathbb{P}$,  with
Radon-Nikodym derivative given by 
 \[
\frac{d\widetilde{\mathbb{P}}_a}{d\mathbb{P}}(x):=e^{\tau_ax -\Lambda(\tau_a)},
\] 
where $\tau_a$ is the unique positive constant such that $X_1$ has mean $a$ under the marginal $\widetilde{\mathbb{P}}_a$ of the tilted measure.  The constant $\bar{\sigma}_a^2$ in \eqref{BRR} is the variance of $X_1$ under $\widetilde{\mathbb{P}}_a$.
The second step of establishing a quantitative CLT is in this case standard  
given the product form of the tilted measure, and appeals to well known 
 Edgeworth expansions that also involve the third moment of $S_n$ under the tilted measure $\otimes^n\widetilde{\mathbb{P}}_a$.

Fix $p\in(1,\infty)$, and  let the projection direction $\Theta^n$ be distributed according to the normalized surface measure on $\sphere^{n-1}$, and let $X^{(n,p)}$ be a random vector independent of $\Theta^n$ that is uniformly distributed on  the  unit $\ell^n_p$ ball.
 In this article  we 
 obtain estimates of tail probabilities of 
 the scaled random projection 
 \begin{equation}\label{W}
\Wp :=  \frac{n^{1/p}}{n^{1/2}} \left\lan X^{(n,p)}, \Theta^n \right\ran = \frac{1}{n} \sum_{i=1}^n \left(n^{1/p}X^{(n,p)}_i\right)\left(n^{1/2}\Theta^n_i\right), 
 \end{equation} 
 conditioned on $\Theta=(\Theta^n)_{n\in\N}=\theta=(\theta^n)_{n\in\N}$, for a.e.\ realization  $\theta$ of $\Theta$. Using terminology that originates in statistical physics, due to the fact that we  condition on the realization
 $\theta$ of $\Theta$ and obtain results
 for almost every realization, we refer to these as ``quenched" deviation estimates.
 While (quenched) sharp large deviations of sums of weighted 
 i.i.d.\ random variables with i.i.d.\ weights 
 have been considered in more recent work \cite{Bovier15},   
 comparing  the expressions for $\Wp$ and $S^n$ in \eqref{W} and \eqref{rep-sn}, respectively,
 we see that $\Wp$ is a randomly  weighted sum of random variables that are not independent,
  with random weights that are also not independent.
 Thus, the analysis in this case  is significantly more challenging
 and requires several new ingredients.
 First, we  instead exploit a known probabilistic representation for the cone
 measure on $\ell^n_p$ spheres \cite{Schechtman90}  to rewrite the tail event 
 $\{\Wp \ge a)$   as the  probability that a certain two-dimensional random vector  lies  
 in a certain domain in $\R^2$ (see Section \ref{subs-reform}),
 and then establish  
 sharp large deviation estimates for the latter.
 This transformation turns out to be  useful even though 
  sharp large deviations in multiple dimensions
  are more involved, and none of the existing results
  (see, e.g., \cites{Andriani97,Barbe05,Joutard17} and references therein)
  apply  to this setting.    We  use Fourier analysis  and a change
  of measure argument   to obtain an asymptotic expansion for the quenched two-dimensional density (see Proposition \ref{prop-mainest} and Section \ref{sec-pfdens})
    and then integrate this density over the appropriate domain.
  To identify the appropriate change of measure or ``tilted'' measure, we first  show (in Lemma \ref{inf_ratefunc}) that the quenched large deviation rate function obtained in~\cite{GanKimRam17} is strictly convex and has a unique minimizer.
  Along the way, we also establish several results of possible independent result including 
    quantitative central limit theorems under the change 
    of measure (see Lemma \ref{clt_expansion}) and multi-dimensional generalized Laplace asymptotics
    (see Proposition \ref{asymptotic-exp}).
  
   In addition,  we also obtain corresponding results for $\ell^n_p$ balls, where 
     $X^{(n,p)}$ is replaced with $\mathscr{X}^{(n,p)}$, a random variable independent of $\Theta^n$ distributed according to the normalized volume measure on a scaled $\ell^n_p$ ball. Obtaining sharp large deviation estimates for random projections of $\ell^n_p$ balls  is substantially more complex than the $\ell^n_p$ sphere setting because the probability of interest is now expressed as an integral over a three-dimensional domain whose boundary is non-smooth at the minimizing point of the Laplace-type functional (see Section \ref{subs-reform}). This
    leads to additional difficulties in the  computation  of the  associated Laplace-type asymptotic integral 
    (see Lemma \ref{asymptotic-exp}).  As elaborated in  Remarks  
    \ref{rem-ballsphere} and \ref{rem-geom}, 
our analytical sharp large deviation
    estimates  do indeed capture additional geometric
    information beyond the large deviation rate function, and in fact we show that
    there is a clear  difference 
    between sharp tail probabilities in
     $\ell_p^n$ balls and spheres, even though they share the same large deviation rate function.  
Analogous sharp large deviation asymptotics can also be obtained in the case $p=\infty$ or, in fact, for more general
  product measures;  the analysis in this case is much  easier (see, e.g., \cite[Section 4.2]{LiaoThesis22}).

   In order to provide evidence of the accuracy of our sharp analytical estimates of the deviation probabilities for finite $n$, we compare them with numerical approximations. Specifically, we use the tilted measure identified in the sharp large deviations analysis to propose an importance sampling scheme that numerically approximates the deviation probabilities. We then compare the estimates obtained from importance sampling with analytical sharp large deviation estimates for a range of $n$.

   \subsection{Outline of the rest of the paper. }
   
  After a summary of common  notation and terminology in Sections \ref{subs-not} and \ref{subs-pre-notation},    precise statements of the main results
    are presented in Sections \ref{subs-analres} and \ref{subs-analres-b}.  
  An importance sampling algorithm for calculating tail probabilities and comparisons with resulting simulations and the  obtained analytical formulas are presented in Section~\ref{IS-algorithm}. 
    The main results rely on an asymptotic independence result for the weights induced
    by the projection direction, which is obtained in Section \ref{subs-prep},
    as well as a  reformulation of the rare event of interest as the event that
 a certain  random vector lies in a two-dimensional (or three-dimensional) domain, which is 
  described in Section \ref{subs-reform}.   Section \ref{subs-reform} also contains  
  an outline of the proofs of the main results, with  the complete proofs of the refined quenched tail estimates
  given in Sections  \ref{subs-mainproof} and \ref{subs-mainproof2} 
  for projections of $\ell_p^n$ spheres,   and in Section \ref{sec-pfmain-b}  for projections of  $\ell_p^n$ balls.
  Both proofs  proceed by first performing 
  asymptotic expansions for the joint densities of the multi-dimensional random vectors, as 
 formulated in Section \ref{subs-densest}.  These expansions are derived from 
   a general result on multi-dimensional generalized Laplace approximations obtained
  in Section \ref{subs-gen-laplace} (see Propositions \ref{asymptotic-exp} and \ref{lem-Laplace} therein)
and estimates obtained in 
  Sections \ref{subs-prep2} and  \ref{pf-estimate}, which  justify the applicability 
  of these  approximations in the present context.   
  Proofs of several  technical results  used in the analysis are deferred to
 Appendices \ref{apsec-infrfn}--\ref{app-GC-class}.

  \subsection{Notation and definitions}
  \label{subs-not}

We use the notation $\mathbb{N}$, $\mathbb{R}$ and $\mathbb{C}$ to denote the set of positive integers, real numbers and
complex numbers, respectively. For a complex number $z\in\mathbb{C}$, we denote $\operatorname{Re}\{z\}$ to be the real part of $z$.  For a set $A$, we denote its complement by $A^c$. Also, given a $m\times d$  matrix $M$,
  let $M^T$ denote its transpose and when  $m=d$, let 
 ${\rm det} M$ denote its determinant.

  Given an extended real-valued function $f:\R^d\to[0,\infty]$, its effective domain is defined as $\{x\in\R^d:f(x)<\infty\}$. For a twice differentiable function $f: \R^d \to \R$ (i.e., for which each partial 
  derivative $\partial_i \partial_j f$ exists for all $i, j \in \{1, \ldots, d\}$), 
let $\hess f (x)$ denote the $d\times d$ Hessian matrix  of $f$ at $x$.   For $q\in\mathbb{N}$, define the function space $\mathbb{L}_q(\mathbb{R}^d)$ to be
\[
\mathbb{L}_q(\mathbb{R}^d):=\left\{f:\mathbb{R}^d\to\mathbb{R}:\int_{\mathbb{R}^d}\abs{f}^qdx<\infty \right\}. 
\]

For $p\in (1,\infty)$ and $n\in\mathbb{N}$, let $\|\cdot\|_{n,p}$ denote
the $p$-norm in $\mathbb{R}^n$, that is, for $x=(x_1,\ldots,x_n)\in\mathbb{R}^n$, 
\[
\norm{x}_{n,p}:=\left(\abs{x_1}^p+\cdots+\abs{x_n}^p \right)^{1/p}.
\]
Let $\mathbb{S}^{n-1}_p$ and  $\mathbb{B}^n_p$ denote the unit $\ell^n_p$ sphere and  ball, respectively:
\begin{equation}\label{p-sphere}
\mathbb{S}^{n-1}_p:=\{x \in \R^n: \norm{x}_{n,p} = 1\}\quad\text{and}\quad \mathbb{B}^n_p:=\{x \in \R^n:\norm{x}_{n,p} \leq1\}.
\end{equation}
Also, define the cone measure on $\ell^{n}_p$ as follows: for any Borel measurable set $A\subset \mathbb{S}^{n-1}_p$,
\begin{equation}\label{cone-meas}
\mu_{n,p}(A) := \frac{{\rm vol}([0,1]A)}{{\rm vol}(\mathbb{B}^n_p)},
\end{equation}
where $[0,1]A:=\{xa\in\mathbb{R}^n: x\in[0,1], a\in A \}$, and ${\rm vol}$ denotes Lebesgue measure.
Note that when $p\in\{1,2,\infty\}$, the (renormalized) cone measure coincides with the (renormalized) surface measure,
and is equal to the unique rotational invariant measure on $\sphere^{n-1}$ with total mass $1$. For the special case  $p = 2$, we use just $\|\cdot\|$ to  denote $\|\cdot\|_{n,2}$, the Euclidean norm on $\mathbb{R}^n$,
  $\sphere^{n-1}$ to denote $\mathbb{S}^{n-1}_2$ and $\sigma_n$ to denote $\mu_{n,2}$. 

We end this section with the definition of a large deviations principle (LDP); we refer to~\cite{DemZeiBook}
 for general background on large deviations theory. For $d \in \N$, let  $\mathcal{P}(\R^d)$ denote
 the space of probability measures on $\R^d$, equipped with the topology of weak convergence, where recall
 that  for $\eta, \eta_n \in {\mathcal P}(\R^d)$, $n \in \N$,  $\eta_n$ is said to converge weakly to
 $\eta$ as $n \rightarrow \infty$, denoted $\eta_n \Rightarrow \eta$, if $\int_{\R^d} f(x) \eta_n(dx) \rightarrow \int_{\R^d} f(x) \eta (dx)$ as $n \rightarrow \infty$ for every bounded and continuous function $f$ on $\R^d$. 

\begin{definition}[Large deviation principle]
The sequence of probability measures $(\eta_n)_{n\in\mathbb{N}}\subset\mathcal{P}(\R^d)$ is said to satisfy a large deviation principle (in $\R^d$) with (speed $n$ and) a good rate function $\mathbb{I}:\mathbb{R}^d\to[0,\infty]$ if $\mathbb{I}$ is lower semicontinuous and for any measurable set $A$,
\[
-\inf_{x\in A^\mathrm{o}}\mathbb{I}(x)\leq\liminf_{n\to\infty}\frac{1}{n}\log \eta_n(A)\leq\limsup_{n\to\infty}\frac{1}{n}\log \eta_n(A)\leq-\inf_{x\in \cl(A)}\mathbb{I}(x),
\]
where $A^\mathrm{o}$ and $\cl(A)$ denote the interior and closure of $A$, respectively.
Moreover, we say that $\mathbb{I}$ is a good rate function if it has compact level sets.
A sequence of random variables $(V_n)_{n\in\mathbb{N}}$ with each $V_n$ defined on some probability space
  $(\Omega_n, {\mathcal F}_n, \mathbb{P}_n),$ is said to satisfy an LDP if the corresponding sequence of laws
$(\mathbb{P}_n^{-1}\circ V_n)_{n\in\mathbb{N}}$ satisfies an LDP.
\end{definition}

\section{Statement of main results}
\label{sec-res}

Fix $p \in (1,\infty)$. Consider a probability space  $(\Omega, {\mathcal F}, \PP)$ on which are defined
three  independent sequences  $X = (X^{(n,p)})_{n \in \N}$  and $\mathscr{X}=(\mathscr{X}^{(n,p)})_{n \in \N}$,
and $\Theta = (\Theta^n)_{n \in \N}$.  Each $X^{(n,p)}$ is distributed according to the cone
measure $\mu_{n,p}$ on the unit $\ell_p^n$ sphere, as defined in \eqref{cone-meas}, and each $\mathscr{X}^{(n,p)}$ is distributed according to the
normalized volume measure on the unit $\ell_p^n$ ball $\mathbb{B}^n_p$ defined in \eqref{p-sphere}.
The random element $\Theta$  takes values in the sequence space $\mathbb{S} := \otimes_{n\in\mathbb{N}} \sphere^{n-1}$, 
with $\Theta^n\in \sphere^{n-1}$ denoting  
the $n$-th element of
that sequence,  and is independent of $X$ (and $\mathscr{X}$) with
  distribution $\smeas$, where $\smeas$ is any probability measure
on $\SS$ whose image under the mapping $\theta \in \SS \mapsto \theta^n \in \sphere^{n-1}$ coincides with  $\sigma_n$, 
the unique rotation invariant measure  on  $\sphere^{n-1}$. 
The dependence between the random vectors $\Theta^n$ for different $n \in \N$ can be
arbitrary.
For $\theta\in\mathbb{S}$,  denote $\mathbb{P}_\theta$ to be the probability measure $\PP$ conditioned on $\Theta = \theta$, and let $\E$ and $\E_{\theta}$
denote expectation
with respect to $\PP$ and $\PP_{\theta}$, respectively.
For $n\in\mathbb{N}$, let $\Wp$ be the normalized scalar projection of $X^{(n,p)}$ along $\Theta^n$
  defined as
\begin{equation}\label{W_p}
\Wp := \frac{n^{1/p}}{n^{1/2}} \sum_{i=1}^n X_i^{(n,p)}\Theta_i^n,
\end{equation}
 and similarly let $\Wpb$ be the normalized scalar projection of $\mathscr{X}^{(n,p)}$ defined as
\begin{equation}\label{W_pb}
\Wpb := \frac{n^{1/p}}{n^{1/2}} \sum_{i=1}^n \mathscr{X}_i^{(n,p)}\Theta_i^n.
\end{equation}

First, in Section \ref{subs-pre-notation}, we introduce notation that is required to state the quenched sharp large deviation estimates. In Section \ref{subs-analres} we recall the
  quenched LDP for $\ell_p^n$ spheres and balls established in \cite{GanKimRam17}  and obtain an important
  simplification of the  quenched LDP rate function obtained therein, which in  particular
  shows that it is convex and has a unique minimum. 
  The latter property will be crucial for our analysis. We then present our sharp large deviation
  results for projections of $\ell_p^n$ spheres.  Corresponding results for $\ell_p^n$ balls are presented
  in Section \ref{subs-analres-b}.  Finally, in Section \ref{subs-reform} we provide a brief
  outline of both proofs, and present a more detailed comparison of our results
  with   classical Bahadur-Ranga Rao bounds.

\subsection{Preliminary notation}
\label{subs-pre-notation}

Fix $p \in (1,\infty)$.
Let $\pnormal \in\mathcal{P}(\mathbb{R})$ be the $p$-Gaussian distribution with density  
\begin{align}\label{pNormal}
\dens_p (y):=\frac{1}{2p^{1/p}\Gamma(1+\frac{1}{p})}e^{-\abs{y}^p/p}, \quad y\in\mathbb{R}, 
\end{align}
where $\Gamma$ is the Gamma function.   For $t_1,t_2\in \mathbb{R}$, define the extended functions
\begin{align} \label{logmgf_lp}
\Lambda_p(t_1,t_2) := \log\left( \int_\mathbb{R} e^{t_1y+t_2\abs{y}^p}\gamma_p(dy)\right),
\end{align}
and   
\begin{equation}
  \label{psi}
  \Psi_p(t_1,t_2) := \int_\mathbb{R} \Lambda_p(ut_1,t_2)\gamma_2(du),
\end{equation}
and observe that they both have effective domain $\mathbb{D}_p:=\R\times (-\infty,1/p)$.
Also, let $\Psi^*_p$ be the Legendre transform of $\Psi_p$: 
\begin{equation}\label{eq:Psi}
\Psi^*_p(t_1,t_2) := \sup_{s_1,s_2\in\mathbb{R}} \{t_1s_1 + t_2 s_2 - \Psi_p(s_1,s_2)\},  \quad t_1,t_2\in \mathbb{R},
\end{equation}
and let
$\mathbb{J}_p\subset\mathbb{R}^2$ be the effective domain of $\Psi_p^*$: 
\begin{equation}\label{def-Jp}
\mathbb{J}_p:=\{(x_1,x_2)\in\mathbb{R}^2:\Psi_p^*(x_1,x_2)<\infty\}.
\end{equation}

Since by \cite[Lemma 5.8]{GanKimRam17}, 
the function $\Lambda_p$ defined in~\eqref{logmgf_lp} is strictly convex on its effective domain, which we denote by $\mathbb{D}_p$,
$\Psi_p$ is also strictly convex on $\mathbb{D}_p$.  
By  \cite[Lemma 5.9]{GanKimRam17}, $\Psi_p$ is essentially smooth, lower-semicontinuous and hence closed. Therefore by  \cite[Theorem 26.5]{RocBook70}, $\nabla \Psi_p$ is one-to-one and onto from the domain of $\Psi_p$ to $\mathbb{J}_p$. Thus, for each $(x_1,x_2)\in\mathbb{J}_p$ there exists a unique $\lambda_x$ such that $\lambda_x\in\mathbb{D}_p$ and $\nabla \Psi_p(\lambda_x)=x$. This in turn implies that $\lambda_x$ uniquely achieves the supremum in~\eqref{eq:Psi}, and hence that
\begin{equation}
\label{gradPsi}\nabla \Psi_p(\lambda_x) = x,
\end{equation}
and 
\begin{equation}\label{lambda_x}
\Psi^*_p(x)=\langle x,\lambda_x\rangle-\Psi_p(\lambda_x).
\end{equation}

\begin{remark}\label{rem-smooth}
Since $\Psi_p$ is a strictly convex infinitely differentiable function on $\mathbb{D}_p$,  the inverse function theorem and \eqref{gradPsi} imply that the mapping $\mathbb{J}_p\ni x\mapsto \lambda_x\in\mathbb{D}_p$ is also infinitely differentiable.
\end{remark}

\subsection{Results on projections of $\ell_p^n$ spheres}
\label{subs-analres}
We first state  quenched LDPs for the sequences $(\Wp)_{n \in \N}$ from \eqref{W_p} and $(\Wpb)_{n \in \N}$ from \eqref{W_pb}. 
It follows from \cite[Theorem 2.5]{GanKimRam17} that
for $\smeas$-a.e.\ $\theta$, under $\mathbb{P}_\theta$, the sequence $(\Wpb)_{n \in\N}$
satisfies an LDP with (speed $n$ and) 
 a quasiconvex good rate function
\begin{equation}\label{rate_func}
\mathbb{I}_{p}(t) = \inf_{\tau_1\in\mathbb{R},\tau_2>0:\tau_1\tau_2^{-1/p}=t}\Psi^*_p(\tau_1,\tau_2),
\end{equation}
where recall that a  quasiconvex function is a function whose level sets are convex. 
Furthermore, it follows from \cite[Lemmas 3.1 and 3.4]{GanKimRam17}
  that $(\Wp)_{n \in \N}$ also satisfies an LDP with the same speed and rate function.
Note that the rate function $\mathbb{I}_{p}$ is insensitive to the projection directions,
in the sense that it is the same for $\sigma$-a.e.\ $\theta$.

We show in the  following lemma  that  the infimum in \eqref{rate_func} is attained uniquely at $(t,1)$,
yielding a simpler form for the rate function and use that to deduce
it is strictly convex and has a unique minimizer.
The latter  is a crucial property both for  obtaining
sharp large deviation estimates and developing importance sampling algorithms.  

\begin{lemma}\label{inf_ratefunc}
   For $p \in (1,\infty)$ and $a>0$ such that $\Psi^*_p(a,1)<\infty$, 
\[
 \inf_{\tau_1\in\mathbb{R},\tau_2>0:\tau_1\tau_2^{-1/p}=a}\Psi^*_p(\tau_1,\tau_2)  =\Psi^*_p(a,1) = \sup_{s_1,s_2\in\mathbb{R}}\left\{ as_1+s_2-\Psi_p(s_1,s_2)\right\}.
\]
\end{lemma}

The proof of Lemma \ref{inf_ratefunc} is relegated  to Appendix~\ref{apsec-infrfn}; when combined
with \cite[Theorem 2.5,  Lemma 3.1 and Lemma 3.4]{GanKimRam17}, it yields the following simpler
form of the quenched LDP.   

\begin{theorem}
   \label{LDP_p}
  Fix $p \in (1,\infty)$. For $\smeas$-a.e.\ $\theta$, under $\mathbb{P}_\theta$, the sequences $(\Wp)_{n \in \N}$ and $(\Wpb)_{n \in \N}$  both satisfy LDPs with the same strictly
  convex, symmetric, good rate function $\mathbb{I}_{p}$ given by 
\begin{equation}
  \label{rate-fn}
\mathbb{I}_{p}(a) :=\Psi^*_p(a,1) = \sup_{s_1,s_2\in\mathbb{R}}\left\{ as_1+s_2-\Psi_p(s_1,s_2)\right\}.
\end{equation}
\end{theorem}

We now introduce notation to state the sharp large deviation estimate for $\Wp$. Recall the definitions of $\Psi_p$, $\Psi^*_p$, $\mathbb{J}_p$ and $\lambda_x$ from Section \ref{subs-pre-notation} and for $x\in\mathbb{J}_p$, define
$\Hess_x = \Hess_{p,x}$ by 
 \begin{equation}\label{gamma_x}
  \Hess_{p,x} := (\hess \Psi_p)(\lambda_x),
 \end{equation}
 where we suppress the dependence on $p$ from  $\lambda_x$ and $\Hess_x$. 
 Also, fix $a>0$ such that $\mathbb{I}_p(a)<\infty$.
 With some abuse of notation, we write $\lambda_a =  \lambda_{a^*}$ and
$\Hess_a = \Hess_{a^*}$, 
 where $a^* = (a,1)$. 
 Note that then 
 $\lambda_a=(\lambda_{a,1},\lambda_{a,2})\in\mathbb{R}^2$
 is the unique maximizer in~\eqref{rate-fn}, that is,
\begin{equation}\label{lambdaa}
\Psi_p^*(a,1) = a\lambda_{a,1}+\lambda_{a,2}-\Psi_p(\lambda_{a,1},\lambda_{a,2}),
\end{equation}
and  
  \begin{equation}\label{Gammaa}
\Hess_{a} :=\left(\hess \Psi_p \right)(\lambda_a). 
\end{equation}
Next,  define  the positive constants $\den_{a} = \den_{p,a}$ and $\numer_a = \numer_{p,a}$
      via the relations  
      \begin{align}
        \label{def-sigmapa}
        \den^2_{a} &:= \langle \Hess_{a} \lambda_a, \lambda_a \rangle\det\Hess_a,  \\
        \label{def-kappapa}
        \numer_{a}^2 &:=  1 - \frac{(\lambda_{a,1}^2+\lambda_{a,2}^2)^{3/2}p(p-1)a}{\abs{\lambda_{a,2}^2(\Hess_{a}^{-1})_{11}-2\lambda_{a,1}\lambda_{a,2}(\Hess_{a}^{-1})_{12}+\lambda_{a,1}^2(\Hess_{a}^{-1})_{22}}(a^2+p^2)^{3/2}}.
      \end{align}  
      \begin{remark}
      Although it is not \emph{a priori} obvious that the right-hand side of~\eqref{def-kappapa} is positive, this will become apparent from the proof of Theorem \ref{main_lp}. 
      \end{remark}
      Finally,  also define the following functions: for $x \in \R$, 
 \begin{equation}
  \label{f_pa} 
\begin{array}{rcl}
\ell_a(x) & := & \Lambda_p(x\lambda_{a,1},\lambda_{a,2}), \\
\ell_{a,1}(x) & := & x \partial_{1}\Lambda_p(x\lambda_{a,1},\lambda_{a,2}), \\
\ell_{a,2}(x) & := & \partial_{2}\Lambda_p(x\lambda_{a,1},\lambda_{a,2}).
\end{array}
\end{equation}
Note that the dependence on $p$ of these functions is again not explicitly notated.

 We are now ready to state the quenched sharp large deviation estimate for scaled projections of  $\ell_p^n$ spheres. 
 Recall for $\theta\in\mathbb{S}$,  we denote $\mathbb{P}_\theta$ to be the probability measure $\PP$ conditioned on $\Theta = \theta$.
\begin{theorem}\label{main_lp}
Fix $p \in (1,\infty)$ and $a > 0$ such that $\mathbb{I}_p(a)<\infty$. Then
  the following statements hold with the matrix $\Hess_a$ as defined in \eqref{Gammaa} and constants  $\den_{a} = \den_{p,a}$ and $\numer_{a} = \numer_{p,a}$ defined as in
\eqref{def-sigmapa} and \eqref{def-kappapa}, respectively:
   \begin{enumerate}[label=(\roman*)]
   \item \label{main_lp_1}
For $n \in \N$, there exist mappings $R_a^n = R_{p,a}^n: \sphere^{n-1} \to \R$ and
       $c_a^n = c_{p,a}^n: \sphere^{n-1} \to \R^2$, defined explicitly in \eqref{def-cnhessn} and \eqref{Rn} as a centered integrated log moment generating function and its gradient, such that
  for $\smeas$-a.e.\ $\theta$,  
\begin{align}\label{tail_prob}
  \mathbb{P}_\theta\left(\Wp>a\right) =
\frac{C^n_a(\theta^n)}{ \numer_a \den_{a}\sqrt{2\pi n} }
e^{-n\mathbb{I}_{p}(a)  + \sqrt{n} R^n_{a}(\theta^n)}(1+o(1)),
\end{align} 
where 
  \begin{align}\label{Cna}
   C^n_a(\theta^n)   :=
 \exp\left( \norm{\Hess_{a}^{-1/2}c^n_{a}(\theta^n)}^2 \right).
 \end{align} 
\item \label{main_lp_2}
  Moreover, there exist sequences of random variables
 $(r_n=r^n_{p,a})_{n \in \N}$, $(s_n=s^n_{p,a})_{n \in \N}$, and
  $(t_{n,i}=t^n_{p,a,i})_{n \in \N}$, $i = 1, 2$, (defined on some common probability space) 
  such that for each $n \in \N$, 
 \begin{align}\label{equiv-Rc}
\left(R^n_{a}(\Theta^n),c^n_{a}(\Theta^n)\right)&\buildrel (d) \over = \left(r_n + \frac{1}{\sqrt{n}}s_n+o\left(\frac{1}{\sqrt{n}}\right),(t_{n,1}+o(1),t_{n,2}+o(1))\right),
  \end{align} 
and as $n \rightarrow \infty$, 
\begin{eqnarray*} 
  && \left(r_n,s_n,t_{n,1},t_{n,2}\right)\Rightarrow (\mathfrak{R},\mathfrak{S},\mathfrak{T}_1,\mathfrak{T}_2),
  \end{eqnarray*}
  where
  \begin{eqnarray*}
  && (\mathfrak{R},\mathfrak{S},\mathfrak{T}_1,\mathfrak{T}_2):=\\
  && \qquad \left(\widetilde{\mathfrak{A}}-\frac{1}{2} \mathbb{E}[\ell'_a(Z)Z]\widetilde{\mathfrak{D}},
  \frac{1}{8}\mathbb{E}[\ell''_a(Z)Z^2]\widetilde{\mathfrak{D}}^2,
  \widetilde{\mathfrak{E}}- \frac{1}{2}\mathbb{E}\left[\ell_{a,1}'(Z)Z \right]\widetilde{\mathfrak{D}},\widetilde{\mathfrak{G}}-\frac{1}{2}\mathbb{E}\left[\ell_{a,2}'(Z)Z \right]\widetilde{\mathfrak{D}}\right),
\end{eqnarray*}
 $Z$ is a standard Gaussian random variable, and $(\widetilde{\mathfrak{A}},\widetilde{\mathfrak{D}},\widetilde{\mathfrak{E}},\widetilde{\mathfrak{G}})$ are jointly Gaussian with mean $0$ and 
covariance matrix $\Sigma_{a} = \Sigma_{p,a}$ that takes the following explicit form: 
\begin{align}\label{cov_p}
\left(
\begin{array}{llll}
\mathrm{Cov}(\ell_a(Z),\ell_a(Z)) & \mathrm{Cov}(\ell_a(Z),Z^2) & \mathrm{Cov}(\ell_a(Z),\ell_{a,1}(Z)) & \mathrm{Cov}(\ell_a(Z),\ell_{a,2}(Z))\\
\mathrm{Cov}(Z^2,\ell_a(Z)) & \mathrm{Cov}(Z^2,Z^2) & \mathrm{Cov}(Z^2,\ell_{a,1}(Z)) & \mathrm{Cov}(Z^2,\ell_{a,2}(Z))\\
\mathrm{Cov}(\ell_{a,1}(Z),f_a(Z)) & \mathrm{Cov}(\ell_{a,1}(Z),Z^2) & \mathrm{Cov}(\ell_{a,1}(Z),\ell_{a,1}(Z)) & \mathrm{Cov}(\ell_{a,1}(Z),\ell_{a,2}(Z))\\
\mathrm{Cov}(\ell_{a,2}(Z),\ell_a(Z)) & \mathrm{Cov}(\ell_{a,2}(Z),Z^2) & \mathrm{Cov}(\ell_{a,2}(Z),\ell_{a,1}(Z)) & \mathrm{Cov}(\ell_{a,2}(Z),\ell_{a,2}(Z))\\
\end{array}
\right).
\end{align}
\end{enumerate} 
\end{theorem}

An outline of the proof of Theorem~\ref{main_lp} is given in Section~\ref{subs-reform},
with full details provided in Sections \ref{subs-mainproof} and \ref{subs-mainproof2}. See also \eqref{cn} and \eqref{Gammanx} for an interpretation of $c^n_a$ and $\Hess_a$ as the scaled mean vector and limiting covariance matrix, under a quenched tilted measure of a two-dimensional vector that arises in a convenient representation for $\Wp$  described in Section \ref{subs-reform}).

\begin{remark}\label{rem-prefactor}
  We will refer to the term $C^n_a(\theta^n) e^{\sqrt{n} R^n_{a}(\theta^n)}/\numer_a \den_{a}\sqrt{2\pi n}$
  in \eqref{tail_prob} as the ``prefactor" since
it provides a multiplicative correction to the exponentially decaying term
 $e^{-n\mathbb{I}_{p}(a) }$, which is identified by the LDP.   
 In addition, it follows from
\eqref{Cna}-\eqref{equiv-Rc} that (in distribution) $R^n_a(\Theta^n)$ and $C^n_a(\Theta^n)$
 both converge to zero as $n\to\infty$;  see also Lemma \ref{lemma-estimate} for more refined estimates. 
Further insight into the form of the prefactor can be found in
 Remarks \ref{rem-geom} and \ref{rem-mainres}.
\end{remark}

  As mentioned above, the most significant term in the prefactor 
  that depends on $\theta$ is $e^{\sqrt{n}R^n_{p,a}(\theta^n)}$. 
  The following proposition describes the additional geometric information contained in this term beyond what is available in the rate function $\mathbb{I}_p$, which is $\sigma$-almost surely insensitive to  the projection sequence $\Theta$.   
 \begin{proposition} \label{p-dependence}
 Fix $p \in (1,\infty)$, $a > 0$ such that $\mathbb{I}_p(a)<\infty$ and let $R^n_{p,a}$ be the mapping in Theorem \ref{main_lp} that is defined explicitly in \eqref{Rn}.  Then
\begin{enumerate}
\item For $p=2$, $R^n_{p,a}(\theta^n)$ is a constant regardless of the direction $\theta^n\in\mathbb{S}^{n-1}$;
\item For $p>2$, the maximum of $R^n_{p,a}(\theta^n)$ over $\theta^n\in\mathbb{S}^{n-1}$ is attained at $(\pm 1,\pm1,\ldots,\pm1)/\sqrt{n}$, while the minimum is attained at $\pm e_j$ for $j=1,\ldots,n$;
\item For $p<2$, the minimum of $R^n_{p,a}(\theta^n)$ over $\theta^n\in\mathbb{S}^{n-1}$ is attained at $(\pm 1,\pm1,\ldots,\pm1)\sqrt{n}$, while the maximum is attained at $\pm e_j$ for $j=1,\ldots,n$,
\end{enumerate}
where $(e_j)_{j=1,\ldots,n}$ are defined to be the standard basis vectors in $\mathbb{R}^n$.
\end{proposition}

 \begin{remark}
     \label{rem-geom}    
     Proposition~\ref{p-dependence} in conjunction with Theorem \ref{main_lp} shows how the sharp large deviation estimates reflect the difference in the geometry of $\ell^n_p$ spheres for
     $p \in (1, 2)$ and $p\in(2,\infty)$ with respect to the relative distribution of mass along different
       rays. This motivates 
 obtaining  sharp large deviation estimates
for projections of more general high-dimensional objects 
to uncover new geometric information about these objects.  
   \end{remark}

  As a corollary,
  combining the two parts of  Theorem~\ref{main_lp}, we obtain an alternative
  expression for the distribution of the conditioned tail probability:
  \begin{corollary}\label{alt-form}
  Fix $p \in (1,\infty)$ and $a > 0$ such that $\mathbb{I}_p(a)<\infty$. For $n\in\mathbb{N}$, recall the definitions of $(r_n)_{n\in\mathbb{N}}$, $(s_n)_{n\in\mathbb{N}}$ and $(t_n)_{n\in\mathbb{N}}$ in Theorem~\ref{main_lp}~\ref{main_lp_2}, and that of ${\mathcal H}_a$ from \eqref{Gammaa}.  Then
  \[
  \mathbb{P}_\Theta\left(\Wp>a\right):=   \mathbb{P}\left(\Wp>a\middle | \Theta^n\right)\buildrel (d) \over=
\frac{M_n}{ \numer_a \den_{a}\sqrt{2\pi n} }
e^{-n\mathbb{I}_{p}(a)  + \sqrt{n} r_n}(1+o(1)),
  \]
  where
  \begin{equation}\label{Mn}
  M_n:=\exp\left( s_n+\norm{\Hess_{a}^{-1/2}t_n}^2  \right).
    \end{equation}
  Moreover, as $n\to\infty$,
\begin{align}\label{joint-conv}
  \left(M_n,r_n\right) \Rightarrow 
  \left(\exp\left(\mathfrak{S}+
  \norm{\Hess_{a}^{-1/2}\mathfrak{T}}^2\right), \mathfrak{R} \right),
  \end{align}
  where 
$(\mathfrak{R},\mathfrak{S},\mathfrak{T}_1,\mathfrak{T}_2)$ is as defined in Theorem~\ref{main_lp}\ref{main_lp_2}.
  \end{corollary}

  \begin{proof}
  By~\eqref{tail_prob},~\eqref{Cna} and~\eqref{equiv-Rc}, the tail probability can be written as
  \begin{align*}
  \mathbb{P}_\Theta\left(\Wp>a\right) &\buildrel (d) \over= \frac{e^{\norm{\Hess_{a}^{-1/2}t_n}^2+o(1)}}{ \numer_a \den_{a}\sqrt{2\pi n} }
e^{-n\mathbb{I}_{p}(a)  + \sqrt{n} r_n+s_n+o(1)}(1+o(1))\\
& = \frac{M_n}{ \numer_a \den_{a}\sqrt{2\pi n} }
e^{-n\mathbb{I}_{p}(a)  + \sqrt{n} r_n}(1+o(1)),
  \end{align*}
  since $\exp(o(1)) = 1+o(1)$.
  Also,  from the relation~\eqref{Mn}, the mapping $(r_n,s_n,t_{n,1},t_{n,2})\mapsto (M_n,r_n)$ is continuous. Therefore,
  we may apply the continuous mapping theorem to the last display, and invoke
   Theorem~\ref{main_lp}\ref{main_lp_2} to obtain the joint convergence stated in~\eqref{joint-conv}.
  \end{proof}

  \begin{remark}
     In Theorem \ref{main_lp} and Corollary \ref{alt-form} 
      we only consider values $p \in (1,\infty)$ because
      for $p \in (0,1)$, $\ell_p^n$ balls are no longer convex and 
      the existence of even an   LDP has not been established.
      Moreover, as shown in \cite[Theorem 2.6]{GanKimRam17}, when $p=1$ the quenched LDP of the projection exists only when the projection directions satisfy $\lim_{n\to\infty}\sqrt{\frac{n}{\log n}} \max_{1 \leq i \leq n}\theta_i^{(n)}=c$ for some constant $c\in(0,\infty)$, and in that case, it is with speed $n/\sqrt{\log n}$ and the rate function is no longer universal but depends on the limiting constant $c$.
      On the other hand, we omit 
  the  case $p = \infty$, or the more general case of product measures, 
  because this is in fact simpler to analyze than the $p \in (1,\infty)$ case;
   details can be found in \cite[Section 4.2]{LiaoThesis22}.   
    \end{remark}

  \subsection{Results on projections of $\ell_p^n$ balls}
\label{subs-analres-b}
  Next, we state the corresponding sharp large deviation results for balls.
For $p \in (1,\infty)$ and $a  > 0$, 
recalling that $\lambda_{a,1}$ is the first coordinate of the maximizer $\lambda_a$ in the expression for
$\Psi_p^*(a)$  in \eqref{lambda_x} and $\Hess_a$ is as defined in \eqref{Gammaa}, 
define the positive constant $\gamma_a=\gamma_{p,a}$ via the relation
\begin{equation}\label{def-gammapa}
\gamma_a^2 :=\lambda_{a,1}^2(1+a\lambda_{a,1})^2(\det\Hess_a)^2\abs{-\frac{a(p-1)}{p^2}\lambda_{a,1}+\frac{2a}{p}(\Hess_a)^{-1}_{12}+(\Hess_a)^{-1}_{22}+\frac{a^2}{p^2}(\Hess_a)^{-1}_{11}}.
\end{equation}
\begin{theorem}\label{main_lp-b}
Fix $p \in (1,\infty)$ and $a > 0$ such that $\mathbb{I}_p(a)<\infty$. Then
for $n \in \N$, 
\begin{align}\label{tail_prob-b}
  \mathbb{P}_\theta\left(\Wpb>a\right) =
\frac{C^n_a(\theta^n)}{ \gamma_a\sqrt{2\pi n} }
e^{-n\mathbb{I}_{p}(a)  + \sqrt{n} R^n_{a}(\theta^n)}(1+o(1)),
\end{align} 
where 
$\gamma_{a} = \gamma_{p,a}$ is the constant defined in~\eqref{def-gammapa}, and $R^n_a$ and $C^n_a$ are the functions defined in Theorem~\ref{main_lp}.
\end{theorem}

\begin{remark}
  \label{rem-ballsphere}
\begin{enumerate}[label=(\roman*)]
\item Note that the tail probability in~\eqref{tail_prob-b} is a geometric quantity,  equal to 
   the volume of  the $p$-spherical cap (at level $a$) of $\ell^n_p$ balls along the direction $\theta^n$.
\item Recall that it follows from the results of \cite{GanKimRam17} (recapitulated here as Theorem \ref{LDP_p})
  that $\ell_p^n$ spheres and balls cannot be distinguished because the 
  large deviation speeds and rate functions for 
  random projections of $\ell^n_p$ balls and spheres coincide.
  In contrast, we see from~\eqref{tail_prob} and~\eqref{tail_prob-b} that although the two 
  prefactors have a similar form, their actual values differ since in general
  $\gamma_a \neq \kappa_a \xi_a$. 
Thus, 
the sharp large deviation estimates obtained here are sufficiently refined to distinguish these two objects, whereas the LDP rate function does not do so.
\item
  As in Remark \ref{rem-geom}, due to the appearance of $R_a^n$ in \eqref{tail_prob-b}, the sharp large deviation estimate provides more insight into the distinction between  the geometry of $\ell_p^n$ balls
  with  $p \in (1,2)$ and $\ell_p^n$ balls with $p \in (2,\infty)$.   
\end{enumerate} 
\end{remark}

Similar to Corollary~\ref{alt-form}, we have the following immediate corollary for balls:
\begin{corollary}
Fix $p \in (1,\infty)$ and $a > 0$ such that $\mathbb{I}_p(a)<\infty$. For $n\in\mathbb{N}$, recall the definitions of $(M_n)_{n\in\mathbb{N}}$ and $(r_n)_{n\in\mathbb{N}}$ in Corollary~\ref{alt-form} and let $\gamma_a$ be as in \eqref{def-gammapa}. Then for $n\in\N$,
  \[
  \mathbb{P}_\Theta\left(\Wpb>a\right) \buildrel (d) \over=
\frac{M_n}{ \gamma_{a}\sqrt{2\pi n} }
e^{-n\mathbb{I}_{p}(a)  + \sqrt{n} r_n}(1+o(1)),
  \]
  where~\eqref{Mn} and~\eqref{joint-conv} hold.
\end{corollary}

 \subsection{Reformulation of the problem and  outline of the proof } 
   \label{subs-reform}

Fix $p\in (1,\infty)$. 
As mentioned in the introduction, one of the reasons the estimate
\eqref{tail_prob} is challenging to establish is that
$\Wp$ and $\Wpb$ are randomly weighted sums of random variables that are not independent, and furthermore,
the random weights are also themselves not independent.
In this section 
 we provide a brief outline of our proof and additional insight into the form of the sharp large deviation
 estimates, contrasting them with existing results,
 and explaining  the role  of various constants.

 The first step of the proof is to  reformulate the probability of the rare event in terms of a certain
multi-dimensional random vector
($\Snp$ in the case of spheres and $\Snpb$ in the case of balls) 
using  a well known probabilistic representation
 for the random vector $X^{(n,p)}$  that we now recall. 
  Assume without loss of generality that the probability  space
  $(\Omega, {\mathcal F}, \PP)$ is large enough to also support 
an  i.i.d.\ sequence of generalized $p$-Gaussian random variables 
$(Y_i^{(p)})_{i \in \N}$,
independent of $\Theta$, and define the 
  $n$-dimensional random vector $Y^{(n,p)} := (Y_1^{(p)}, \ldots, Y_n^{(p)})$, 
  where each $Y_j^{(p)}$ has  density $\dens_p$ defined in~\eqref{pNormal}. 
  Then, it  follows from~\cite[Lemma 1]{Schechtman90} (see also a statement of this property at the bottom of p.\ 548 in \cite{Bor90})  that 
 \begin{equation}\label{equv_relation}
X^{(n,p)} \buildrel (d) \over =\frac{Y^{(n,p)}}{\norm{Y^{(n,p)}}_{n,p}}, \qquad n \in\N, 
\end{equation}
where recall that   $\|x\|_{n,p}$ denotes the $p$-norm in $\R^n$. 
Define the $\mathbb{R}^2$-valued random vector 
\begin{equation} \label{joint_pdf}
\Snp:=\frac{1}{n}\sum_{j=1}^n\left(\sqrt{n} \Theta^n_jY^{(p)}_j,\abs{Y^{(p)}_j}^p \right).
\end{equation}

In view of~\eqref{W_p} and the independence of $X^{(n,p)}$,~\eqref{equv_relation}, and $\Theta$, for $a >0$ and $\theta \in \SS$,  we may rewrite
 the tail probability on the left-hand side of \eqref{tail_prob} as
\begin{align}
 \notag
\mathbb{P}_\theta\left(\Wp>a\right) &=
\mathbb{P}\left(\frac{n^{1/p}}{n}\sum_{j=1}^n\frac{\sqrt{n}\theta^n_jY_j^{(p)}}{\norm{Y^{(n,p)}}_{n,p}}>a\right) \\
 \notag
&= \mathbb{P}\left(\frac{1}{n}\sum_{j=1}^n\sqrt{n}\theta^n_jY^{(p)}_j>a\left(\frac{1}{n}\sum_{j=1}^n\abs{Y^{(p)}_j}^p \right)^{1/p}\right)\\
\label{eq-reform}
& = \mathbb{P}_\theta \left(\Snp\in \bDpa\right),  
\end{align}
where $\bDpa$ is the two-dimensional domain defined by 
\begin{equation}
  \label{def-doma}
  \bDpa :=\left\{(x_1,x_2)\in\mathbb{R}^2: x_2>0,x_1>ax_2^{1/p} \right\}. 
\end{equation}

On the other hand, again from~\cite[Lemma 1]{Schechtman90}, we also have an equivalent representation for $\mathscr{X}^{(n,p)}$:
\begin{equation}\label{equv_relation-b}
\mathscr{X}^{(n,p)} \buildrel (d) \over =\mathscr{U}^{1/n}\frac{Y^{(n,p)}}{\norm{Y^{(n,p)}}_{n,p}}, \qquad n \in\N, 
\end{equation}
where $\mathscr{U}$ is a uniform random variable on $(0,1)$, independent of the sequence
  $(Y^{(n,p)})_{n \in \N}$.
Define the $\mathbb{R}^3$-valued random vector
\begin{equation} \label{joint_pdf-b}
  \Snpb:=\left(\frac{1}{n}\sum_{j=1}^n\sqrt{n} \Theta^n_jY^{(p)}_j,\frac{1}{n}\sum_{j=1}^n\abs{Y^{(p)}_j}^p,\mathscr{U}^{1/n} \right)
   =  \left( \bar{S}^{(n,p)}, \mathscr{U}^{1/n} \right).
\end{equation}
From the equivalent representation~\eqref{equv_relation-b}, for $a>0$ and $\theta\in\SS$, we may rewrite the tail probability of $\Wpb$ as
\begin{align}
 \notag
\mathbb{P}_\theta\left(\Wpb>a\right) &=
\mathbb{P}\left(\frac{n^{1/p}}{n}\sum_{j=1}^n\mathscr{U}^{1/n}\frac{\sqrt{n}\theta^n_jY_j^{(p)}}{\norm{Y^{(n,p)}}_{n,p}}>a\right) \\
 \notag
&= \mathbb{P}\left(\mathscr{U}^{1/n}\frac{1}{n}\sum_{j=1}^n\sqrt{n}\theta^n_jY^{(p)}_j>a\left(\frac{1}{n}\sum_{j=1}^n\abs{Y^{(p)}_j}^p \right)^{1/p}\right)\\
\label{eq-reform-b}
& = \mathbb{P}_\theta \left(\Snpb\in \bDpab\right),  
\end{align}
where  $\bDpab$ is the three-dimensional domain given  by 
\begin{equation}
  \label{def-doma-b}
  \bDpab :=\left\{(x_1,x_2,y)\in\mathbb{R}^3: 1 >  y >  0, x_2>0,x_1y>ax_2^{1/p} \right\}.  
\end{equation}
\begin{remark}
  Throughout the paper, we will typically use an overline to denote
  quantities related to  these multi-dimensional reformulations, and script fonts for quantities related to $\ell^n_p$ balls. 
\end{remark}

 While several results on sharp large deviations in multiple dimensions
 have been obtained (see, e.g., \cite{Andriani97,Joutard17} as well as \cite{Barbe05} for a comprehensive
 list of references), none of these cover the cases of
 interest in \eqref{eq-reform} and \eqref{eq-reform-b}. 
In particular, 
  the work  \cite{Andriani97} considers  empirical means of 
  i.i.d.\ random vectors whereas, under $\PP_{\theta}$,
     $\bar{S}^{(n,p)}$ is the empirical mean  of non-identical random vectors. 
  Moreover, the results of  \cite{Joutard17} also
  do not apply since the condition imposed in \cite[Assumption (A.2)]{Joutard17}
  is not satisfied here due to the additional $\sqrt{n}$ factor in the exponent of~\eqref{tail_prob} compared with~\cite[Equation (3)]{Joutard17}.
  Instead, our proof proceeds by first exploiting quantitative asymptotic independence results of
    the weights $(\Theta^n_i)_{j=1,\ldots,n}$    obtained
    in Section \ref{subs-prep},  and combining them with new 
     asymptotic estimates for certain Laplace-type integrals stated in Section \ref{sec-pfmain}.

\begin{remark}
  \label{rem-mainres} 
Comparing the estimate in \eqref{tail_prob} 
with the sharp large deviation estimate 
for the projection of an i.i.d.\ sum 
 onto the vector $\mathfrak{I}^n = (1,1,\ldots,1)/\sqrt{n}$
  given in ~\eqref{BRR},  we 
     see that $\den_a$ in \eqref{tail_prob} plays a role similar to $\bar{\sigma}_a \tau_a$ in \eqref{BRR}.
     On the other hand, the additional constant $\kappa_a$ in  \eqref{tail_prob} arises due to the geometry of the
     domain $\bDpa$ defined in~\eqref{def-doma} and the fact that we obtain this estimate by reformulating it in terms of a two-dimensional  problem.   
From a technical point of view, the additional $\theta^n$-dependent
terms $R^n_a(\theta^n)$ and $C^n_a(\theta^n)$ arise because we are
 considering (quenched) sharp large deviations of a vector $\Snp$ whose independent summands 
 are not identically distributed under $\PP_\theta$ on account of the different weights 
 arising from the coordinates of  $\theta^n$.
 From their  exact definitions  given in~\eqref{Rn} and~\eqref{def-cnhessn},
 it is easy to see that both these terms would vanish if we considered $\theta^n \in \SS^{n-1}$ with identical weights such as  
 $\theta^n = \mathfrak{I}^n = (1,1,\ldots,1)/\sqrt{n}$.
  \end{remark}
  
\section{An Importance Sampling Algorithm}\label{IS-algorithm}
To numerically compute  the tail probability
$\PP_\theta (\Wp > a) = \mathbb{E}_\theta[1_{\left\{\Wp >a\right\}}]$
using standard Markov Chain Monte Carlo (MCMC), 
 for any $\theta^n \in \mathbb{S}^{n-1}$, one would have to generate independent 
samples of $X^{(n,p)}$ from the cone measure $\mu_{n,p}$ defined in~\eqref{cone-meas},
and use the  empirical mean as an estimate of  the expectation. 
However, since the probability is very small, this is inefficient or computationally
infeasible for even moderate values of $n$.
In this section, we propose an alternative {\it importance sampling (IS) 
algorithm} to more efficiently
  compute the tail probability numerically, for a range of values of $n$,
  and compare this with the analytical estimate obtained in Theorem \ref{main_lp}. 
For $a>0$, fix $p \in (1,\infty)$ and recall the constant $\lambda_a$ defined in~\eqref{lambdaa}. 
 Also,  recall the definition of the density $\dens_p$ in \eqref{pNormal}. 
Given $n \in \N$,
let $\widetilde{Y}^{(n,p)}:=(\widetilde{Y}^{(n,p)}_1,\ldots,\widetilde{Y}^{(n,p)}_n)$, where $\widetilde{Y}^{(n,p)}_j$, $j= 1,\ldots, n,$ are
 random variables defined on $(\Omega, {\mathcal F}, \PP)$ that are independent under $\mathbb{P}_\theta$ for each $\theta\in\SS$, and  such that
   $\widetilde{Y}^{(n,p)}_j$ has density 
  \begin{equation}\label{tilt_pdf}
    \widetilde{f}_{p,j}^{n}(y) :=  \exp\left(\left\langle \lambda_a,\left(\sqrt{n}\theta^n_jy,\abs{y}^p\right)\right\rangle-\Lambda_p\left(\sqrt{n}\theta^n_i\lambda_{a,1},\lambda_{a,2}\right)\right) f_p(y), 
    \quad y \in \R,
  \end{equation} 
    where we suppress from the notation the explicit dependence of $\widetilde{f}_{p,j}^{n}$ on $\theta^n$. Also define
\begin{equation}\label{Wtilde}
\widetilde{W}^{(n,p)} :=\frac{n^{1/p}}{n^{1/2}}\sum_{j=1}^n
\frac{\widetilde{Y}^{(n,p)}_j \Theta^n_j}{\norm{\widetilde{Y}^{(n,p)}}_{n,p}}. 
\end{equation}
In view of \eqref{tilt_pdf} and \eqref{Wtilde},
    it then follows that 
\begin{equation}
  \label{imp-samp}
\mathbb{P}_\theta(\Wp>a) = 
\mathbb{E}_\theta\left[1_{\left\{\widetilde{W}^{(n,p)}>a\right\}}\prod_{j=1}^n\exp\left(-\left\langle \lambda_a,\left(\sqrt{n}\theta^n_j\widetilde{Y}^{(n,p)}_j,\abs{\widetilde{Y}^{(n,p)}_j}^p\right)\right\rangle+\Lambda_p\left(\sqrt{n}\theta^n_i\lambda_{a,1},\lambda_{a,2}\right)\right)\right ]. 
\end{equation}
The IS algorithm
estimates the tail probability on the left-hand side of \eqref{imp-samp}
by first sampling a direction $\theta^n$ according to $\sigma_n$ and then sampling from  i.i.d.\  copies of the vector $\widetilde{Y}^{(n,p)}:=(\widetilde{Y}^{p}_1,\ldots,\widetilde{Y}^{p}_n)$,  independently
of the $\theta^n$ sample, 
to approximate the expectation on the right-hand side of \eqref{imp-samp} by a standard Monte Carlo
  estimate.

\begin{figure}[h!]
  \centering
  \begin{subfigure}[b]{0.4\linewidth}
    \includegraphics[width=\linewidth]{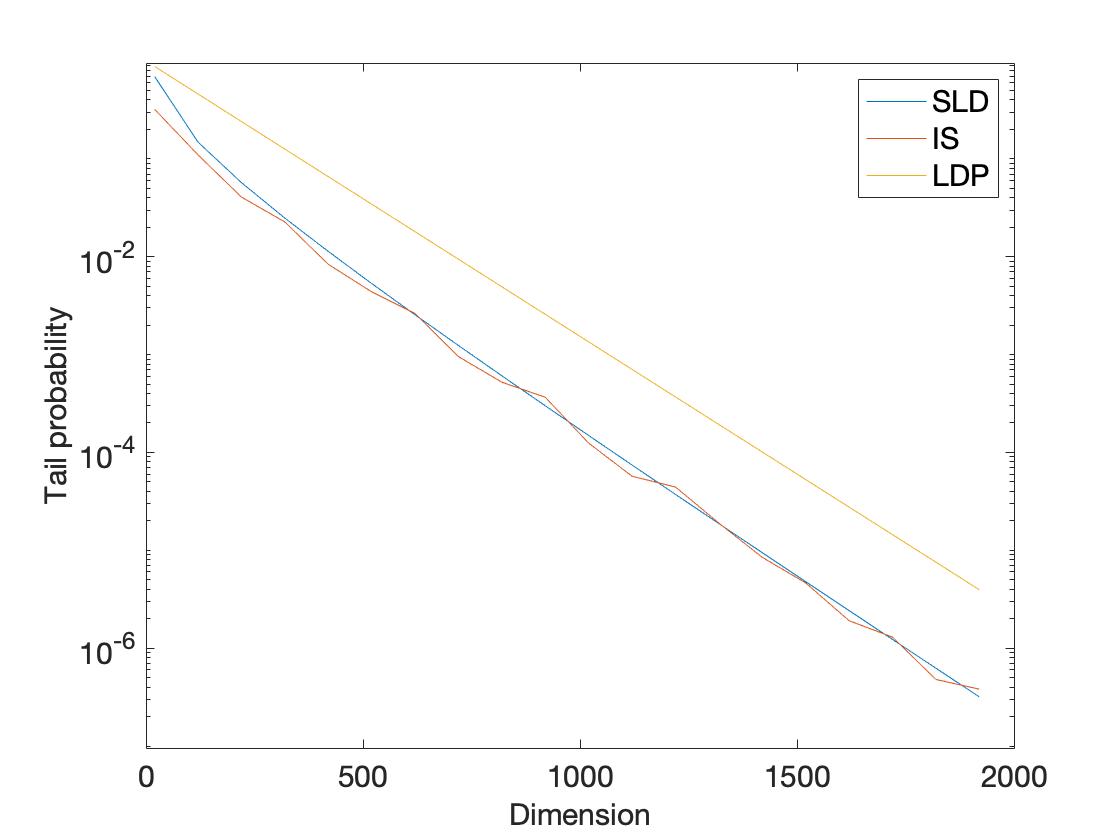}
    \caption{$a=0.1$.}
  \end{subfigure}
  \begin{subfigure}[b]{0.4\linewidth}
    \includegraphics[width=\linewidth]{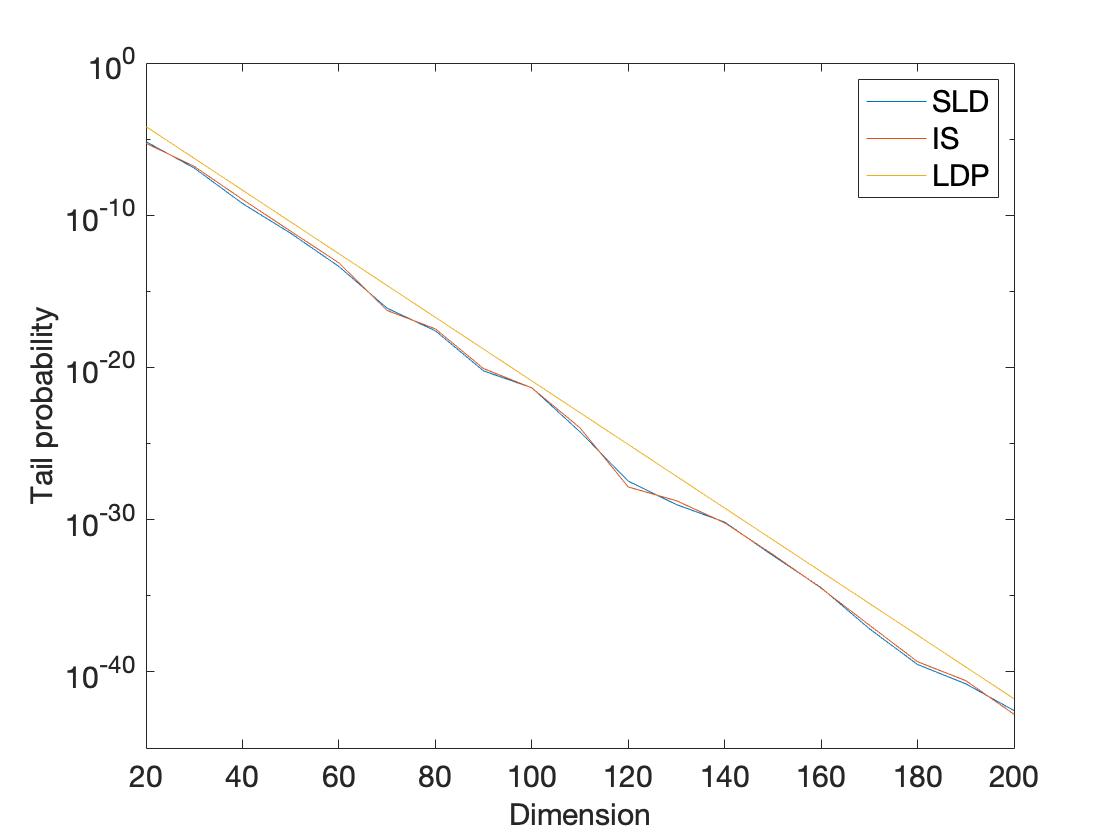}
    \caption{$a=0.7$.}
  \end{subfigure}
  \caption{Log scale plot of estimates of $P_\theta(W^{(n,3)} > a)$ vs.\ dimension.}
    \label{p3}
\end{figure}

The results are displayed in
Figures $1$--$2$ and Tables $1$--$2$. 
In each case, the IS estimate is computed as above, the
  LDP estimate is $e^{-n \mathbb{I}_p(a)}$  (i.e., with $1$ as a prefactor),
  and the sharp large deviation (SLD) estimate is  the prefactor (see Remark \ref{rem-prefactor}) times $e^{-n\mathbb{I}_p(a)}$. 
We consider  $p=3$ with only $100$ samples since we do not have closed form expressions for various functions needed in the IS simulation, thus requiring greater computational effort per sample. In Table~\ref{table-p3a07} we also calculate
the confidence interval of the
IS estimate and tabulate the relative distance between the SLD and IS estimates, computed as
$\left(\text{SLD}-\text{IS}\right)\times100/\text{IS}$.
First, we see from Figure~\ref{p3} that the LDP estimate is not a good enough approximation,
but the sharp large deviation (SLD) estimate  does a much better job.
For large $a$, namely $a = 0.7$,  in Figure~\ref{p3}(B)
and Table~\ref{table-p3a07}, we
see that the  SLD and IS estimates  match pretty well even for
small $n$ (namely, even $n = 20$).
However, this is not the case for $a$ small,
namely for $a = 0.1$.  In this case, as evident from Figure~\ref{p3}(A) and  Table~\ref{table-p3a01},  the
SLD estimate appears to achieve the same accuracy only for much larger $n$, 
which likely reflects  the dependence of the $o(1)$ term in \eqref{tail_prob} on $a$.

 \begin{table}[h!]
\centering
\begin{tabular}{ |c|c|c|c|c|}  
\hline
$n$ & SLD & IS & Relative distance  & Confidence Interval\\
\hline 
$20$		& $6.8707\times10^{-6}$	& $5.3317\times10^{-6}$	&$27.18\%$& $[3.4203\times10^{-6},7.2430\times10^{-6}]$\\
$80$		& $2.5403\times10^{-18}$	& $3.4245\times10^{-18}$	&$-25.82\%$& $[1.5542\times10^{-18},5.2948\times10^{-18}]$\\
$140$	& $6.9378\times10^{-31}$	& $6.1856\times10^{-31}$	&$12.16\%$& $[2.5305\times10^{-31},9.8407\times10^{-31}]$\\
$200$	& $2.8813\times10^{-43}$	& $1.6547\times10^{-43}$	&$74.13\%$& $[4.0920\times10^{-44},2.9002\times10^{-43}]$\\
\hline
\end{tabular}
\caption{Estimates of $P_\theta(W^{(n,p)} > 0.7)$ for $p = 3$. The sample size for IS is  $100$.}
\label{table-p3a07}
\end{table} 

 \begin{table}[h!]
\centering
\begin{tabular}{ |c|c|c|c|c|}  
\hline
$n$ & SLD & IS & Relative distance& Confidence Interval\\
\hline 
$20$		& $6.9193\times10^{-1}$	& $3.2004\times10^{-1}$	&$116.20\%$&  $[2.5636\times10^{-1},3.8372\times10^{-1}]$\\
$420$		& $1.1317\times10^{-2}$	& $8.3597\times10^{-3}$	&$35.38\%$& $[5.6085\times10^{-3},1.1110\times10^{-2}]$\\
$820$	& $6.0651\times10^{-4}$	& $5.2198\times10^{-4}$	&$16.19\%$& $[3.2412\times10^{-4},7.1985\times10^{-4}]$\\
$1220$	& $3.7235\times10^{-5}$	& $4.4306\times10^{-5}$	&$15.96\%$& $[2.8312\times10^{-5},6.0299\times10^{-5}]$\\
\hline
\end{tabular}
\caption{Estimates of $P_\theta(W^{(n,p)} > 0.1)$ for $p = 3$. The sample size for IS is  $100$.}
\label{table-p3a01}
\end{table}

Finally, we also ran simulations for different realizations $\theta$ of the direction sequence $\Theta$.
 We see from Figure~\ref{compare} that different projection direction sequences result
 in  fluctuations around the
quantity $e^{-n\mathbb{I}_p(a)}/\kappa_a \xi_a \sqrt{2 \pi n} $, which is the 
 basic sharp large deviation
 estimate obtained by  ignoring the $\theta^n$-dependent
   terms in the prefactor in   \eqref{tail_prob}.  As  shown in Theorem~\ref{main_lp}\ref{main_lp_2}, these fluctuations converge in distribution to
functionals  of a multi-dimensional Gaussian vector with an explicit covariance matrix.

\begin{figure}[h!]
\includegraphics[scale=0.2]{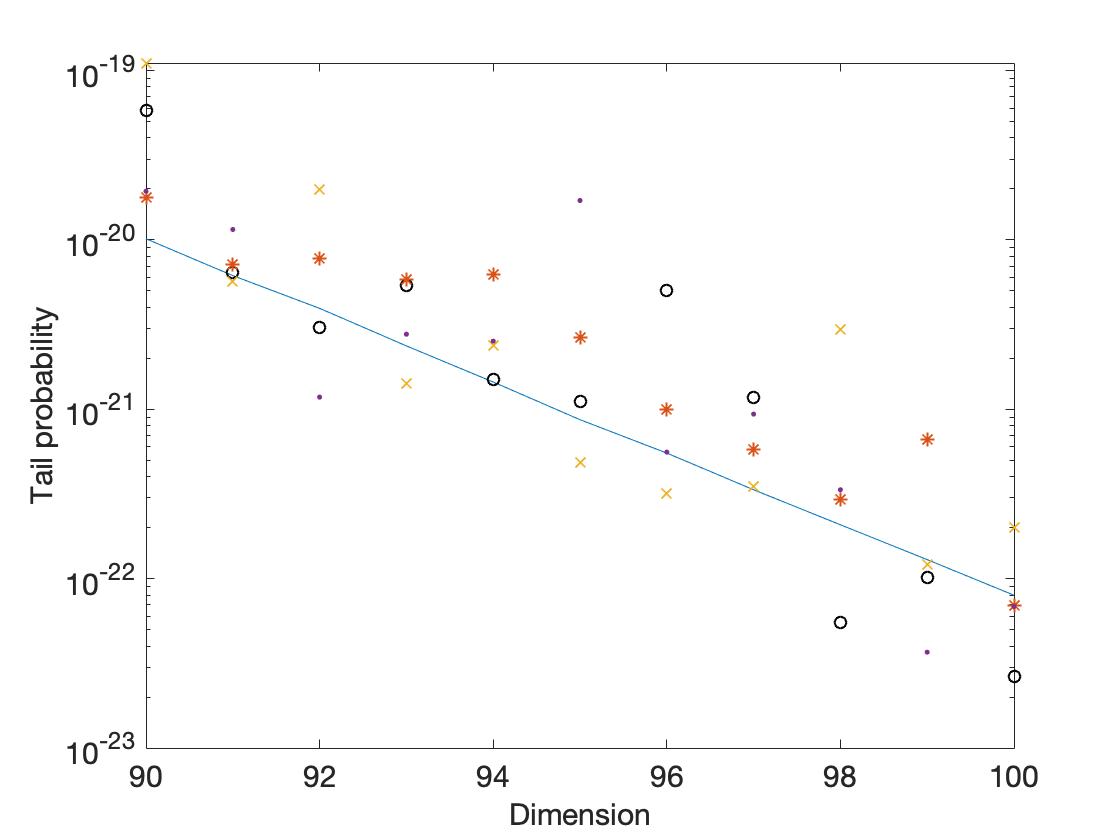}
\caption{Log scale plot of estimates of $P_\theta(W^{(n,p)} > 0.7)$ vs.\ dimension for $p = 3$. Solid line is $e^{-n\mathbb{I}_p(a)}/\kappa_a \xi_a \sqrt{2 \pi n} $. Scatter points are SLD estimates for different direction sequences.}
\label{compare}
\end{figure}

\section{Asymptotic Independence Results for the Weights}
\label{subs-prep}

Recall that $\mathcal{P}(\mathbb{R})$ is the set of probability measures on $\mathbb{R}$.
For $p\in[1,\infty)$, denote
\[\mathcal{P}_p(\mathbb{R}):=\left\{\nu\in\mathcal{P}(\mathbb{R}): \int_\mathbb{R} \abs{u}^p\nu(du)<\infty \right\},\]
and equip $\mathcal{P}_p(\mathbb{R})$ with the $p$-Wasserstein distance defined to be
\begin{equation}\label{wass}
  \mathcal{W}_p(\nu,\nu') := \inf_{\pi\in\Pi(\nu,\nu')}\int_{\mathbb{R}^2}\abs{x-y}^p\pi(dx,dy), \quad \nu,\nu'\in\mathcal{P}_p(\mathbb{R}),
\end{equation}
where $\Pi(\nu,\nu')$ denotes the set of couplings of $\nu$ and $\nu'$ or equivalently,
the set of probability measures on $\mathbb{R}^2$ whose first and second marginals coincide with $\nu$ and $\nu'$, respectively.

We now define a function with polynomial growth in the natural way. 

\begin{definition}
  \label{def-polygrowth}
Given $m \in \N$, we say that a function $f:\mathbb{R}\to\mathbb{R}$ has polynomial growth of degree $m$ if there exist $T\in\mathbb{R}$ and $C\in(0,\infty)$ such that 
\[
\abs{f(t)} \leq C\left(\abs{t}^m+1\right),\quad \text{for}\quad \abs{t} > T.
\]
We say a function $f:\mathbb{R} \to \mathbb{R}$ has polynomial growth if
it has polynomial growth of degree $m$ for some $m \in \N$. 
\end{definition}

Next, we recall the definition of the $p$-Wasserstein distance on probability measures.
\begin{lemma} [Definition 6.8 and Theorem 6.9 of~\cite{Villani08}]\label{wass-conv}
Let $(\nu^n)_{n\in\mathbb{N}}\subset\mathcal{P}_p(\mathbb{R})$ and $\nu\in\mathcal{P}_p(\mathbb{R})$. Then the following two statements are equivalent:
\begin{enumerate}
\item  $\mathcal{W}_p(\nu^n,\nu)\to 0$.
\item For any continuous $\phi:\mathbb{R}\to\mathbb{R}$ that has polynomial growth of degree $p$
\[
\int_\mathbb{R}\phi(x)\nu^n(dx)\to\int_\mathbb{R}\phi(x)\nu(dx). 
\]
\end{enumerate}
\end{lemma}

For each $n \in \N$ and $\theta \in \SS$, let $L_\theta^n$
denote the empirical measure of the coordinates of the scaled projection direction $\sqrt{n} \theta^n$:
\begin{align}\label{emp_L}
\empn := \frac{1}{n}\sum_{i=1}^n\delta_{\sqrt{n}\theta^n_i}.
\end{align}

The following strong law of large numbers for $(\empn)_{n \in \N}$ was established
in~\cite[Lemma 5.11]{GanKimRam17}.   Recall that $\normal$ denotes the standard normal distribution. 
\begin{lemma}[Lemma 5.11 of \cite{GanKimRam17}]
  \label{slln}
For $p\in (1,\infty)$, for $\smeas$-a.e.\ $\theta \in \SS$,
\[
\mathcal{W}_p\left(\empn,\normal\right)\to 0, \quad \text{ as } n \rightarrow \infty. 
\]
\end{lemma}

We now establish a  central limit theorem refinement of Lemma \ref{slln}.
Given 
an i.i.d.\ array
$(Z^n = (Z_j^n, j = 1 \ldots, n))_{n \in \N}$ of standard normal random variables, for any twice continuously differentiable function $\phi$, define 
\begin{equation}
  \label{def-sn}
  \hat{s}_n(\phi)  :=\sum_{j=1}^n\frac{\phi''(Z^n_j)}{2}\left( \frac{\sqrt{n}Z^n_j}{\norm{Z^n}_{n,2}}-Z^n_j\right)^2,
  \end{equation}
and set 
\begin{align}
  \label{def-rn}
  \hat{r}_n (\phi) & := \frac{1}{\sqrt{n}}\sum_{j=1}^n \left[\phi (Z^n_j) 
- \int_{\R} \phi(x) \normal(dx)  
+ \phi'(Z^n_j)\left( \frac{\sqrt{n}Z^n_j}{\norm{Z^n}_{n,2}}-Z^n_j\right)\right]. 
  \end{align}
  For any probability measure $\pi\in\mathcal{P}(\R)$, define $\pi(F):=\int_{\R}F(x)\pi(dx)$, for any Borel measurable function $F:\R\to\R$.

\begin{lemma} \label{clt_expansion}
  Given a thrice continuously differentiable function $F:\R \to \R$ and two twice continuously differentiable functions $G_1,G_2:\R\to\R$   such that  $F'''$, $G_1''$ and $G_2''$ have polynomial growth in the sense of Definition \ref{def-polygrowth}, we have the following expansion:
\begin{align*}
&  \sqrt{n} \left( L^n_\Theta (F) -  \normal (F),
L^n_\Theta (G_1) -\normal (G_1),L^n_\Theta (G_2) - \normal (G_2) \right) \\
&\qquad \buildrel (d) \over = \left(\hat{r}_n(F)+\frac{1}{\sqrt{n}}\hat{s}_n(F)+o\left(\frac{1}{\sqrt{n}}\right),  \hat{r}_n(G_1) +o(1),\hat{r}_n(G_2) +o(1)\right), 
\end{align*}
where $\hat{s}_n$ and $\hat{r}_n$ are as defined in \eqref{def-sn} and \eqref{def-rn},
and  as $n \rightarrow \infty$, 
\begin{align*}
&(\hat{r}_n(F),\hat{s}_n(F),\hat{r}_n(G_1),\hat{r}_n(G_2))\\ 
&\qquad\Rightarrow \left(\widetilde{\mathfrak{A}}-\frac{1}{2}\mathbb{E}[F'(Z)Z]\widetilde{\mathfrak{D}},\frac{1}{8}\mathbb{E}[F''(Z)Z^2]\widetilde{\mathfrak{D}}^2,\widetilde{\mathfrak{E}}-\frac{1}{2}\mathbb{E}\left[G_1'(Z)Z \right]\widetilde{\mathfrak{D}},\widetilde{\mathfrak{G}}-\frac{1}{2}\mathbb{E}\left[G_2'(Z)Z \right]\widetilde{\mathfrak{D}}\right)
\end{align*}
where $(\widetilde{\mathfrak{A}},\widetilde{\mathfrak{D}},\widetilde{\mathfrak{E}},\widetilde{\mathfrak{G}})$ is jointly Gaussian with mean $0$ and covariance matrix
\begin{align*}
\left(
\begin{array}{llll}
\mathrm{Cov}(F(Z),F(Z)) & \mathrm{Cov}(F(Z),Z^2) & \mathrm{Cov}(F(Z),G_1(Z)) &\mathrm{Cov}(F(Z),G_2(Z)\\
\mathrm{Cov}(Z^2,F(Z)) & \mathrm{Cov}(Z^2,Z^2) & \mathrm{Cov}(Z^2,G_1(Z)) & \mathrm{Cov}(Z^2,G_2(Z))\\
\mathrm{Cov}(G_1(Z),F(Z)) &\mathrm{Cov}(G_1(Z),Z^2) & \mathrm{Cov}(G_1(Z),G_1(Z)) & \mathrm{Cov}(G_1(Z),G_2(Z))\\
\mathrm{Cov}(G_2(Z),F(Z)) &\mathrm{Cov}(G_2(Z),Z^2) & \mathrm{Cov}(G_2(Z),G_1(Z)) & \mathrm{Cov}(G_2(Z),G_2(Z))
\end{array}
\right),
\end{align*}
and $Z$ is a standard normal random variable. 
\end{lemma}

This result is similar in spirit to~\cite[Theorem 1.1]{Kabluchko17}, which establishes a central limit theorem
for the sequence of $q$-norms of $\sqrt{n}\Theta^n$, $n \in \N$.   
 Lemma \ref{clt_expansion} above provides fluctuation  estimates
for suitable joint functionals of $\sqrt{n}\Theta^n$, for which
we first apply a Taylor expansion to the functionals.
The proof of Lemma  \ref{clt_expansion} is deferred to Appendix~\ref{subs-clt}.

\section{Proof of the sharp large deviation estimate for spheres}
\label{sec-pfmain}

Throughout this section, fix $p \in (1,\infty)$ and  for $n \in \N$, 
 recall from Section \ref{subs-reform} the definition of the two-dimensional
random vector $\Sn :=\Snp= \frac{1}{n} \sum_{j=1}^n (\sqrt{n} \Theta_j^n Y_j, |Y_j|^p)$,
where $(Y_j)_{j \in \N}$ is an i.i.d. sequence of random variables with common density $f_p$ as in
\eqref{pNormal}, 
and for $\theta \in \SS$, let  $\bar{h}^n_\theta$ denote the (joint) density of $\Sn$ under $\PP_{\theta}$,  where in this section we will typically
 suppress the dependence of $\bar{h}^n_\theta$, $\Sn$ and $Y_j$ and other quantities on $p$.
In view of  
\eqref{eq-reform}, 
we then have 
\begin{equation} 
\label{integral}
\mathbb{P}_\theta\left(\Wp>a\right)   = \mathbb{P}_\theta \left(\Snp\in \bDpa\right)  = \int_{\bDa} \bar{h}^n_\theta (x,y) dx dy, 
\end{equation}
where  $\bDa = \bDpa$  is the domain defined in \eqref{def-doma}.

\begin{remark}
  \label{rem-notcon}
  Note that $\bar{h}_\theta^n$ depends on $\theta$ only through $\theta^n$.  For notational
  simplicity throughout we will
  adopt the  convention that  for quantities that depend on both $n$ and $\theta^n$,
  we will use a superscript $n$ to denote the former dependence and a subscript $\theta$ instead of
$\theta^n$  to denote the dependence on $\theta^n$.  
\end{remark}

The key ingredients required to estimate the tail probability in \eqref{integral} are
an asymptotic expansion for the joint density
$\bar{h}^n_\theta$ carried out in Proposition \ref{prop-mainest} of
Section \ref{subs-densest}, a multi-dimensional generalized Laplace approximation stated in Proposition \ref{lem-Laplace} of Section \ref{subs-gen-laplace}, and a certain estimate that justifies the application of this Laplace approximation that is stated in 
Lemma \ref{lemma-estimate} of Section \ref{pf-estimate}.   
The proof of Proposition \ref{prop-mainest} is somewhat involved and 
hence deferred to Section \ref{sec-pfdens}.  Instead, these results are first 
used in Sections \ref{subs-mainproof} and \ref{subs-mainproof2} to prove Theorem \ref{main_lp}.  
We first state a preliminary result in Section \ref{subs-prep2}.

 \subsection{Estimates on the joint logarithmic moment generating function}
 \label{subs-prep2}

 We obtain an estimate on the  growth  of the log moment generating function $\Lambda_p$ of $(Y_j,\abs{Y_j}^p)$
 defined in~\eqref{logmgf_lp}, which will be useful in the subsequent discussion.  
The following expression was established in~\cite[Lemma 5.7]{GanKimRam17}:  
\begin{equation}
  \label{lambda-est}
\Lambda_p(t_1,t_2) = -\frac{1}{p}\log(1-pt_2)+\log M_{\gamma_p}\left(\frac{t_1}{(1-pt_2)^{1/p})} \right),
\end{equation}
for
\begin{equation}\label{dom-Lambda}
(t_1,t_2)\in \mathbb{D}_p:=\left \{(t_1,t_2)\in\mathbb{R}^2:t_2<1/p \right \},
\end{equation}
 where
\begin{equation} \label{mg\dens_p}
M_{\gamma_p}(t):=\mathbb{E}\left[e^{tY_j} \right],\quad t \in \mathbb{R},
\end{equation}
is the moment generating function of $Y_j$.
In order to understand the growth in $t_1$ of the derivatives of $\Lambda_p$, it suffices to understand the derivatives of $\log M_{\gamma_p}$.
\begin{lemma} \label{growth_lmgf}
For $1<p<\infty$,  let $M_{\gamma_p}$ and $\Lambda_p$ be as defined in~\eqref{mg\dens_p} and~\eqref{logmgf_lp}, respectively.  Then for every $k\in\mathbb{N}\cup\{0\}$, the derivative
\[
t\mapsto\frac{d^k}{dt^k} \log M_{\gamma_p}(t),
\]
exists and
has at most polynomial growth, in the sense of Definition~\ref{def-polygrowth}. Therefore, for $j,k\in\mathbb{N}\cup\{0\}$, and any $t_2<1/p$, the function
\[
t_1\mapsto\partial_1^j\partial_2^k\Lambda_p(t_1,t_2)
\]
has at most polynomial growth.
\end{lemma}
The proof of Lemma~\ref{growth_lmgf} involves  conceptually straightforward (though detailed) estimates, and is thus
deferred to Appendix~\ref{pf_growth}.

\subsection{An asymptotic expansion for the joint density}
\label{subs-densest}

The main result of this section is Proposition \ref{prop-mainest}, which
provides an asymptotic expansion for  the joint density $\bar{h}^n_\theta$ of the two-dimensional random vector 
  $\Sn$ under $\PP_{\theta}$.
To state the result, for  $n\in\mathbb{N}$,
define  
\begin{align}\label{Znj}
\bV^n_{j}:=(\sqrt{n} \Theta^n_jY_j,\abs{Y_j}^p), \qquad j = 1, \ldots, n. 
\end{align} 
For $t=(t_1,t_2)\in\mathbb{C}^2$, the Laplace transform of $\left(Y_j,\abs{Y_j}^p\right)$ is given by 
\begin{equation}\label{phi_p}
\Phi_p(t_1,t_2):= \mathbb{E}\left[ e^{t_1 Y_j+t_2\abs{Y_j}^p} \right].
\end{equation}
The observation $|e^{t_1 Y_j+t_2\abs{Y_j}^p}| = e^{\operatorname{Re}\{t_1\}Y_j+\operatorname{Re}\{t_2\}\abs{Y_j}^p}$ shows that $\Phi_p$ is finite precisely when $\operatorname{Re}\{t_2\}<1/p$, or equivalently, $(\operatorname{Re}\{t_1\},\operatorname{Re}\{t_2\})$ lies in $\mathbb{D}_p$, the effective domain of $\Lambda_p$ defined in \eqref{dom-Lambda}.
For $t=(t_1,t_2)\in\mathbb{D}_p$
and $\theta \in \SS$,
also define
  \begin{align}\label{psin_p}
\Psi^{n}_{p,\theta}(t) := \frac{1}{n}\sum_{j=1}^n\log\Phi_p(\sqrt{n}\theta^n_jt_1,t_2) = \int_\mathbb{R}\log\Phi_p(ut_1,t_2)\empn(du),
\end{align}
  where $\empn$ is the empirical measure of the coordinates of
  $\sqrt{n} \theta^n$, as  defined in~\eqref{emp_L}.

\begin{remark}
  \label{psipn}
  Since $\log\Phi_p = \Lambda_p$ on $\mathbb{D}_p$, for $(t_1,t_2)\in\mathbb{D}_p$, $\R\ni u\mapsto\log\Phi_p(ut_1,t_2)$ is continuous and has polynomial growth by Lemma~\ref{growth_lmgf}. Hence, for every $t=(t_1,t_2)\in \mathbb{D}_p$ and $\sigma$-a.e.\ $\theta$, the convergence of $L^n_\theta$ to $\gamma_2$ established in Lemma \ref{slln} shows that as $n\to\infty$,
  \[
  \Psi^{n}_{p,\theta}(t)\to\int_\R\log\Phi_p(ut_1,t_2)\gamma_2(du)=\Psi_p(t),
  \]
  where the last equality holds by the definition of $\Psi_p$ given in~\eqref{psi}.
\end{remark}

Next, recall the definition of $\mathbb{J}_p$ from
\eqref{def-Jp} and for  $x \in \mathbb{J}_p$, the definition of $\lambda_x$ from \eqref{def-Jp}-\eqref{gradPsi}.Then
for $\theta \in \SS$, define
\begin{equation}\label{Rn}
R^n_{x}(\theta^n):=\sqrt{n}(\Psi^{n}_{p,\theta}(\lambda_x)-\Psi_p(\lambda_x)), 
\end{equation}
and
\begin{align}
  \label{def-cnhessn}
  c^n_{x}(\theta^n)  :=\sqrt{n}\nabla\left(\Psi^{n}_{p,\theta}(\lambda_x)-\Psi_p(\lambda_x) \right),\qquad \qquad 
\Hess^n_{x}(\theta^n) :=\hess \Psi^{n}_{p,\theta}(\lambda_x),
\end{align}
where we drop the explicit dependence on $p$ from $c^n_{x}$,  $\Hess^n_{x}$  and  $R^n_x$, and note that the right-hand sides above depend on $\theta$ only through $\theta^n$ (see Remark \ref{rem-notcon}).

For $a > 0$, with the same abuse of notation
  used for $\Hess_a$ in Section \ref{subs-analres},  we let $c_a^n$ and $R_a^n$ denote
the functions  $c_{a^*}^n$ and $R_{a^*}^n$, respectively, where $a^* = (a,1)$. 
  We  show in Section \ref{subs-com} that $c^n_{x}(\theta^n)$ and $\Hess^n_{x}(\theta^n)$ are the mean vector and
covariance matrix, respectively, of $\frac{1}{\sqrt{n}}\sum_{j=1}^n (\bV_{j}^n -x)$, with $\bV_j^n$ as in~\eqref{Znj}, 
under a certain quenched tilted measure; see \eqref{cn} and~\eqref{Gammanx}.

\begin{proposition} \label{prop-mainest}
Fix $p \in (1,\infty)$, $n\in\N$,  and recall the
   definitions of $\Psi_p,  \Psi_p^*, {\mathbb J}_p$ and
   $\Psi^n_{p,\theta}$  given in \eqref{psi}, \eqref{eq:Psi}, \eqref{def-Jp}
   and \eqref{psin_p},  respectively, and for  
   $x \in {\mathbb J}_p$,  recall the definitions of
   $\Hess_{x}$,  $c^n_{x}(\cdot)$ and $R^n_x(\cdot)$ from
   ~\eqref{gamma_x},~\eqref{def-cnhessn} and \eqref{Rn}, 
  respectively.  
Then for  $\smeas$-a.e.\ $\theta$,
\begin{align} \label{gn_est}
  \bar{h}^n_\theta(x) = \frac{n}{2\pi}\gn (x)e^{-n\Psi^*_p(x)}(1+o(1)), 
\end{align}
and the expansion in \eqref{gn_est} is uniform on any compact subset of $\mathbb{J}_p$,
where $\gn$ is the infinitely differentiable function defined by
\begin{equation}
  \label{def-gn} 
  \gn(x) :=   (\det\Hess_{x})^{-1/2} e^{\sqrt{n} R^n_x(\theta^n)}e^{\norm{\Hess_{x}^{-1/2}c^n_{x}(\theta^n)}^2}.
\end{equation}

\end{proposition}

Section \ref{sec-pfdens} is devoted to 
establishing Proposition \ref{prop-mainest}, with the
final proof given in  Section \ref{sec-proppf}.

\subsection{A multi-dimensional generalized Laplace approximation}\label{subs-gen-laplace}
The formula \eqref{integral} and the expression for $\bar{h}^n_\theta$ in \eqref{gn_est}-\eqref{def-gn} show that the tail probability can be expressed as a Laplace-type integral over the domain $\bDa$ defined in \eqref{def-doma}. However, to estimate this integral, we cannot directly
apply conventional Laplace approximations such as those in \cite[Chapter 8]{Bleistein86} or \cite[Chapter V]{Wong01} due to the additional dependence of $n$ in $\gn$. 
Instead, in Propositions \ref{asymptotic-exp} and \ref{lem-Laplace}, we first establish a generalization of multi-dimensional Laplace approximations that can accommodate such $n$-dependent terms, which may be of independent interest.

\begin{definition}\label{def-int_form}
Given $m$, $d\in\N$, $\alpha \in (0,1)$ and a bounded domain $D\subset \R^{m+d}$, we say that the sequence $h^n:\R^{m+d}\to\R$, $n\in\N$, admits a $(f,x^*,\alpha,g^n)$-representation if for each $n\in\N$,
\[
h^n(x) = g^n(x)e^{-nf(x)},\quad x\in \R^{m+d},
\]
where 
\begin{enumerate}
\item $f$ is a nonnegative function that is twice continuously differentiable in $D$ and achieves its minimum on $\cl (D)$, the closure of $D$, at a unique point $x^*$,
\item there exists $C\in(0,\infty)$ such that for each $n\in\N$ sufficiently large, $g^n(x)=\exp(r^n(x))$ is continuously differentiable with 
\[
 \abs{r^n(x)}\leq Cn^\alpha\norm{x}_2 \quad \text{for all $x$ in a neighborhood of $x^*$.}
 \]
\end{enumerate}
\end{definition}

We start by establishing a Laplace asymptotics result, which extends the one-dimensional result in \cite[Chapter 9.2]{Olv97}.
\begin{proposition}\label{asymptotic-exp}
Given $m$, $d\in\N$, and a bounded domain $D\subset \R^m\times\R^d_+$ containing the origin. Suppose the sequence $h^n:\R^{m+d}\to\R$, $n\in\N$, admits  a $(f,x^*,\alpha,g^n)$-representation on $D$ with $x^*=(0,0,\ldots,0)$.
Then we have the following asymptotic expansion: 
\begin{align}\label{laplace-asymp}
\int_Dh^n(x)dx=\frac{(2\pi)^{\frac{m}{2}}}{n^{d+\frac{m}{2}}}\frac{g^n(x^*)}{\prod_{i=1}^d\partial_{m+i} f(x^*)\sqrt{\prod_{j=1}^m\abs{\partial^2_{j,j} f(x^*)}}}e^{-nf(x^*)}(1+o(1)).
\end{align}
\end{proposition}
The proof is deferred to Appendix \ref{app-asymptotic}. We now obtain an alternative representation for this integral.  To state the result we need to introduce the definition of
  Weingarten maps. Let $\mathcal{D}$ be a hypersurface in $\R^d$. Denote  the tangent space at a point $x\in\mathcal{D}$ to be $T_x(\mathcal{D})$ and the normal vector field at $x$ to be $N_x$. Then the Weingarten map at $x$ is defined to be the linear map $L_x:T_x(\mathcal{D})\to T_x(\mathcal{D})$ where $L_x(v):=\partial_vN_x$ and $\partial_v$ is the directional derivative in the direction of $v$.
  Also, for a map $L$, let $L^{-1}$ denote its inverse and 
recall that  $\det (A)$ denotes 
the determinant of a matrix $A$. (See also \cite[Section 4]{Andriani97} for more information on Weingarten maps).
\begin{proposition}
  \label{lem-Laplace}
  For $m,n\in \N$, let  $\newdom \subset \R^{m+d}$ be a bounded domain whose  boundary is a differentiable $(d-1)$-dimensional hypersurface.
  Let $h^n:\R^{m+d}\to\R$, $n \in \N$, be a sequence of functions that admits a $(f,x^*,\alpha,g^n)$-representation on $\mathcal{D}$ in the sense of Definition \ref{def-int_form}.
  Then
  \[  \Integ^n :=  \int_{\newdom} h^n(x)dx =   \frac{(2\pi)^{(d-1)/2}\det(L_1^{-1}(L_1-L_2))^{-1/2}}{n^{(d+1)/2}\langle(\hess \newpsi(x^*))^{-1} \nabla\newpsi(x^*),\nabla\newpsi(x^*)\rangle^{1/2}}g^n(x^*)  e^{-n\newpsi(x^*)}(1+o(1)), \]
  where for $i = 1, 2$, $L_i$ is the   Weingarten map at $x^* \in \partial \newdom$
  of the surface $\Curve_i$,  given by
  \[ \Curve_1 := \{y: \newpsi (y) = \newpsi (x^*) \} \qquad \mbox{  and }  \qquad 
  \Curve_2 := \partial \newdom. 
  \]
  \end{proposition}
  \begin{proof}
  The proof will make use of arguments from  \cite{Bleistein86} as well as a result from \cite{Andriani97}.
  Since $\partial \mathcal{D}$ is a differentiable $(d-1)$-dimensional hypersurface, there exists a one-to-one continuously differentiable transformation $\trans :\mathcal{N}\to \tilde{\mathcal{N}}\subset\R\times\R^{d-1}_+$ such that $F$ maps $x^*$ to the origin. Setting $J_\trans(x)\in\R^d\times\R^d$ to be the Jacobian matrix of $\trans$ at $x$, we can write
  \[
  \Integ^n:=\int_{\newdom} g^n(x)e^{-n\newpsi(x)} dx =   \int_{\trans(\newdom)} \abs{\det J_\trans(x)}g^n(\trans^{-1}(x)) e^{-n\newpsi(F^{-1}(x))}dx.
  \]
  By the assumption in Definition \ref{def-int_form} there exist $\alpha\in(0,1)$ and $C\in(0,\infty)$ such that
  $g^n(x)=\exp(r^n(x))$ with $\abs{r^n(x)}\leq Cn^\alpha\norm{x}_2$  on a neighborhood
  of $\newdom$. By the differentiability of $\trans$, we have $\abs{r^n(\trans^{-1}(x))}\leq Cn^\alpha\norm{x}_2$. Hence, Proposition~\ref{asymptotic-exp} with $m$, $d$ $g^n$, $f$ and $\mathcal{D}$ therein replaced with $d-1$, $1$, $ \abs{\det J_\Gamma(x)}g^n( \Gamma^{-1}(x))$, $\newpsi\circ \Gamma^{-1}$ and $\Gamma(\mathcal{D})$, respectively, implies there exists a constant $C'=C'(\Gamma,\mathcal{D},f)\in(0,\infty)$ that does not depend on $g^n$ such that
\begin{align}\label{eq-exp0.5}
  \mathcal{I}^n&=\frac{(2\pi)^{(d-1)/2}C'}{n^{(d+1)/2} }g^n(x^*)e^{-n\newpsi(x^*)}(1+o(1)).
  \end{align}
  
  In order to deduce the constant $C'$, we note that the same formula also holds when $g^n\equiv 1$ and hence it follows that 
  \begin{align}\label{eq-exp2}
\widetilde{\Integ}^n:=\int_{\newdom} e^{-n\newpsi(x)}dx=\frac{(2\pi)^{(d-1)/2}C'}{n^{(d+1)/2} }e^{-n\newpsi(x^*)}(1+o(1)).
\end{align}
  Also note that
     $\widetilde{\Integ}^n$ coincides with the integral in \cite[Equation (8.3.63)]{Bleistein86} when
  $\lambda, n, \phi$ and $g_0$ therein are replaced with $n, d$, $-\newpsi$ and $1$ here.
  By the stated properties of  $\newdom$, there exists a  local chart of a coordinate system
    $\mathcal{G}:\mathcal{N}_x\to\mathcal{U}$ of $\mathcal{N}_x\subset\partial\newdom$ around $x^*$, for some subset $\mathcal{U}\subset\R^{d-1}$.
    Let  $\Jac_*$ be the Jacobian matrix of the transformation $\mathcal{G}$ 
  at $x^*$, and let  $\Jac_*^T$ denote its transpose.   
 Then, under the stated conditions on $\newpsi$ and $\newdom$, the formula
  \cite[Equation (8.3.63)]{Bleistein86} yields the following estimate:
    \begin{align}\label{eq-exp1} 
 \widetilde{\Integ}^n = \frac{(2\pi)^{(d-1)/2}
      \left|{\rm det}(\Jac_*^T \Jac_*)\right|^{1/2}}{n^{(d+1)/2} |\det\hess(\newpsi\circ\mathcal{G}(x^*))|^{1/2} |\nabla \newpsi (x^*)|}e^{-n\newpsi(x^*)}(1+o(1)).
    \end{align}
  Next, to further simplify the expression in the last display, by  \cite[Equations (4.5) and (4.6)]{Andriani97} it follows, after identifying $DG(0)$, $I$ and $A$ therein with $\Jac_*$, $\newpsi$ and $\hess(\newpsi\circ\mathcal{G}(x^*))$, respectively, that
    \begin{align}\label{eq-exp1.5} \frac{ \left|{\rm det}(\Jac_*^T \Jac_*)\right|^{1/2}}{|\det\hess(\newpsi\circ\mathcal{G}(x^*))|^{1/2} |\nabla \newpsi (x^*)|}
  = \frac{\det(L_1^{-1}(L_1-L_2))^{-1/2}}{\langle(\hess \newpsi(x^*))^{-1} \nabla\newpsi(x^*),\nabla\newpsi(x^*)\rangle^{1/2}},
  \end{align}
  with $L_1, L_2$ as in the proposition. (Note that there is an erroneous additional factor of $\sqrt{2\pi n}$ in the denominator of the expression in \cite[Equation (4.6)]{Andriani97}, which we have corrected). Comparing \eqref{eq-exp1}, \eqref{eq-exp1.5} and \eqref{eq-exp2}, we see that 
  \begin{equation}\label{eq-C}
  C'=\frac{\det(L_1^{-1}(L_1-L_2))^{-1/2}}{\langle(\hess \newpsi(x^*))^{-1} \nabla\newpsi(x^*),\nabla\newpsi(x^*)\rangle^{1/2}}.
  \end{equation}
  The proposition then follows on substituting the above expression for $C'$ into \eqref{eq-exp0.5}.
    
\end{proof} 

\subsection{Continuity estimates for terms in the prefactor} \label{pf-estimate}
  In order to apply Proposition \ref{lem-Laplace}  to the expression for $\bar{h}_\theta^n$ given in  \eqref{gn_est}--\eqref{def-gn}, we need to 
verify that $\bar{h}_\theta^n$ satisfies Definition \ref{def-int_form}.   The following lemma will be useful in verifying property (2) of Definition \ref{def-int_form}.
\begin{lemma}\label{lemma-estimate}
Fix $p\in(1,\infty)$.
For every $y\in \mathbb{J}_p$, $\alpha\in(1/2,1)$ and $\varepsilon>0$ with a fixed $ B_{\varepsilon}(y)\subset \mathbb{J}_p$, there exists  $\Xi=\Xi(y,\alpha,\varepsilon)\in(0,\infty)$ such that for $\smeas$-a.e.\ $\theta$, there exists $N=N(y,\alpha,\varepsilon,\theta)\in\N$ such that for $x\in B_{\varepsilon}(y)$ and $n\geq N$, 
 \begin{align}\label{eq-Rb}
\abs{\sqrt{n}R^n_x(\theta^n)-\sqrt{n}R^n_y(\theta^n)}\leq \Xi n^\alpha\norm{x-y}_2,
\end{align}
and
\begin{align}\label{eq-cb}
\abs{\norm{\Hess_{x}^{-1/2}c^n_{x}(\theta^n)}_2^2-\norm{\Hess_{y}^{-1/2}c^n_{y}(\theta^n)}_2^2}\leq \Xi n^{2\alpha-1}\norm{x-y}_2.
\end{align}
\end{lemma}

Before presenting the proof of the lemma, we provide an alternative formulation of \eqref{eq-Rb} and a related result. Fix $\alpha\in(1/2,1)$, $p\in(1,\infty)$, $y\in\mathbb{J}_p$ and $\varepsilon>0$ such that  $B_{\varepsilon}(y)\subset\mathbb{J}_p$.  
From \eqref{logmgf_lp}, \eqref{phi_p} and the fact that $Y_j$ is a $p$-Gaussian random variable, it follows that the equality $\Lambda_p=\log\Phi_p$ holds on the domain $\mathbb{D}_p$ of $\Phi_p$ defined in \eqref{dom-Lambda}. When combined with \eqref{Rn}, \eqref{psin_p} and \eqref{psi}, this shows that
\[
\sqrt{n}R^n_x(\theta^n) = \sum_{j=1}^n\left(\log\Phi_p(\sqrt{n}\theta^n_j\lambda_{x,1},\lambda_{x,2})-\mathbb{E}\left[\log\Phi_p(Z\lambda_{x,1},\lambda_{x,2})\right]\right),
\]
where $Z\sim \gamma_2$ is a standard normal random variable. Since $x\mapsto\lambda_x$ is infinitely differentiable by Remark \ref{rem-smooth}, 
there exists $C'\in(0,\infty)$ such that $\norm{\lambda_x-\lambda_y}\leq C'\norm{x-y}$ for $x\in B_{\varepsilon}(y)$. Therefore, to show \eqref{eq-Rb}, it suffices to show that given any fixed $s=(s_1,s_2)\in \mathbb{D}_p$, for every $\varepsilon'>0$ such that  $B_{\varepsilon'}(s)\subset\mathbb{D}_p$,  there exist $C=C(s,\alpha,\varepsilon')\in(0,\infty)$  and a random integer $N=N(s,\alpha,\varepsilon')$ such that $\mathbb{P}$-almost surely,
\begin{equation}\label{eq-wts5}
\abs{ \sum_{j=1}^n\left(\mathcal{K}_s(\sqrt{n}\Theta^n_jt_1,t_2)-\mathbb{E}\left[\mathcal{K}_s(Zt_1,t_2)\right]\right)}\leq Cn^\alpha\norm{t},\quad \text{for} \quad \norm{t}<\varepsilon' \quad \text{and $n\geq N$},
\end{equation}
where 
\begin{align}\label{eq-K}
\mathcal{K}_s(t_1,t_2):=\log\Phi_p((s_1+t_1),(s_2+t_2))-\log\Phi_p(s_1,s_2),
\end{align}
for 
\begin{align}
(t_1,t_2)\in\mathbb{D}_{p,s}:=\{(u_1,u_2)\in\R^2:u_2<1/p-s_2\}.\label{def-dps}
\end{align}

Let $Z$, $(Z_j)_{j\in\mathbb{N}}$ be independent standard Gaussian random variables on the probability space $(\Omega', \mathcal{F}',\mathbb{P}')$. Then  letting $Z^{(n)}:=(Z_1,\ldots,Z_n)$ we have (e.g. see Section~\ref{subs-reform} or~\cite[Lemma 1]{Schechtman90}),
\begin{equation}\label{eq-eqind}
 \left(\Theta^n_1,\ldots,\Theta^n_n\right) \buildrel (d) \over = \frac{(Z_1,\dots,Z_n)}{\norm{Z^{(n)}}}.
\end{equation}
The following result on Gaussian vectors will be used to prove Lemma \ref{lemma-estimate}.
\begin{lemma}\label{lem-conti-K}
With the notation above, fix $\alpha\in(1/2,1)$, let $\mathbb{D}=\R\times(-\infty,T)$ and let $\mathcal{K}:\mathbb{D}\to \R$ be a twice continuously differentiable function. Suppose for $t_2\in(-\infty,T)$, the mappings $t_1\mapsto\partial_1\mathcal{K}(t_1,t_2)$, $t_1\mapsto\partial_{12}\mathcal{K}(t_1,t_2)$ and $t_1\mapsto\partial_{11}\mathcal{K}(t_1,t_2)$ have polynomial growth in the sense of Definition \ref{def-polygrowth}. Then  for every $\varepsilon>0$ such that $B_\varepsilon(0)\subset \mathbb{D}$, there exist $C=C(s,\alpha,\varepsilon)\in(0,\infty)$ and a random integer $N=N(s,\alpha,\varepsilon)$ on $(\Omega',\mathcal{F}',\mathbb{P}')$ such that $\mathbb{P}'$-almost surely, for 
\begin{equation}\label{eq-wts}
\abs{ \sum_{j=1}^n\left(\mathcal{K}\left(\frac{\sqrt{n}Z_j}{\norm{Z^{(n)}}}t_1,t_2\right)-\mathbb{E}\left[\mathcal{K}(Zt_1,t_2)\right]\right)}\leq Cn^\alpha\norm{t},\quad \text{for $\norm{t}<\varepsilon$ and  $ n\geq N$}.
\end{equation}
\end{lemma}
Deferring the proof of Lemma \ref{lem-conti-K} to Appendix \ref{app-GC-class}. we now use it to prove Lemma \ref{lemma-estimate}.

\begin{remark}\label{rm-equiv}
From \eqref{eq-eqind}, we will use in our proof the following equivalence that a statement about $Z^{(n)}/\|Z^{(n)}\|$ holds $\mathbb{P}'$-almost surely if and only if the same statement with $Z^{(n)}/\|Z^{(n)}\|$ replaced by $\Theta^n$ holds $\mathbb{P}$-almost surely.
\end{remark}
\begin{proof}[Proof of Lemma \ref{lemma-estimate}]
We begin with the proof of \eqref{eq-Rb}.  Since $\Phi_p$ is finite near the origin, $\Phi_p$ is infinitely differentiable on its domain $\mathbb{D}_p$ and hence, for any $s\in\mathbb{D}_p$, the functional $\mathcal{K}_s$ from \eqref{eq-K} is twice continuously differentiable on its domain $\mathbb{D}_{p,s}$ defined in \eqref{def-dps}. 
Since $\log \Phi_p = \Lambda_p$ on $\mathbb{D}_p$, the expression in \eqref{eq-K} and  Lemma \ref{growth_lmgf} imply that
 the mappings $t_1\mapsto\partial_1\mathcal{K}_s(t_1,t_2)$, $t_1\mapsto\partial_{12}\mathcal{K}_s(t_1,t_2)$ and $t_1\mapsto\partial_{11}\mathcal{K}_s(t_1,t_2)$ have polynomial growth for $t_2<1/p-s_2$. Therefore, Lemma \ref{lem-conti-K} implies that for each $s\in\mathbb{D}_p$ and $\alpha\in(1/2,1)$, for $\varepsilon>0$ with  $B_\varepsilon(0)\subset \mathbb{D}$, there exist $C=C(s,\alpha,\varepsilon)\in(0,\infty)$ and a random integer $N=N(s,\alpha,\varepsilon)$ such that \eqref{eq-wts} holds $\mathbb{P}'$-almost surely.
 Due to the relation $\sigma_n = \mathbb{P} \circ (\Theta^{(n)})^{-1}$, 
  Remark \ref{rm-equiv} implies that\eqref{eq-wts5} holds and thus, that \eqref{eq-Rb} also holds.

We now turn to the proof of  \eqref{eq-cb}. Fix $s\in\mathbb{D}_p$. For $i=1,2$, and $(t_1,t_2)\in\mathbb{D}_{p,s}$, define
\begin{align}\label{eq-KK1}
\tilde{\mathcal{K}}_{i}(t_1,t_2):=\partial_i\log\Phi_p(t_1,t_2),
\end{align}
and
\begin{align}\label{eq-KK}
\mathcal{K}_{s,i}(t_1,t_2):=\tilde{\mathcal{K}}_{i}((s_1+t_1),(s_2+t_2))-\tilde{\mathcal{K}}_{i}(s_1,s_2).
\end{align}
Note that for $i=1,2$, by the smoothness of $\Phi_p$, $\mathcal{K}_{s,i}$ is twice continuously differentiable on its domain $\mathbb{D}_{p,s}$ and also the mappings $t_1\mapsto\partial_1\mathcal{K}_{s,i}(t_1,t_2)$, $t_1\mapsto\partial_{12}\mathcal{K}_{s,i}(t_1,t_2)$ and $t_1\mapsto\partial_{11}\mathcal{K}_{s,i}(t_1,t_2)$ have polynomial growth for $t_2\in\mathbb{D}_{p,s}$ by Lemma \ref{growth_lmgf} and the fact that $\log \Phi_p = \Lambda_p$ on $\mathbb{D}_p$ . Thus,  for $i=1,2$,
$\mathcal{K}_{s,i}$ satisfies the assumption in Lemma \ref{lem-conti-K}.
Hence, \eqref{def-cnhessn}, \eqref{psin_p}, \eqref{eq-eqind}, Lemma \ref{lem-conti-K} and the equivalence between statements about $Z^{(n)}/\|Z^{(n)}\|$ and $\Theta^n$,
 imply that for every $y\in \mathbb{J}_p$, $\alpha\in(1/2,1)$ and $\varepsilon>0$ with $B_{\varepsilon}(y)\subset \mathbb{J}_p$, there exists $\Xi=\Xi(y,\varepsilon,\alpha)\in(0,\infty)$,  and for $\smeas$-a.e.\ $\theta$, there exists $N=N(y,\alpha,\varepsilon,\theta)\in\N$ such that for $x\in B_{\varepsilon}(y)$ and $n\geq N$,  
 \begin{align}\label{eq-cb1}
\abs{\sqrt{n}c^n_{x,i}(\theta^n)-\sqrt{n}c^n_{y,i}(\theta^n)}\leq \Xi n^\alpha\norm{x-y}_2, \quad i=1,2,
\end{align}
where $c^n_{y,i}(\theta^n)$ denotes the $i$-th coordinate of $c^n_x(\theta^n)$. By the smoothness of $x\mapsto \lambda_x$ in Remark \ref{rem-smooth}, there exists  $C'\in(0,\infty)$ such that $\norm{\lambda_x}\leq C'$ for $x\in B_{\varepsilon'}(y)$. Hence, \eqref{def-cnhessn} and Lemma \ref{lem-conti-K}, with $\mathcal{K}$ replaced by $\tilde{\mathcal{K}}_{i}$, imply
that for every $y\in \mathbb{J}_p$ and $\alpha\in(1/2,1)$, there exists $\Xi'=\Xi'(y,\varepsilon,\alpha)\in(0,\infty)$, and for $\smeas$-a.e.\ $\theta$, there exists $N'=N'(y,\alpha,\varepsilon,\theta)\in\N$ such that for $x\in B_{\varepsilon}(y)$ and $n\geq N'$,  
 \begin{align}\label{eq-cb11}
\abs{\sqrt{n}c^n_{x,i}(\theta^n)}\leq \Xi' n^\alpha\norm{\lambda_x}_2\leq C'\Xi' n^\alpha, \quad i=1,2.
\end{align}

Recall the definition of $\Hess_x$ in \eqref{gamma_x}. The following relation
\begin{align*}
\norm{\Hess_{x}^{-1/2}c^n_{x}(\theta^n)}_2^2=\sum_{i=1}^2\left(\sum_{j=1}^2(\Hess_x^{-1/2})_{ij}c^n_{x,j}(\theta^n)\right)^2
\end{align*}
implies that the left-hand side of \eqref{eq-cb} can be written as
\begin{align}\label{eq-cb2}
\abs{\norm{\Hess_{x}^{-1/2}c^n_{x}(\theta^n)}_2^2-\norm{\Hess_{y}^{-1/2}c^n_{y}(\theta^n)}_2^2}\leq \sum_{i=1}^2A^n_i\tilde{A}^n_i,
\end{align}
where for $i=1,2$,
\begin{align*}
A^n_{i}:=\abs{\sum_{j=1}^2((\Hess_x^{-1/2})_{ij}c^n_{x,j}(\theta^n)-(\Hess_y^{-1/2})_{ij}c^n_{y,j}(\theta^n))},\\
\tilde{A}^n_{i}:=\abs{\sum_{j=1}^2((\Hess_x^{-1/2})_{ij}c^n_{x,j}(\theta^n)+(\Hess_y^{-1/2})_{ij}c^n_{y,j}(\theta^n))}.
\end{align*}

By the smoothness of $\Psi_p$ and $\lambda_x$ in Remark \ref{rem-smooth}, $x\mapsto\Hess^{-1/2}_x$ is also infinitely differentiable and there exists $C''\in(0,\infty)$ such that $|(\Hess^{-1/2}_x)_{ij}-(\Hess^{-1/2}_y)_{ij}|\leq C''\norm{x-y}$ and $|(\Hess^{-1/2}_x)_{ij}|\leq C''$ for $x\in B_{\varepsilon}(y)$ and $i,j=1,2$. Hence, by \eqref{eq-cb1} and \eqref{eq-cb11}, for $\smeas$-a.e.\ $\theta$, $x\in B_{\varepsilon}(y)$ and $n\geq \max\{N,N'\}$,
\begin{align} 
A^n_i
&\leq \sum_{j=1}^2\abs{((\Hess_x^{-1/2})_{ij}c^n_{x,j}(\theta^n)-(\Hess_y^{-1/2})_{ij}c^n_{y,j}(\theta^n))}\nonumber\\
&\leq \abs{(\Hess_x^{-1/2})_{ij}-(\Hess_y^{-1/2})_{ij}}\sum_{j=1}^2\abs{c^n_{x,j}(\theta^n)}+\abs{(\Hess_x^{-1/2})_{ij}}\sum_{j=1}^2\abs{c^n_{x,j}(\theta^n)-c^n_{y,j}(\theta^n))}\nonumber\\
&\leq 2(C''C'\Xi'+C''\Xi) n^{\alpha-1/2}\norm{x-y}_2.\label{eq-cb3}
\end{align}
Next, by \eqref{eq-cb11}, for $\smeas$-a.e.\ $\theta$, $x\in B_{\varepsilon}(y)$ and $n\geq N'$,
\begin{align}\label{eq-cb4}
\tilde{A}^n_i \leq 2C''C'\Xi'n^{\alpha-1/2}.
\end{align}
Therefore \eqref{eq-cb} follows from \eqref{eq-cb2}-\eqref{eq-cb4}.

\end{proof}

\subsection{Proof of Theorem~\ref{main_lp}\ref{main_lp_1}}
\label{subs-mainproof}

We are now ready to prove the main estimate \eqref{tail_prob}.  
 Fix $p \in (1,\infty)$ and $a > 0$ 
 such that $\mathbb{I}_p(a) < \infty$ and recall the definition of the domain $\bDa = \bDpa$ given in~\eqref{def-doma}.
Since $\mathbb{I}_p(a)$ is convex and symmetric, $\mathbb{I}_p(a)$ is increasing for $a\in\R_+$. Thus, 
 \eqref{rate_func} and  Lemma~\ref{inf_ratefunc} imply that
\begin{align*}
\inf_{x\in \bDa}\Psi_p^*(x) &= \inf_{t> a}\mathbb{I}_p(t)
=\mathbb{I}_p(a)
= \inf_{\tau_1\in\mathbb{R},\tau_2>0:\tau_1\tau_2^{-1/p}=a}\Psi^*_p(\tau_1,\tau_2)
 =\Psi^*_p(a,1).
\end{align*}
Hence, the infimum of $\Psi_p^*$ over the closure $\cldom$ of $\bDa$  
is attained at $\xstar:=(a,1)$.  
Moreover due to \eqref{rate-fn}, the assumption $\mathbb{I}_p(a)<\infty$ implies  $\Psi^*_p(a,1)<\infty$, and hence,
$a^* = (a,1)\in\mathbb{J}_p$, defined in~\eqref{def-Jp}. Further, by \eqref{def-doma}, $\xstar$ is a point on the smooth part of the the boundary 
  $\partial \bDa$ of $\bDa$.
Let $U \subset \R^2_{>} := \{(x,y) \in\R^2: x > 0, y > 0\}$ be an open neighborhood 
of $\xstar$ to be chosen below
and note that the boundary of  $U\cap \bDa$ is also smooth at  $a^*$.
Then, for  $\theta \in \SS$,
we can split the probability of interest from \eqref{integral} into two parts: 
\begin{align}\label{split}
\mathbb{P}_\theta\left(\Sn\in \bDa\right) = \mathbb{P}_\theta\left(\Sn\in \bDa\cap U\right) +\mathbb{P}_\theta\left(\Sn\in \bDa\cap U^c\right).
\end{align}
The proof will proceed in two steps.
  In the key first step, we will
  estimate the first term on the right-hand side of \eqref{split}
  by  integrating the estimate of the density
$\bar{h}^n_\theta$ of $\Sn$ obtained
  in Proposition \ref{prop-mainest} over the domain $\bDa \cap U$, 
  and then analyze the asymptotics of the resulting Laplace type integral, as $n \rightarrow \infty$ using Proposition \ref{prop-mainest}, Proposition \ref{lem-Laplace} and Lemma \ref{lemma-estimate}. 
  The second step will involve  using
  the LDP for $(\Sn)_{n \in \N}$ to show that the second term on the right-hand side
of \eqref{split} is negligible.

\noindent 
{\bf Step 1.}  Using the expressions for $\bar{h}_\theta^n$
    from~\eqref{gn_est}
     and the fact that  the domain $\bDa \cap U\subset \R^2$ is bounded,  we have for $\sigma$-a.e.\ $\theta$, 
\begin{align}
  \mathbb{P}_\theta\left(\Sn\in \bDa\cap U\right) &= 
   \int_{\bDa\cap U}\bar{h}^n_\theta(x)dx = \frac{n}{2 \pi} \Integ^n_\theta(1+o(1)),
  \label{split11}
\end{align}
where 
\begin{align}
  \label{split1-1}
\Integ^n_{\theta} &:= \int_{\bDa\cap U} \gn (x) e^{-n\Psi^*_p(x)} dx,
\end{align}
 where $\gn$ is as defined in \eqref{def-gn}.

To apply Proposition \ref{lem-Laplace}, we first prove the following:

\begin{lemma}\label{lem-gn-rep}
For $\alpha\in(1/2,1)$, the function $ \gn (x) e^{-n\Psi^*_p(x)}$ admits a $(\Psi^*_p,a^*,\alpha,\gn)$-representation on the bounded region ${\bDa\cap U}$.
\end{lemma}

\begin{proof}
To verify property (1) of Definition \ref{def-int_form}, first note that $\Psi^*_p$ is nonnegative since it is a rate function by Theorem \ref{LDP_p}. Next note that by \eqref{psi}, \eqref{lambda-est} and Lemma \ref{growth_lmgf}, $\Psi_p$ is twice
     (in fact infinitely)  differentiable on $\mathbb{D}_p = \R \times \{t_2: t_2 < 1/p\}$. Hence, by the duality of the
     Legendre transform~\cite[Section III.D]{Zia09}, it follows that  $\Psi_p^*$
     is twice differentiable in $\bDa$  and achieves its minimum uniquely
     at $a^*=(a,1) \in \partial(\bDa\cap U)$. Thus, property (1) of Definition \ref{def-int_form} holds.

We next turn to the verification of property (2) of Definition \ref{def-int_form}. From \eqref{def-gn}, it follows that $\gn(x)=\exp(r^n_\theta(x))$, where
     \[
     r^n_\theta(x):=\log \det \Hess_x^{-1/2}+\sqrt{n}R^n_x(\theta^n)+\norm{\Hess_{x}^{-1/2}c^n_{x}(\theta^n)}^2, \quad x\in\bDa\cap U.
     \]
 Lemma \ref{lemma-estimate} and the smoothness of $x\mapsto\log\Hess_x^{-1/2}$, which follows from Remark \ref{rem-smooth} and \eqref{gamma_x}, imply that for a sufficiently small neighborhood of $x^*$, for any $\alpha\in(1/2,1)$, there exist $C\in(1,\infty)$ and a finite random variable $N$ such that for $\sigma$-a.e.\ $\theta$, $r^n_\theta$ satisfies property (2) of Definition \ref{def-int_form} and the claim follows.  
 \end{proof}
 
 Given the claim, Proposition \ref{lem-Laplace} applied with $d=2$, $\mathcal{D}=\bDa\cap U$ and $h^n(x)=\gn (x) e^{-n\Psi^*_p(x)}$ shows that for $\sigma$-a.e.\ $\theta$,
 \begin{align}
\Integ^n_{\theta}= \frac{(2\pi)^{1/2}}{n^{3/2}}\frac{(L^{-1}_{a,1}(L_{a,1}-L_{a,2}))^{-1/2}}{\langle (\hess\Psi_p^*(a^*))^{-1}\nabla\Psi^*_p(\xstar),\nabla\Psi^*_p(\xstar)\rangle^{1/2}}\gn(a^*)e^{-n\Psi^*_p(\xstar)}(1+o(1)),\label{split12}
\end{align}
where $L_{a,1}$ and $L_{a,2}$ are the Weingarten maps  of the curves  ${\mathcal C}_1 : = \{x\in\mathbb{R}^2:\Psi^*_p(x)=\Psi^*_p(a,1)\}$ and ${\mathcal C}_2 := \{x\in\mathbb{R}^2:x_1= ax_2^{1/p}\}$, evaluated at $\xstar = (a,1)$.
To further simplify \eqref{split12}, first note that by the duality
 of the Legendre transform,
 and the definition of $\lambda_{a,j}$ in \eqref{lambdaa}, we have
\begin{align}
  \partial_j\Psi^*_p(\xstar) = \lambda_{a,j}, \quad \text{for}\quad j=1,2,
  \label{duality1}
\end{align}
and
\begin{align}
  (\hess\Psi_p^*(a^*))^{-1} = \hess\Psi_p(\lambda_a)  = \Hess_{a}. 
  \label{duality2}
\end{align}

Hence,
\begin{align}
\Integ^n_{\theta}= \frac{(2\pi)^{1/2}}{n^{3/2}}\frac{(L^{-1}_{a,1}(L_{a,1}-L_{a,2}))^{-1/2}}{\langle \Hess_{a}\lambda_a,\lambda_a\rangle^{1/2}}\gn(a^*)e^{-n\Psi^*_p(\xstar)}(1+o(1)).\label{split122}
\end{align}

 Next, observe that \cite[Example 4.3]{Andriani97} shows that in $\mathbb{R}^2$, the Weingarten map is reduced to  multiplication by the inverse of the radius of the osculating circle, which is equal to the absolute value of the curvature.  
Recall that for a curve in $\R^2$ defined by the equation  $T(x,y)=0$ for a sufficiently smooth map $T:\R^2 \to  \R$,
 the curvature at a point $x^*$ on the curve is given by
the formula  
\[
\frac{T_y^2T_{xx}-2T_xT_yT_{xy}+T_x^2T_{yy}}{(T_x^2+T_y^2)^{3/2}} (x^*).
\]
Thus, to calculate the  curvature of the curve ${\mathcal C}_1$ at $\xstar$, use
the above formula with $T(x,y) = \Psi^*_p(x,y) - \Psi^*_p(a,1)$ and $x^* = \xstar$,  and substitute  the relations 
$\partial_j\Psi^*_p(\xstar) = \lambda_{a,j}$, $j=1,2$, and the definition of $\Hess_a$ mentioned above to 
conclude that
\begin{align}
L_{a,1}&=\frac{\abs{\lambda_{a,2}^2(\Hess_{a}^{-1})_{11}-2\lambda_{a,1}\lambda_{a,2}(\Hess_{a}^{-1})_{12}+\lambda_{a,1}^2(\Hess_{a}^{-1})_{22}}}{(\lambda_{a,1}^2+\lambda_{a,2}^2)^{3/2}}.
\end{align}
On the other hand, the curvature of the graph of a  function $y=\widetilde{T}(x)$ at the point $(x,\widetilde{T}(x))$
for sufficiently smooth $\widetilde{T}: \R \to \R$ is given by 
$\abs{\widetilde{T}''(x)}/(1+(\widetilde{T}')^2(x))^{3/2}.$ 
  Recalling the definition of $\bDa$ from \eqref{def-doma}, we can apply this with $\widetilde{T}(x) = (x/a)^p$ to
  compute the curvature of  ${\mathcal C}_2 = \partial \bDa$ at $\xstar$ as: 
\begin{equation} 
L_{a,2} = \frac{p(p-1)a}{(a^2+p^2)^{3/2}}.
\end{equation}
Substituting these calculations back into the expressions ~\eqref{split11} and~\eqref{split122}, and
recalling the definitions of $\gn$ from \eqref{def-gn}, $C^n_a(\theta^n)$ from \eqref{Cna} and $\den_a$ and $\numer_a$ from \eqref{def-sigmapa} and \eqref{def-kappapa}, we conclude that for $\sigma$-a.e.\ $\theta$,
\begin{align}\label{split1}
\mathbb{P}_\theta\left(\Sn \in \bDa\cap U\right) = \frac{C^n_a(\theta^n)}{ \numer_a \den_{a}\sqrt{2\pi n} }
e^{-n\mathbb{I}_{p}(a)  + \sqrt{n} R^n_{a}(\theta^n)}(1+o(1)).
\end{align}

\noindent
    {\bf Step 2.} We now turn to  the second term in~\eqref{split}.
    Note that by the continuity of $\Psi^*_p$, there exists $\eta >0$ such that
\[
\inf_{y\in \bDa\cap U^c}\Psi^*_p(y)>\Psi^*_p(\xstar)+\eta.
\]
By the refinement in Lemma~\ref{inf_ratefunc} of the (quenched) large deviation principle for
$\Sn$ established in \cite[Proposition 5.3]{GanKimRam17}, $\Psi_p^*$ achieves its unique minimum in $\bDa$ at $a^*=(a,1)$. Thus, for $\sigma$-a.e.\ $\theta$,
\begin{align}\label{split2}
  \limsup_{n\to\infty}\frac{1}{n}\log \mathbb{P}_\theta\left(\Sn
  \in \bDa\cap U^c\right)\leq -\Psi^*_p(\xstar)-\eta, 
\end{align}
which shows that  the term in~\eqref{split2} is negligible with respect to~\eqref{split1}.

When combined, \eqref{Cna}, \eqref{Rn}, \eqref{def-gn}, \eqref{split},  \eqref{split1} and \eqref{split2}
together yield~\eqref{tail_prob}.   This completes the proof of Theorem \ref{main_lp_1}.

\subsection{Proof of Theorem~\ref{main_lp}\ref{main_lp_2}}
\label{subs-mainproof2}
 We start by obtaining expansions for $R^n_{a}(\Theta^n)$ and $c^n_a(\Theta^n)$. 
First, note that
the functions $\ell_a$, $\ell_{a,1}$ and $\ell_{a,2}$ defined in~\eqref{f_pa}  and their derivatives up to second order (for $\ell_{a,1}$ and $\ell_{a,2}$) and third order (for $\ell_a$) are continuous and have at most polynomial growth by Lemma~\ref{growth_lmgf}.
Therefore, setting 
\[
r_n :=\hat{r}_n(\ell_a), \quad s_n:=\hat{s}_n(\ell_a), \quad t_{n,1} :=\hat{r}_n(\ell_{a,1}), \quad t_{n,2:}=\hat{r}_n(\ell_{a,2}),
\]
where $\hat{r}_n$ and $\hat{s}_n$ are defined in \eqref{def-rn} and \eqref{def-sn}, respectively, we can apply \eqref{def-cnhessn}, \eqref{Rn} and Lemma~\ref{clt_expansion} to obtain
\begin{align*}
R^n_{a}(\Theta^n) &= \frac{1}{\sqrt{n}}\sum_{i=1}^n\left(\ell_a(\sqrt{n}\Theta^n_i)-\mathbb{E}[\ell_a(Z)]\right) \\
&\buildrel (d) \over = r_n+\frac{1}{\sqrt{n}}s_n+o\left(\frac{1}{\sqrt{n}}\right), \\
c^n_a(\Theta^n)&=
\frac{1}{\sqrt{n}}\sum_{i=1}^n\left (
\begin{array}{l}
\ell_{a,1}(\sqrt{n}\Theta^n_i)-\mathbb{E}[\ell_{a,1}(Z)] \\
\ell_{a,2}(\sqrt{n}\Theta^n_i)-\mathbb{E}[\ell_{a,2}(Z)]
\end{array}
\right)
\\
&\buildrel (d) \over = \left (
\begin{array}{l}
t_{n,1} \\
t_{n,2}
\end{array}
\right) +o(1). 
\end{align*}
Moreover, Lemma~\ref{clt_expansion} also shows that
we have the convergence
\begin{eqnarray*} 
  && (r_n,s_n,t_{n,1},t_{n,2})\Rightarrow \\
  && \qquad \left(\widetilde{\mathfrak{A}}-\frac{1}{2}\mathbb{E}[\ell'_a(Z)Z]\widetilde{\mathfrak{D}},\frac{1}{8}\mathbb{E}[\ell''_a(Z)Z^2]\widetilde{\mathfrak{D}}^2,\widetilde{\mathfrak{E}}-\frac{1}{2}\mathbb{E}\left[\ell_{a,1}'(Z)Z \right]\widetilde{\mathfrak{D}},\widetilde{\mathfrak{G}}-\frac{1}{2}\mathbb{E}\left[\ell_{a,2}'(Z)Z \right]\widetilde{\mathfrak{D}}\right),
\end{eqnarray*}
where $(\widetilde{\mathfrak{A}},\widetilde{\mathfrak{D}},\widetilde{\mathfrak{E}},\widetilde{\mathfrak{G}})$ is jointly Gaussian with mean $0$ and  covariance matrix~\eqref{cov_p}.

\section{Proof of the sharp large deviation estimate for balls}\label{sec-pfmain-b}

\subsection{Preliminary Notation}

Fix $p \in (1,\infty)$ and $a > 0$
such that $\mathbb{I}_p(a) < \infty$.
The definitions in Section~\ref{subs-reform}, specifically \eqref{eq-reform-b}, 
yield the following 
expression for the tail probability of projections of $\ell_p^n$ balls:
\begin{align}\label{int-rep}
\mathbb{P}_\theta\left(\Wpb>a\right) = \int_{\bDab}\bar{\mathfrak{h}}^n_\theta(x_1,x_2,y)dx_1dx_2dy,
\end{align}
where for $\theta \in \mathbb{S}$, $\bar{\mathfrak{h}}^n_\theta(x_1,x_2,y)$ is the density under $\mathbb{P}_{\theta}$ of
the random vector $\Snpb = (\Snp, \mathscr{U}^{1/n})$  defined in~\eqref{joint_pdf-b}, and $\bDab:=\bDpab \subset \R^3$ is the domain defined in~\eqref{def-doma-b}.
By the independence of $\mathscr{U}$ and $Y^{(n,p)}$,  for   $x \in \R^2$ and $y\in (0,1]$, 
  $\bar{\mathfrak{h}}^n_\theta(x_1,x_2,y)$ is the product
  of $\bar{h}^n_\theta (x_1, x_2)$,  the density  of $\Snp$ under $\mathbb{P}_\theta$ evaluated at $(x_1, x_2)$,
  and the density of  $\mathscr{U}^{1/n}$ at $y$,
  which is equal to  $\frac{n}{y}e^{n\log y}$.
  Hence, by  Proposition~\ref{prop-mainest}, we have the following uniform estimate for 
  $\bar{\mathfrak{h}}^n_\theta$: for $\sigma$ a.e. $\theta$,  
\begin{align} \label{gn_est-b}
  \bar{\mathfrak{h}}^n_\theta(x_1,x_2,y) &= \frac{n^2}{2\pi}\gnt(x_1,x_2,y)  e^{-n  \mfn(x_1,x_2,y)}(1+o(1)), \quad (x_1, x_2) \in\mathbb{R}^2,\quad y\in (0,1], 
\end{align}
where 
\begin{align}\label{eq-gnt}
\gnt(x_1,x_2,y) := \frac{1}{y}\gn  (x_1,  x_2),
\end{align}
with $\gn$ defined in \eqref{def-gn}, and
\begin{equation}
  \label{def-mfn}
  \mfn(x,y) := \Psi_p^* (x) - \log y, \quad x = (x_1, x_2) \in \mathbb{R}^2, y \in (0,1]. 
\end{equation}

    Thus, as in Section \ref{subs-mainproof}, the integral \eqref{int-rep} of interest
    is once again a Laplace-type integral, and so one expects the significant contribution to
    come from  the value of the integrand in   a neighborhood of the point   
    where the minimum of $\mfn$ over $\bDab$ is achieved.
Now, for any $x \in \R^2$, the minimum of $F(x,y)$ over $y \in (0,1]$ is clearly attained when  $y = 1$,
and by Lemma~\ref{inf_ratefunc} the minimum of $F(x,1)$ over the region
$\{x \in \R^2: x_2 > 0, x_1 x_2^{1/p} = a\}$ is attained at $x = (a,1)$.
Together with 
the strict convexity of the function $\Psi_p^*(a,1)$ established in Theorem \ref{LDP_p} and the fact that
its minimum is attained at $0$, this shows that for  $a > 0$, the minimizing point is given by 
\begin{equation}
    \label{mfn-min}
    \mathop{\arg \min}\limits_{(x_1,x_2,y) \in \bDab} \mfn(x_1,x_2, y) =  \mathop{\arg \min}
    \limits_{(x_1,x_2,y):  0 \leq y \leq 1, x_2 \geq 0, x_1  \geq a x_2^{1/p}}  \mfn(x_1,x_2, y) 
    = (a,1,1). 
\end{equation}
   However, in this case,  the boundary of the domain
   $\bDab$ is not smooth at the minimizing point  $(a,1,1)$, and so instead of Proposition \ref{lem-Laplace}, we apply Proposition \ref{asymptotic-exp} to prove Theorem \ref{main_lp-b}.

\subsection{Proof of the sharp quenched estimate for $\ell_p^n$-balls}\label{subsec-proof-b}

We now prove Theorem~\ref{main_lp-b}.

\begin{proof}[Proof of Theorem~\ref{main_lp-b}]
Fix $p \in (1,\infty)$ and $a > 0$
such that $\mathbb{I}_p(a) < \infty$.
For $\theta\in \mathbb{S}$, 
  recall that the density of $\Snb$ can be expressed as in 
 \eqref{int-rep} and \eqref{gn_est-b},
and recall the assertion in \eqref{mfn-min}
that the minimum of the function $\mfn$ in \eqref{def-mfn} on $\bDab$ is attained at 
$(a,1,1)$. 
Thus, for any open neighborhood $\calU$
of $(a,1,1)$ whose closure  does not intersect the plane $y=0$,
we split the probability into two parts. Fix $\theta \in \SS$.  Then
\begin{align}\label{split-b}
\mathbb{P}_\theta\left(\Snb\in \bDab\right) = \mathbb{P}_\theta\left(\Snb\in \bDab\cap \calU\right) +\mathbb{P}_\theta\left(\Snb\in \bDab\cap \calU^c\right).
\end{align}

For the first term in~\eqref{split-b}, we have the following estimate from~\eqref{int-rep} and \eqref{gn_est-b}:
\begin{align}
  \mathbb{P}_\theta\left(\Snb\in \bDab\cap \calU\right)
&= \frac{n^2}{2\pi}\int_{\bDab\cap U} \gnt(x_1, x_2,y) e^{-nF(x_1,x_2,y)} dx_1 dx_2 dy, 
  \label{split1-b}
\end{align}
where $\gnt$ and $F$ are given in \eqref{eq-gnt} and \eqref{def-mfn}.

The bulk of the proof is devoted to the asymptotics of the Laplace type integral in \eqref{split1-b}. 
In order to apply Proposition~\ref{asymptotic-exp}, we first perform a change of variables to transform the domain of integration. Let $\mathfrak{T}:\bDab\to\R^3$ be the mapping that takes $(x_1,x_2,y)$ to $(\mathcal{X},\mathcal{Y},\mathcal{Z})$ such that
\begin{align}\label{change-calxyz}
\mathcal{X}=x_1y-ax_2^{1/p},\quad\mathcal{Y}=1-y,\quad\mathcal{Z}=x_2-1.
\end{align}
Note that the transformation $\mathfrak{T}$ is invertible in a neighborhood of $(a,1,1)$,  the Jacobian of this transformation  at $(a,1,1)$ is $1$, the image of $\bDab$ under this transformation is 
\[
 \widetilde{\mathscr{D}}_{a}:=\left\{ (\mathcal{X},\mathcal{Y},\mathcal{Z})\in\mathbb{R}^3:  0 < \mathcal{Y}\ <  1, \mathcal{Z}>-1,\mathcal{X}>0 \right\}, 
\]
and $\mathfrak{T}$ maps the minimizer $(a,1,1)$ of $\mfn$ to $(0,0,0)$. Hence, under the transformation $\mathfrak{T}$, setting $\tilde{\calU}:=\mathfrak{T}(\calU)$, we rewrite~\eqref{split1-b} as
\begin{align}
  \mathbb{P}_\theta\left(\Snb\in \bDab\cap \calU\right)
&= \frac{n^2}{2\pi}\int_{\widetilde{\mathscr{D}}_{a}\cap \widetilde{\calU}}  \gnt\circ\mathfrak{T}^{-1}(\mathcal{X},\mathcal{Y},\mathcal{Z}) e^{-nF\circ\mathfrak{T}^{-1}(\mathcal{X},\mathcal{Y},\mathcal{Z})} d\mathcal{X}d\mathcal{Y}d\mathcal{Z}.
  \label{split1-b2}
\end{align}

Let $\cnew_{ijk}:=\partial^i_1\partial^j_2\partial^k_3F(a,1,1)$.  Then, from~\eqref{change-calxyz}, we have
\begin{align*}
\frac{\partial F \circ \mathfrak{T}^{-1}}{\partial \mathcal{X}}(0,0,0) &= \left.\cnew_{100}\frac{\partial x_1}{\partial \mathcal{X}}+\cnew_{010}\frac{\partial x_2}{\partial \mathcal{X}}+\cnew_{001}\frac{\partial y}{\partial \mathcal{X}}\right |_{(0,0,0)}\\
&= \left.\cnew_{100}\frac{1}{1-\mathcal{Y}}\right |_{(0,0,0)}\\
&=\cnew_{100};\\
\frac{\partial F \circ \mathfrak{T}^{-1}}{\partial \mathcal{Y}}(0,0,0) &=\left.\cnew_{100}\frac{\mathcal{X}+a(1+\mathcal{Z})^{1/p}}{(1-\mathcal{Y})^2}-\cnew_{001}\right |_{(0,0,0)}\\
&=a\cnew_{100}-\cnew_{001}; \\
\frac{\partial F \circ \mathfrak{T}^{-1}}{\partial \mathcal{Z}}(0,0,0) &= \left.\cnew_{100}\frac{a}{p}\frac{(1+\mathcal{Z})^{(1-p)/p}}{1-\mathcal{Y}}+\cnew_{010}\right |_{(0,0,0)}\\
&=0; \\
 \frac{\partial^2 F  \circ \mathfrak{T}^{-1}}{\partial \mathcal{Z}^2}(0,0,0) &=  \left(\cnew_{200}\frac{a}{p}\frac{(1+\mathcal{Z})^{(1-p)/p}}{1-\mathcal{Y}}+\cnew_{110} \right)\frac{a}{p}\frac{(1+\mathcal{Z})^{(1-p)/p}}{1-\mathcal{Y}}+\cnew_{100}\left(-\frac{a(p-1)}{p^2}\right)\frac{(1+\mathcal{Z})^{(1-2p)/p}}{1-\mathcal{Y}}\\
 &\quad\left.+\cnew_{110}\frac{a}{p}\frac{(1+\mathcal{Z})^{(1-p)/p}}{1-\mathcal{Y}}+\cnew_{020}\right |_{(0,0,0)}\\
 &=\frac{a^2}{p^2}\cnew_{200}+\frac{2a}{p}\cnew_{110}-\frac{a(p-1)}{p^2}\cnew_{100}+\cnew_{020}.
 \end{align*}
Combining~\eqref{def-mfn} with the duality relations \eqref{duality1}-\eqref{duality2},
imply the following identities:  
\begin{align*}
\cnew_{100}& = \partial_{x_1}\Psi^*_p(a^*)=\lambda_{a,1},\\
\cnew_{001}&= -1,\\
\cnew_{110}&=\partial^2_{x_1,x_2}\Psi^*_p(a^*)=(\Hess_a)^{-1}_{12},\\
\cnew_{020}&=\partial^2_{x_2,x_2}\Psi^*_p(a^*)=(\Hess_a)^{-1}_{22},\\
\cnew_{200}&=\partial^2_{x_1,x_1}\Psi^*_p(a^*)=(\Hess_a)^{-1}_{11}.\\
\end{align*}

To apply Proposition \ref{asymptotic-exp}, we first prove the following: \\
\noindent{\bf Claim}: For $\alpha\in(1/2,1)$, the integrand $ \gnt\circ\mathfrak{T}^{-1}(\mathcal{X},\mathcal{Y},\mathcal{Z}) e^{-nF\circ\mathfrak{T}^{-1}(\mathcal{X},\mathcal{Y},\mathcal{Z})}$ in \eqref{split1-b2} admits a $(F\circ\mathfrak{T}^{-1},  (0,0,0), \alpha,  \gnt\circ\mathfrak{T}^{-1})$-representation on the bounded region $\widetilde{\mathscr{D}}_{a}\cap \widetilde{\calU}$.

\noindent {\bf Proof of Claim}:  Let $U$ be  a
neighborhood of $a^* = (a,1)$.  Then  it follows from 
Lemma \ref{lem-gn-rep} that the function $x \mapsto \bar{g}^n_\theta(x)e^{-n \Psi_p^*(x)}$ admits a
$(\Psi_p^*,(a,1), \alpha, \bar{g}_\theta^n)$-representation on  $\bDa \cap U$, and thus
properties (1) and (2) of Definition \ref{def-int_form} hold with $f = \Psi_p^*$,
$g^n = \bar{g}_\theta^n$ and $x^* = (a,1)$.
It is easy to see that this  implies that the corresponding properties also hold for the
representation stated in the claim. 
Indeed,  since 
$F(x,y) = \Psi_p^*(x) - \ln y$ by \eqref{def-mfn}, 
$\bDab = \{(x,y): x \in \bDa, 1 \geq y \geq 0\}$  by \eqref{def-doma} and \eqref{def-doma-b}, and 
$\bar{\calU}$ does not
 intersect the hyperplane $\{y = 0\}$,
this implies that $F$ is twice continuously differentiable on
$\bDab \cap \calU$ with unique minimizer on the closure of $\bDab \cap \calU$ achieved at $(a, 1, 1)$.
On the other hand,   Lemma \ref{lem-gn-rep} and property (2) of Definition \ref{def-int_form}
imply that $\bar{g}_\theta^n$ is continuously differentiable and there exists $C < \infty$ such that 
$\ln \bar{g}_\theta^n(x) \leq C n^\alpha ||x||_2$ for all $x$ in a neighborhood $U'$ of $(a,1)$.
By \eqref{eq-gnt},  it follows that on any neighborhood $\calU^\prime$ of $(a,1,1)$  of the form
$\{ (x,y): x \in U', 1 >  y > \varepsilon\}$ for some $\varepsilon > 0$,
$\gnt$ is twice continuously differentiable and satisfies
$\ln \gnt (x_1,x_2,y) \leq \frac{C}{\varepsilon} n^\alpha ||(x_1,x_2,y)||_2.$
Combining these two properties with the fact that 
$\mathfrak{T}$ is infinitely differentiable,  maps $(a,1,1)$ to $(0,0,0)$ and $\bDab \cap \calU$ to $\widetilde{\mathscr{D}}_{a}\cap \widetilde{\calU}$, and  has an infinitely differentiable inverse on $\widetilde{\mathscr{D}}_{a}\cap \widetilde{\calU}$,
it follows that properties (1) and (2) of Definition \ref{def-int_form} hold with
$f = F\circ\mathfrak{T}^{-1}$,
 $D = \widetilde{\mathscr{D}}_a \cap \widetilde{\calU}$, 
   $g^n = \gnt \circ \mathfrak{T}^{-1}$ and $x^* = (0,0,0)$.
 This proves the claim.  \hfill \qedsymbol

 Note that  $\widetilde{\mathscr{D}}_{a}  \subset \R_+ \times \R^2$.  Thus, 
we apply Proposition~\ref{asymptotic-exp} with $m = 1$ and $d= 2$ 
to the transformed integral~\eqref{split1-b2}  with $g^n$, $f$, $x^*$ and $D$ therein replaced with $\gnt\circ\mathfrak{T}^{-1}$, $F\circ\mathfrak{T}^{-1}$, $(0,0,0)$ and $\widetilde{\mathscr{D}}_{a}$, to obtain 
\[ \mathbb{P}_\theta\left(\Snb\in \bDab\cap U\right) 
= \frac{n^2}{2\pi}\times \frac{\sqrt{2\pi}}{n^{5/2}}
\frac{\gnt\circ\mathfrak{T}^{-1}(0,0,0)e^{-nF\circ\mathfrak{T}^{-1}(0,0,0)}}{\partial_{\mathcal{X}} F\circ\mathfrak{T}^{-1}(0,0,0)\partial_{\mathcal{Y}} F\circ\mathfrak{T}^{-1}(0,0,0) \sqrt{|\partial^2_{\mathcal{Z}} F\circ\mathfrak{T}^{-1}(0,0,0)|}} (1+o(1)). \]
The  expression on the right-hand side can be  simplified further 
using first the relations $\mathfrak{T}^{-1}(0,0,0) = (a,1,1)$, $F(a,1,1) = \Psi_p^*(a^*)$, $\gnt(a,1,1) = \bar{g}_\theta^n(a^*)$, which follow from \eqref{eq-gnt} and \eqref{def-mfn}, 
together with  the expressions for the partial derivatives of $F$ 
 calculated above, to obtain 
\[ 
\mathbb{P}_\theta\left(\Snb\in \bDab\cap U\right) 
= \frac{1}{\sqrt{2\pi n} \tilde{\gamma}_a} \gn(a^*)e^{-n\Psi^*_p(a^*)}(1+o(1)),
\]
where
\[ \tilde{\gamma}_a := \lambda_{a,1} (a \lambda_{a,1} + 1)\sqrt{\left|-\frac{a(p-1)}{p^2}\lambda_{a,1}+\frac{2a}{p}(\Hess_a)^{-1}_{12}+(\Hess_a)^{-1}_{22}+\frac{a^2}{p^2}(\Hess_a)^{-1}_{11}\right|}. 
  \]
Substituting for $\gn$ and $\Psi^*_p$ using the relations  \eqref{def-gn},   \eqref{Cna} and \eqref{rate-fn},   we then obtain 
  \begin{align} \mathbb{P}_\theta\left(\Snb\in \bDab\cap U\right) 
&= \frac{C^n_a(\theta^n)}{\gamma_a\sqrt{2\pi n}}e^{-n\mathbb{I}_{p}(a)  + \sqrt{n} R^n_{a}(\theta^n)}(1+o(1)),\label{split11-b}
\end{align}  
  where $\gamma_a = ({\rm det} {\mathcal H}_a) \tilde{\gamma}_a$, which coincides with the definition given
  in \eqref{def-gammapa}. 

For the second term in~\eqref{split-b}, as in the proof of $\ell_p^n$ spheres in~\eqref{split2},
one can invoke the
quenched large deviation principle for $\Snb$ established in \cite[Proposition 5.3]{GanKimRam17} along with the fact that the rate function has a unique minimum, as proved in Lemma~\ref{inf_ratefunc} to 
show that it is negligible with respect to~\eqref{split11-b}.
When combined  with  \eqref{eq-reform-b}, \eqref{split-b} and \eqref{split11-b}, this  yields~\eqref{tail_prob-b}.
\end{proof}

\section{The joint density estimate}
\label{sec-pfdens}

This section is devoted to the proof of the density estimate
stated in Proposition \ref{prop-mainest}.
As usual,  throughout
fix $p \in (1,\infty)$.   In Section \ref{subs-asymptotics} 
 an identity  for the joint density is established in
terms of an integral. This integral is then  shown in Section \ref{subs-com}
to admit an alternative
representation  as  an expectation 
with respect to a tilted measure.
The latter representation is used in Section \ref{subs-integrand} 
to obtain certain asymptotic estimates.
These results are finally combined in Section \ref{sec-proppf} to
prove Proposition \ref{prop-mainest}.

\subsection{An integral representation for the joint density}
\label{subs-asymptotics}

\begin{lemma}[Representation for the density of $\bar{S}^n$ under $\mathbb{P}_\theta$] \label{g_est}
Fix $n \in \N$ and $\theta \in \SS$, and  recall the definitions of $\Psi_p$, $\mathbb{J}_p$, $\lambda_x$, $\Phi_p$ and $\Psi^{n}_{p,\theta}$ in \eqref{psi}, \eqref{def-Jp}, \eqref{lambda_x}, \eqref{phi_p} and \eqref{psin_p}, respectively,
    and recall    that $\bar{h}^n_\theta$ is the density,  under $\PP_{\theta}$, of $\Sn$ 
    defined in~\eqref{joint_pdf}. Then for all sufficiently large $n$,  and $x\in\mathbb{J}_p$, the following identity holds, 
\begin{align}\label{hn}
  \bar{h}^n_\theta(x) = \left(\frac{n}{2\pi}\right)^2e^{-n\Psi^*_p(x)}e^{n(\Psi^{n}_{p,\theta}(\lambda_x)-\Psi_p(\lambda_x))}
\mathcal{I}^n_\theta(x),
\end{align}
where 
\begin{align}\label{In}
\mathcal{I}^n_\theta(x):=\int_{\mathbb{R}^2} e^{-i\langle t,nx \rangle}\prod_{j=1}^n\frac{\Phi_p(\sqrt{n}\theta^n_j(\lambda_{x,1}+it_1),\lambda_{x,2}+it_2)}{\Phi_p(\sqrt{n}\theta^n_j\lambda_{x,1},\lambda_{x,2})}dt.
\end{align}
Moreover, there exists $s>1$ such that $(t_1,t_2)\mapsto(\prod_{j=1}^n\Phi_p(\sqrt{n}\theta^n_j(\lambda_{x,1}+it_1),\lambda_{x,2}+it_2))^{s/n}$ lies in $\mathbb{L}_1(\R^2)$ for all sufficiently large $n$.
\end{lemma}
\begin{proof}
Let $\mathbb{D}_p$ be as in \eqref{dom-Lambda}, fix $x\in\mathbb{J}_p$ 
  and omit the subscript $x$ from $\lambda_x\in\mathbb{D}_p\subset\R^2$ and the superscript $p$ from many quantities for notational simplicity. 
Recall the definition of $\bV^n_{j}$ in~\eqref{Znj} and
for $\theta\in\mathbb{S}$,
  let 
$\bln_\theta$ be the density of the sum $\sum_{j=1}^n\bV^n_{j}$ under $\PP_{\theta}$.
The moment generating function of this sum is given by
\begin{align*}
\int_{\R^2}e^{\langle \lambda,y\rangle}\bln_\theta(y)dy&= \mathbb{E}_\theta\left[e^{\langle \lambda, \sum_{j=1}^n\bV^n_j\rangle}\right]\\
&=\prod_{j=1}^n\mathbb{E}_\theta\left [e^{\lambda_1\sqrt{n}\theta^n_jY_j+\lambda_2\abs{Y_j}^p} \right]<\infty,
\end{align*}
where $Y_1,\ldots,Y_n$ are i.i.d. with density $f_p$ defined in~\eqref{pNormal} and the finiteness follows because $\lambda\in\mathbb{D}_p$ and thus $\lambda_2<1/p$.
Then the Fourier transform of
the integrable function $y \mapsto e^{\langle \lambda,y\rangle}\bln_\theta(y)$ is given as follows\footnote{Note that we use the convention for characteristic functions and thus put $i$ in place of $-2\pi i$ in the Fourier transform.} : for  $t\in\mathbb{R}^2$,
\begin{align}
\int_{\mathbb{R}^2} e^{\langle \lambda+it,y\rangle}\bln_\theta(y)dy&=\mathbb{E}_{\theta}\left[e^{\langle\lambda+it,\sum_{j=1}^n\bV^n_{j}\rangle} \right]\nonumber\\
& = \prod_{j=1}^n\mathbb{E}_{\theta}\left[e^{\langle\lambda+it,\bV^n_{j}\rangle} \right]\nonumber\\
& = \prod_{j=1}^n\Phi_p(\sqrt{n}\theta^n_j(\lambda_1+it_1),\lambda_2+it_2).\label{prod_Z}
\end{align}

We now make the following claim:\\
{\bf Claim.}
{\it  There exists $s>1$ such that for any $\lambda\in \mathbb{D}_p$, $t\in\R^2$ and  $j,k\in\{1,\ldots,n\}$, $j\neq k$, we have 
\begin{equation}\label{integrability}
K^{n,j,k}_\theta(\lambda,t):=\int_{\mathbb{R}^2}\abs{\Phi_p(\sqrt{n}\theta^n_k(\lambda_1+it_1),\lambda_2+it_2)\Phi_p(\sqrt{n}\theta^n_j(\lambda_1+it_1),\lambda_2+it_2)}^sdt<\infty.
\end{equation}}

We defer the proof of the claim, first showing how the lemma follows from the claim. 
  Let $s > 1$ be as in the claim. 
  Since the moment generating function $\Phi_p$ is bounded, the claim holds for any $s'>s$. Now,  pick any integer
  $n > 2s$. Then  H\"older's inequality and the claim imply that the right-hand side of~\eqref{prod_Z} lies in $\mathbb{L}_1(\mathbb{R}^2)$.  Hence, the second assertion of the lemma holds for any such  $n$. We may then apply the inverse Fourier transform formula to conclude that, for all sufficiently large $n$,
\begin{align}\label{f_est}
\bln_\theta(x) = \left(\frac{1}{2\pi}\right)^2\int_{\mathbb{R}^2} e^{-\langle \lambda+it,x\rangle}\prod_{j=1}^n\Phi_p(\sqrt{n}\theta^n_j(\lambda_1+it_1),\lambda_2+it_2)dt.
\end{align}

Next, recall that for any $x\in\mathbb{J}_p$, $\lambda = \lambda_x$ is chosen so that~\eqref{lambda_x} is satisfied.
Also,  by~\eqref{joint_pdf} and~\eqref{Znj}, we have 
\[
\Sn = \frac{1}{n}\sum_{j=1}^n\bV^n_{j}. 
\]
Hence, using \eqref{f_est},~\eqref{lambda_x} and~\eqref{psin_p}, we see that the density $\bar{h}^n_\theta$ of $\Sn$
under $\PP_\theta$ is given by 
\begin{align*}
\bar{h}^n_\theta(x)& = n^2\bln_\theta(nx)\\
 &= \left(\frac{n}{2\pi}\right)^2\int_{\mathbb{R}^2} e^{-\langle \lambda+it,nx\rangle}\prod_{j=1}^n\Phi_p(\sqrt{n}\theta^n_j(\lambda_1+it_1),\lambda_2+it_2)dt\\
&= \left(\frac{n}{2\pi}\right)^2e^{-n\Psi^*_p(x)}\int_{\mathbb{R}^2} e^{n(\Psi^*_p(x)-\langle\lambda,x\rangle)}e^{-i\langle t,nx \rangle}\prod_{j=1}^n\Phi_p(\sqrt{n}\theta^n_j(\lambda_1+it_1),\lambda_2+it_2)dt\\
& = \left(\frac{n}{2\pi}\right)^2e^{-n\Psi^*_p(x)}\int_{\mathbb{R}^2} e^{-n\Psi_p(\lambda)}e^{-i\langle t,nx \rangle}\prod_{j=1}^n\Phi_p(\sqrt{n}\theta^n_j(\lambda_1+it_1),\lambda_2+it_2)dt\\
& =\left(\frac{n}{2\pi}\right)^2e^{-n\Psi^*_p(x)}e^{n(\Psi^n_{p,\theta}(\lambda)-\Psi_p(\lambda))}\int_{\mathbb{R}^2} e^{-i\langle t,nx \rangle}\prod_{j=1}^n\frac{\Phi_p(\sqrt{n}\theta^n_j(\lambda_1+it_1),\lambda_2+it_2)}{\Phi_p(\sqrt{n}\theta^n_j\lambda_1,\lambda_2)}dt, 
\end{align*}
for $x\in\mathbb{J}_p$. 
Since the right-hand side above coincides with the expression for $\bar{h}^n_\theta$ given in~\eqref{hn} and~\eqref{In}, this proves the first part of the lemma given the claim. \\

To complete the proof of the lemma, it only remains to prove the claim.

\noindent
{\bf Proof of the claim.}
Fix $n\in\N$, $j,k\in\{1,\ldots,n\}$, $j\neq k$, and 
set $\theta_1:=\theta^n_j$ and $\theta_2:=\theta^n_k$.
Let $\bar{\upsilon} := \bar{\upsilon}^{n,j,k}_\theta$  
  denote the density of $\bV^n_j+\bV^n_k$ under $\PP_\theta$.  
We assert that to prove the claim
 it suffices to show that the function $\R^2\ni z\mapsto e^{\langle \lambda, z\rangle}\bar{\upsilon}(z)$
 lies in $\mathbb{L}_{1+r}(\mathbb{R}^2)$ for some $r \in (0,\infty)$.  Indeed, then 
 by the Hausdorff-Young inequality~\cite[Theorem 8.21]{Folland99}, the Fourier transform
of $z\mapsto e^{\langle \lambda, z\rangle}\bar{\upsilon}(z)$  lies in $\mathbb{L}_{s}$, where $s$ is the  ``conjugate exponent'' of
$1+r$. By \eqref{Znj} and \eqref{phi_p}, this is  equivalent to saying that 
\eqref{integrability} holds with $s = 1+1/r > 0$. 

To this end, we start by obtaining  a convenient expression for $\bar{\upsilon}$. Note from~\eqref{Znj} that $\bV^n_{j}+\bV^n_k=T(Y_j,Y_k)$, where $(Y_j)_{j\in\N}$ are i.i.d. with common density $f_p$ and $T:=T^{n,j,k}:\R^2\to\R\times \R_+$ is the differentiable transformation given by, 
\[
T(y_1,y_2) = (\sqrt{n}(\theta_1y_1+\theta_2y_2),\abs{y_1}^p+\abs{y_2}^p), \quad (y_1,y_2)\in\R^2.
\]
Given $(z_1,z_2)\in\R\times\R_+$, we solve for $(z_1,z_2)=T(y_1,y_2)$. For $z_2>0$, consider the curves $\{y\in\R^2:z_1=\sqrt{n}(\theta_1y_1+\theta_2y_2)\}$ and $\{y\in\R^2:z_2=\abs{y_1}^p+\abs{y_2}^p\}$, which describe a line and an $\ell^2_p$ sphere, respectively. Using Lagrange multipliers, it is straightforward to deduce that, the equation $(z_1,z_2)=T(y_1,y_2)$, which describes the intersection of these two curves, then has two solutions when $\abs{z_1}< z_2^{1/p}\sqrt{n}(\abs{\theta_1}^{p/(p-1)}+\abs{\theta_2}^{p/(p-1)})^{(p-1)/p}=:M(z_2)$, one solution when $|z_1|=M(z_2)$ and no $y\in\R^2$ such that $T(y)=(z_1,z_2)$ and when  $\abs{z_1}> M(z_2)$.

For $\abs{z_1}<M(z_2)$, we define $y^+$ and $y^-$ to be the two solutions to $T(y)=z$. Thus, $T$ is locally invertible on its range and hence, by the change of variables formula and the differentiability of $T$, we may write the density $\bar{\upsilon}$ as
\begin{align*}
\bar{\upsilon}(z_1,z_2) &= \left(f_p(y^+_1)f_p(y^+_2)\abs{\frac{\partial(y^+_1,y^+_2)}{\partial(z_1,z_2)}}+f_p(y^-_1)f_p(y^-_2)\abs{\frac{\partial(y^-_1,y^-_2)}{\partial(z_1,z_2)}}\right)1_{\{z_2>0,\abs{z_1}<M(z_2)\}}\\
&=\left(\abs{\frac{\partial(y^+_1,y^+_2)}{\partial(z_1,z_2)}}+\abs{\frac{\partial(y^-_1,y^-_2)}{\partial(z_1,z_2)}}\right)e^{-z_2/p}1_{\{z_2>0,\{\abs{z_1}<M(z_2)\}}.
\end{align*}
Here, $|\partial(y_1,y_2)/\partial(z_1,z_2)|$ is the Jacobian of the transformation $T$ at $(y_1,y_2)$, which is given by the explicit formula
\begin{align}\label{jacob}
\mathcal{J}_T(y):=\abs{\frac{\partial(y_1,y_2)}{\partial(z_1,z_2)}}=\frac{1}{\sqrt{n}p}\abs{\frac{1}{\theta_2\text{sgn}(y_1)\abs{y_1}^{p-1}-\theta_1\text{sgn}(y_2)\abs{y_2}^{p-1}}},
\end{align}
where $\text{sgn}(\cdot)$ denotes the sign function.

For $r >0$, the above discussion shows that
\begin{align*}
&\int_{\R^2} \abs{e^{\lambda_1z_1+\lambda_2z_2}\bar{\upsilon}(z_1,z_2)}^{1+r}dz_1dz_2 \\
&\qquad =
\int_{\R^2} \abs{e^{\lambda_1z_1+\lambda_2z_2-z_2/p}}^{1+r}\left(\abs{\frac{\partial(y^+_1,y^+_2)}{\partial(z_1,z_2)}}+\abs{\frac{\partial(y^-_1,y^-_2)}{\partial(z_1,z_2)}} \right)^{1+r}1_{\{z_2>0,\abs{z_1}<M(z_2)\}}dz_1dz_2\\
&\qquad \leq 2^{r}\int_{\R^2} \abs{e^{\lambda_1z_1+\lambda_2z_2-z_2/p}}^{1+r}\left(\abs{\frac{\partial(y^+_1,y^+_2)}{\partial(z_1,z_2)}}^{1+r}+\abs{\frac{\partial(y^-_1,y^-_2)}{\partial(z_1,z_2)}}^{1+r} \right)1_{\{z_2>0,\abs{z_1}<M(z_2)\}}dz_1dz_2\\
&\qquad =2^{r}\int_{\R^2} \abs{e^{\lambda_1\sqrt{n}(\theta_1y_1+\theta_2y_2)+\left(\lambda_2-\frac{1}{p}\right)\left(\abs{y_1}^p+\abs{y_2}^p)\right)}}^{1+r}\abs{\frac{\partial(y_1,y_2)}{\partial(z_1,z_2)}}^{r}dy_1dy_2,
\end{align*}
where the inequality follows from $(a+b)^{1+r}\leq 2^r(a^{1+r}+b^{1+r})$ for $a,b\in\R_+$, and the last equality uses the definition of $T$.
Next, let $\mathcal{N}\subset\R^2$ be a neighborhood of the origin. Then
\begin{align*}
&\int_{\R^2} \abs{e^{\lambda_1z_1+\lambda_2z_2}\bar{\upsilon}(z_1,z_2)}^{1+r}dz_1dz_2 \\
&\qquad =
2^r\int_{\R^2\cap \mathcal{N}} \abs{e^{\lambda_1\sqrt{n}(\theta_1y_1+\theta_2y_2)+\left(\lambda_2-\frac{1}{p}\right)\left(\abs{y_1}^p+\abs{y_2}^p)\right)}}^{1+r}\abs{\frac{\partial(y_1,y_2)}{\partial(z_1,z_2)}}^{r}dy_1dy_2\\
&\quad\qquad +
2^r\int_{\R^2\cap \mathcal{N}^c}\abs{e^{\lambda_1\sqrt{n}(\theta_1y_1+\theta_2y_2)+\left(\lambda_2-\frac{1}{p}\right)\left(\abs{y_1}^p+\abs{y_2}^p)\right)}}^{1+r}\abs{\frac{\partial(y_1,y_2)}{\partial(z_1,z_2)}}^{r}dy_1dy_2.
\end{align*}
Since $p \in (1,\infty)$ and $x \in {\mathbb J}_p$ implies $p\lambda_2=p\lambda_{x,2}<1$, it follows that $e^{\lambda_1\sqrt{n}(\theta_1y_1+\theta_2y_2)+\left(\lambda_2-\frac{1}{p}\right)\left(\abs{y_1}^p+\abs{y_2}^p)\right)}$ lies in $\mathbb{L}_r(\R^2)$ for any $r>0$. 
Moreover, since $p>1$, by \eqref{jacob}, there exists $r_1>0$ small enough such that the Jacobian $\mathcal{J}_T$ lies in $\mathbb{L}_{r_1}(\mathcal{N})$. On the other hand, there exists $0<r_2<\infty$ large enough such that the Jacobian $\mathcal{J}_T$ lies in $\mathbb{L}_{r_2}(\mathcal{N}^c)$.
Thus, by H\"older's inequality, there exists $r >0$ such that the last display is finite.  This completes the proof of the claim, and therefore of the lemma.

\end{proof}

\subsection{Representation of the integrand in terms of a tilted measure}
\label{subs-com}

We next obtain a representation for the integrand of the integral $\mathcal{I}^n_\theta$ in~\eqref{In} using a change of measure.   
Once again, from Section~\ref{subs-reform}, recall the i.i.d.\ sequence of random variables
  $(Y_j)_{j \in \N}$ defined on $(\Omega,\mathcal{F},\mathbb{P})$ that have density $f_p$ and are independent of
  $\Theta=(\Theta^n)_{n\in\N}$.
Fix $a>0$ such that $\mathbb{I}_p(a)<\infty$, recall the definition of $\lambda=\lambda_a$ from~\eqref{lambda_x}. Fix $n \in \N$, and consider a ``tilted'' measure $\widetilde{\mathbb{P}}^n=\widetilde{\mathbb{P}}^{n,a}$
on $(\Omega,\mathcal{F})$ such that the (marginal) distribution of $\Theta^n$
remains
unchanged but conditioned on $\Theta = \theta \in \SS$,   
$\{Y^n_j, j = 1, \ldots, n\}$ are still  independent, but not
identically distributed, with   $Y_j^n$ having  
   density  $\newdens^{n}_j=\newdens^{n,a}_{\theta,j}$ given by
\begin{equation}\label{tilt_density}
\newdens^n_j(y) :=  \exp\left(\left\langle \lambda_a,\left(\sqrt{n}\theta^n_jy,\abs{y}^p\right)\right\rangle-\Lambda_p\left(\sqrt{n}\theta^n_i\lambda_{1},\lambda_{2}\right)\right) f_p(y), 
    \quad y \in \R, 
\end{equation}
with $\Lambda_p$ as defined in~\eqref{logmgf_lp} and as before we omit the explicit dependence and other quantities of $\newdens^n_j$ on $p$ and $a$. For $\theta\in\SS$, denote by $\widetilde{\mathbb{P}}^n_\theta$ and $\widetilde{\mathbb{E}}^n_{\theta}$ the probability and the expectation taken with respect to $\widetilde{\mathbb{P}}^n$, conditioned on $\theta$, and likewise, let $\widetilde{\mathrm{Var}}^n_\theta(\cdot)$ and $\widetilde{\mathrm{Cov}}^n_\theta(\cdot,\cdot)$ denote the conditional variance and conditional covariance, respectively, under $\widetilde{\mathbb{P}}^n_\theta$.

 Recall from~\eqref{logmgf_lp} and~\eqref{phi_p} that $\Lambda_p(t)=\log \Phi_p(t)$ for $t\in\mathbb{R}^2$. 
Then,  by \eqref{Znj}, \eqref{phi_p} and \eqref{tilt_density}, it follows that for $j=1, \ldots, n$ and $\beta = (\beta_1, \beta_2) \in \R^2$, 
\begin{equation}
  \label{tilt_mgf}
  \begin{array}{rcl}
\widetilde{\mathbb{E}}^n_{\theta} \left[ e^{\langle \beta, \bV^n_{j} \rangle} \right]
&= & \frac{\Phi_p(\sqrt{n}\theta_j^n (\beta_{1} + \lambda_1), \beta_{2} + \lambda_2)}{\Phi_p(\sqrt{n}\theta_j^n \lambda_1, \lambda_2)},
\end{array}
\end{equation}
and hence, 
\begin{equation}
\label{tilt_mean}
\widetilde{\mathbb{E}}^n_{\theta}\left[\bV^n_{j}\right] = \nabla_{\beta} \left. 
\widetilde{\mathbb{E}}^n_{\theta} \left[ e^{\langle \beta, \bV^n_{j} \rangle} \right]\right|_{\beta =(0,0)}
= \nabla\log \Phi_p(\sqrt{n}\theta_j^n \lambda_1, \lambda_2).
 \end{equation}
Denoting $\bV^n_j=\left(\bV^n_{j,1},\bV^n_{j,2}\right)$, by~\eqref{tilt_mgf}, we also have for $k,l=1,2$,
\begin{align}
\widetilde{\mathrm{Cov}}^n_\theta\left(\bV^n_{j,k},\bV^n_{j,l}\right)& =
\widetilde{\mathbb{E}}^n_\theta\left[\bV^n_{j,k}\bV^n_{j,l}\right]-\widetilde{\mathbb{E}}^n_\theta\left[\bV^n_{j,k}\right]\widetilde{\mathbb{E}}^n_\theta\left[\bV^n_{j,l}\right]\nonumber\\
&=\partial^2_{\beta_k,\beta_l}\left.\widetilde{\mathbb{E}}^n_{\theta} \left[ e^{\langle \beta, \bV^n_{j} \rangle} \right]\right|_{\beta =(0,0)}-
\widetilde{\mathbb{E}}^n_\theta\left[\bV^n_{j,k}\right]\widetilde{\mathbb{E}}^n_\theta\left[\bV^n_{j,l}\right]\nonumber\\
&=\partial_{k,l}^2 \log \Phi_p(\sqrt{n}\theta_j^n \lambda_1, \lambda_2). \label{tilt_cov}
\end{align} 

For $x\in\mathbb{J}_p$, define $\hV^n_{x}$ to be
\begin{equation}\label{Zbar}
\hV^n_{x}:=\frac{1}{\sqrt{n}}\sum_{j=1}^n\left(\bV^n_{j}-x\right).
\end{equation}

\begin{lemma}\label{conv_c_gamma}
For $x\in\mathbb{J}_p$ and $\theta\in\mathbb{S}$, recall the definitions of $\bar{V}^n_{j}$, $\Phi_p$, $c^n_{x}$, $\Hess^n_{x}$ and $\hV^n_x$ given in~\eqref{Znj}~\eqref{phi_p},~\eqref{def-cnhessn} and~\eqref{Zbar}. Then
\begin{align}
  c^n_{x}(\theta^n) &= \widetilde{\mathbb{E}}^n_\theta\left[\hV^n_{x} \right],\label{cn}\\
\left\langle \Hess^n_{x}\left(\theta^n\right) t, t\right\rangle  &=  \widetilde{{\rm Var}}_{\theta}^n \left(\left \langle t, \hV^n_{x} \right\rangle\right)
 , \quad \mbox{ for all } t \in \R^2.  \label{Gammanx} 
  \end{align}
Moreover, for $t=(t_1,t_2)\in\mathbb{R}^2$,
\begin{equation} \label{mu_hat}
\hat{\mu}^n_{x,\theta}(t):=\widetilde{\mathbb{E}}^n_\theta\left[e^{i\left\langle t,\sqrt{n}\hV^n_{x}\right\rangle} \right]=e^{-i\langle t,nx \rangle}\prod_{j=1}^n\frac{\Phi_p(\sqrt{n}\theta^n_j(\lambda_{x,1}+it_1),\lambda_{x,2}+it_2)}{\Phi_p(\sqrt{n}\theta^n_j\lambda_{x,1},\lambda_{x,2})}.
\end{equation}
Furthermore, for $\smeas$-a.e.\ $\theta$, as $n \rightarrow \infty$,
$\Hess^n_{x}(\theta^n)$ converges to the quantity $\Hess_{x}$ defined in~\eqref{gamma_x}. 
\end{lemma} 
\begin{proof}  
We fix $\theta \in \SS$ and $x$ in the domain $\mathbb{J}_p$ of $\Psi^*_p$ defined in~\eqref{def-Jp} and omit the subscript $x$ from $\lambda_x$ for notational simplicity. 
By~\eqref{tilt_mean}, \eqref{Zbar}, the definition of $\Psi_{p,\theta}^n$ in  \eqref{psin_p} and \eqref{gradPsi},
we have,
\begin{align}
  \widetilde{\mathbb{E}}^n_\theta\left[\hV^n_{x} \right]
  =\frac{1}{\sqrt{n}}\sum_{j=1}^n\left(-x+\widetilde{\mathbb{E}}_{\theta}^n
  [\bV_{j}^n] \right)
  &=\frac{1}{\sqrt{n}}\sum_{j=1}^n\left(-x+\nabla \log \left(\Phi_p(\sqrt{n}\theta^n_j\lambda_1,\lambda_2)\right)\right)\nonumber\\
& = \frac{1}{\sqrt{n}}\left(-nx+n\nabla\Psi^n_{p,\theta}(\lambda)\right) \nonumber\\
&=\sqrt{n}\nabla\left(\Psi^n_{p,\theta}(\lambda)-\Psi_p(\lambda)\right)\nonumber.
\end{align}
When combined with \eqref{def-cnhessn}, this proves \eqref{cn}.
Similarly, by the independence of $\bV^n_{j}, j = 1, \ldots, n,$ under $\widetilde{\mathbb{P}}^n_{\theta}$, \eqref{tilt_cov}, the definition of $\Psi^n_{p,\theta}$ in \eqref{psin_p} and the definition of $\Hess^n_x$ in~\eqref{def-cnhessn}, it follows that 
\begin{align*}
   \widetilde{{\rm Var}}_{\theta}^n \left( \left\langle t, \hV^n_{x} \right\rangle\right)
   &=  \frac{1}{n} \sum_{j=1}^n  \widetilde{{\rm Var}}_{\theta}^n \left(\left\langle t, \bV^n_{j}\right \rangle\right)
   = \langle  \Hess^n_{x}(\theta^n) t, t\rangle, 
\end{align*}
which proves \eqref{Gammanx}. 
Also,  by the definitions of $\hat{\mu}^n_{x,\theta}$ and $\hV^n_x$ in~\eqref{mu_hat} and \eqref{Zbar}, respectively,  the independence of $\bV^n_{j}, j = 1, \ldots, n,$ under $\widetilde{\mathbb{P}}^n_{\theta}$ and the relation 
 \eqref{tilt_mgf}, it follows that 
for $t\in\mathbb{R}^2$,
\begin{align*}
  \hat{\mu}^n_{x,\theta}(t)& = e^{-i\langle t,nx \rangle}\prod_{j=1}^n\widetilde{\mathbb{E}}^n_\theta\left[e^{i\langle t,\bV^n_{j}\rangle} \right]  
  = e^{-i\langle t,nx \rangle}\prod_{j=1}^n
 \frac{\Phi_p(\sqrt{n}\theta^n_j(\lambda_1+it_1),\lambda_2+it_2)}{\Phi_p(\sqrt{n}\theta^n_j\lambda_1,\lambda_2)}, 
\end{align*}
which proves \eqref{mu_hat}.

    It only remains to 
establish the  convergence stated in the last assertion of the lemma. By~\eqref{def-cnhessn} and \eqref{psin_p}, it follows that for each $i,j = 1,2$, there exists $\alpha, \beta\in\N$ such that the entry $(\Hess^n_{x}(\theta^n) )_{ij}$ can be written as
\begin{align*}
\left(\Hess^n_{x}(\theta^n) \right)_{ij}&= \frac{1}{n}\sum_{j=1}^n(\sqrt{n}\theta^n_j)^\alpha\partial_1^{\alpha}\partial_2^{\beta}\log\Phi_p(\sqrt{n}\theta^n_j\lambda_1,\lambda_2)\\
& = \int_\mathbb{R}u^{\alpha}\partial_1^{\alpha}\partial_2^{\beta}\log\Phi_p(u\lambda_1,\lambda_2)\empn(du).
\end{align*}
Since, the moment generating function $\Phi_p$ is infinitely differentiable, the mapping $u\mapsto\phi(u):=u^{\alpha}\partial_1^{\alpha}\partial_2^{\beta}\log\Phi_p(u\lambda_1,\lambda_2)$ is continuous. Moreover, $\phi$ has polynomial growth by Lemma~\ref{growth_lmgf}.
Since Lemma~\ref{slln} implies that $\mathcal{W}_p\left(\empn,\normal\right)\to 0$ as $n\to\infty$, it follows that
\begin{align*}
\left(\Hess^n_{x}(\theta^n) \right)_{ij}\to\int_\mathbb{R}u^{\alpha}\partial_1^{\alpha}\partial_2^{\beta}\log\Phi_p(u\lambda_1,\lambda_2)\gamma_2(du)=\left(\Hess_x\right)_{ij},
\end{align*}
where from Lemma \ref{wass-conv}(2) that, as $n$ tends to infinity, the last equality follows by the definition of $\Hess_x$ in~\eqref{gamma_x}.
\end{proof}

\subsection{Estimates of the integrand}
\label{subs-integrand}
\begin{lemma}\label{moment-est}
Fix $x\in\mathbb{J}_p$. Recall the definitions of $\hV^n_{x}$ and
$(\bar{V}^n_{j})_{j=1,\ldots,n}$ given in \eqref{Zbar} and \eqref{Znj}, respectively. There exist constants $\widetilde{C}<\infty$ such that for all $n\in\N$ and for $\sigma$-a.e.\ $\theta$,
\begin{align}\label{moment-bnd}
\frac{1}{n}\sum_{j=1}^n\widetilde{\mathbb{E}}^n_\theta\left[ \norm{\bar{V}^n_{j}-\widetilde{\mathbb{E}}^n_\theta[\bar{V}^n_{j}]}^3\right] < \widetilde{C}, \quad \frac{1}{n}\sum_{j=1}^n\widetilde{\mathbb{E}}^n_\theta\left[ \norm{\bar{V}^n_{j}-\widetilde{\mathbb{E}}^n_\theta[\bar{V}^n_{j}]}^4\right] < \widetilde{C},
\end{align}
and for all $n\in\N$,
\begin{align}\label{moment-bnd2}
\widetilde{\mathbb{E}}^n_\theta\left[ \norm{\hV^n_{x} - \widetilde{\mathbb{E}}_{\theta}^n[\hV^n_{x}}^3\right]
< \widetilde{C}.
\end{align}
\end{lemma}

\begin{proof}
Due to the following standard inequalities,
$(a^2+b^2)^{3/2}\leq C'(\abs{a}^3+\abs{b}^3)$ and 
$ \frac{1}{n}\sum_{j=1}^n|a_j|^3 \leq (\frac{1}{n}\sum_{j=1}^n\abs{a_j}^4)^{3/4}$, 
to show \eqref{moment-bnd} it suffices to show the boundedness of 
\[
\frac{1}{n}\sum_{j=1}^n\widetilde{\mathbb{E}}^n_\theta\left[\left(\bar{V}^n_{j,1}-\widetilde{\mathbb{E}}^n_\theta\left[\bar{V}^n_{j,1}\right] \right)^4\right]
\quad \text{and} \quad
 \frac{1}{n}\sum_{j=1}^n\widetilde{\mathbb{E}}^n_\theta\left[\left(\bar{V}^n_{j,2}-\widetilde{\mathbb{E}}^n_\theta\left[\bar{V}^n_{j,2}\right] \right)^4\right].
\]
We show boundedness of just the first term; boundedness of the second can be shown analogously.  Using following relation between cumulants and central moments, by simple calculation we have
\begin{align}
&\frac{1}{n}\sum_{j=1}^n\widetilde{\mathbb{E}}^n_\theta\left[\left(\bar{V}^n_{j,1}-\widetilde{\mathbb{E}}^n_\theta\left[\bar{V}^n_{j,1}\right] \right)^4\right] \nonumber\\
&\quad= \frac{3}{n}\sum_{j=1}^n\widetilde{\mathbb{E}}^n_\theta\left[\left(\bar{V}^n_{j,1}-\widetilde{\mathbb{E}}^n_\theta\left[\bar{V}^n_{j,1}\right] \right)^2\right]+\int_\mathbb{R}\partial^4_1(\log\Phi_p(u\lambda_{x,1},\lambda_{x,2}))\empn(du)\nonumber\\
&\quad = 3\widetilde{\rm Var}^n_\theta\left (\hV^n_{x,1}\right)+\int_\mathbb{R}\partial^4_1(\log\Phi_p(u\lambda_{x,1},\lambda_{x,2}))\empn(du).\label{cumu-central}
\end{align}
Now, by \eqref{Gammanx}, $\widetilde{\rm Var}^n_\theta (\hV^n_{x,1})=\left(\Hess^n_{x}(\theta^n)\right)_{11}$ and so by the last assertion of Lemma \ref{conv_c_gamma}, for $\sigma$-a.e.\ $\theta$, as $n\to\infty$, $\widetilde{\rm Var}^n_\theta (\hV^n_{x,1})$ converges to $(\Hess_x)_{11}$. Also, since the function $\R\ni \mapsto \partial^4_1(\log\Phi_p(u\lambda_{x,1},\lambda_{x,2}))$ is continuous and has polynomial growth (the latter by Lemma \ref{growth_lmgf}), Lemma \ref{slln} and Lemma \ref{wass-conv}(2) together show that for $\sigma$-a.e.\ $\theta$, the second term on the right-hand-side of \eqref{cumu-central} also has a finite limit as $n\to\infty$. Therefore, for $\sigma$-a.e.\ $\theta$, the sum of the two terms is uniformly bounded.

Next, we deal with the second inequality. By~\eqref{Zbar}, we have
\[
\hV^n_{x} - \widetilde{\mathbb{E}}_{\theta}^n[\hV^n_{x}] = \frac{1}{\sqrt{n}}\sum_{j=1}^n \left(\bar{V}^n_j-\widetilde{\mathbb{E}}^n_\theta[\bar{V}^n_{j}]\right).
\]
By Jensen's inequality, we further obtain
\begin{align*}
&\widetilde{\mathbb{E}}^n_\theta\left[ \norm{\frac{1}{\sqrt{n}}\sum_{j=1}^n \left(\bar{V}^n_j-\widetilde{\mathbb{E}}^n_\theta[\bar{V}^n_{j}]\right)}^3\right] \\
&\qquad\leq  \left(\widetilde{\mathbb{E}}^n_\theta\left[ \norm{\frac{1}{\sqrt{n}}\sum_{j=1}^n \left(\bar{V}^n_j-\widetilde{\mathbb{E}}^n_\theta[\bar{V}^n_{j}]\right)}^4\right]\right)^{3/4}\\
&\qquad\leq \left(\frac{2}{n^2}\widetilde{\mathbb{E}}^n_\theta\left[ \left(\sum_{j=1}^n \left(\bar{V}^n_{j,1}-\widetilde{\mathbb{E}}^n_\theta[\bar{V}^n_{j,1}]\right)\right)^4\right]+\frac{2}{n^2}\widetilde{\mathbb{E}}^n_\theta\left[ \left(\sum_{j=1}^n \left(\bar{V}^n_{j,2}-\widetilde{\mathbb{E}}^n_\theta[\bar{V}^n_{j,2}]\right)\right)^4\right]\right)^{3/4}.
\end{align*}
Now, to show the boundedness of the last display, it suffices to show the boundedness of 
\[
\frac{1}{n^2}\widetilde{\mathbb{E}}^n_\theta\left[ \left(\sum_{j=1}^n \left(\bar{V}^n_{j,m}-\widetilde{\mathbb{E}}^n_\theta[\bar{V}^n_{j,m}]\right)\right)^4\right]
\quad \text{for} \quad m=1,2.
\]
We show the boundedness of the first term above, and the second follows similarly.
For $m\in\{1,2\}$, by the independence of $(\bar{V}^n_{j,1})_{j=1,\ldots,n}$, we have
\begin{align*}
&\frac{1}{n^2}\widetilde{\mathbb{E}}^n_\theta\left[ \left(\sum_{j=1}^n \left(\bar{V}^n_{j,m}-\widetilde{\mathbb{E}}^n_\theta[\bar{V}^n_{j,m}]\right)\right)^4\right]\\
&\qquad= \frac{1}{n^2}\sum_{j=1}^n\widetilde{\mathbb{E}}^n_\theta\left[ \left(\bar{V}^n_{j,m}-\widetilde{\mathbb{E}}^n_\theta[\bar{V}^n_{j,m}]\right)^4\right]+\frac{6}{n^2}\sum_{1\leq i<j\leq n}\widetilde{\mathbb{E}}^n_\theta\left[ \left(\bar{V}^n_{i,1}-\widetilde{\mathbb{E}}^n_\theta[\bar{V}^n_{i,1}]\right)^2\right]\widetilde{\mathbb{E}}^n_\theta\left[ \left(\bar{V}^n_{j,m}-\widetilde{\mathbb{E}}^n_\theta[\bar{V}^n_{j,m}]\right)^2\right]\\
&\qquad \leq \frac{1}{n^2}\sum_{j=1}^n\widetilde{\mathbb{E}}^n_\theta\left[ \left(\bar{V}^n_{j,m}-\widetilde{\mathbb{E}}^n_\theta[\bar{V}^n_{j,m}]\right)^4\right]+
6\left(\frac{2}{n}\sum_{j=1}^n\widetilde{\mathbb{E}}^n_\theta\left[ \left(\bar{V}^n_{j,m}-\widetilde{\mathbb{E}}^n_\theta[\bar{V}^n_{j,m}]\right)^4\right]\right)^2
\end{align*}
which is bounded above by \eqref{moment-bnd}. This proves \eqref{moment-bnd2}.
\end{proof}

\begin{lemma}\label{mu_est}
  Fix $x\in\mathbb{J}_p$ and recall the definitions of $\Hess_x$ $\Phi_p$, $c^n_{x}$, $\Hess^n_{x}$, $\hV^n_{x}$ and
$\hat{\mu}^n_{x,\theta}$ given in \eqref{gamma_x}, \eqref{phi_p}, \eqref{def-cnhessn} \eqref{Zbar} and \eqref{mu_hat}, respectively. Then for $\smeas$-a.e.\ $\theta$  and every neighborhood $U\subset\mathbb{R}^2$ of the origin, there exist a neighborhood $\widetilde{U}$ of $x$ and a constant $C\in(0,1)$  such that for all sufficiently large $n$,
\begin{equation} \label{sup_ct}
  \sup_{t\in U^c} \abs{\hat{\mu}^n_{y,\theta}(t)}^{1/n}< C, \quad y\in \widetilde{U}. 
\end{equation}
Furthermore, for $\smeas$-a.e.\ $\theta$, there exist a neighborhood $U\subset\mathbb{R}^2$ of the origin and a neighborhood $\widetilde{U}$ of $x$ such that for all sufficiently large $n$, 
\begin{align}
\abs{\hat{\mu}^n_{y,\theta}\left(\frac{t}{\sqrt{n}}\right)e^{-itc^n_{y}(\theta^n)}} \leq\exp\left(-\frac{1}{2}\left\langle (\Hess_y -\varepsilon I)t,t\right\rangle \right),\quad  y\in\widetilde{U},\quad t\in U.\label{l2_bound}
\end{align}
\end{lemma}
\begin{proof}
  We omit the subscript $x$ of $\lambda_x$ for notational simplicity.
  Now, for $\theta \in \SS$,  and $t\in\R^2$, the relation \eqref{tilt_mgf} yields the inequality
    \begin{equation}
      \label{ineq1}
\abs{\frac{\Phi_p(\sqrt{n}\theta^n_j(\lambda_1+it_1),\lambda_2+it_2)}{\Phi_p(\sqrt{n}\theta^n_j\lambda_1,\lambda_2)}} =\abs{\widetilde{\mathbb{E}}^n_\theta\left[e^{i\langle t, \bar{V}^n_{j}\rangle}\right]}\leq \widetilde{\mathbb{E}}^n_\theta\left[\abs{e^{i\langle t, \bar{V}^n_{j}\rangle}}\right]\leq 1.
\end{equation}
Noting from~\eqref{phi_p} that $\Phi_p(t)$ is the Fourier transform of the joint density of $(Y_1,|Y_1|^p)$, evaluated at $+it$, we can apply the Riemann-Lebesgue lemma~\cite[Theorem 8.22]{Folland99} to obtain
\[
\norm{(\sqrt{n}\theta^n_j(i\lambda_1-t_1),i\lambda_2-t_2)}\to\infty  \quad
\Rightarrow \quad \abs{\Phi_p(\sqrt{n}\theta^n_j(\lambda_1+it_1),\lambda_2+it_2)}\to 0.
\]
Now for $\theta^n_j \neq 0$,
 $\norm{t}\to \infty$ implies $||(\sqrt{n}\theta^n_j(i\lambda_1-t_1),i\lambda_2-t_2)||\to\infty$. Thus, under the assumption that $\theta^n_j \neq 0$, we see that
\[ 
\lim_{\norm{t}\to\infty} \abs{\frac{\Phi_p(\sqrt{n}\theta^n_j(\lambda_1+it_1),\lambda_2+it_2)}{\Phi_p(\sqrt{n}\theta^n_j\lambda_1,\lambda_2)}}=0.
\]
Since $\Phi_p$ is a moment generating function which converges to $0$ at infinity, $\Phi_p$ is strictly smaller than $1$ other than at the origin. For any neighborhood of the origin $U\subset\mathbb{R}^2$ and any $0<K<\infty$, there exists $0<r<1$ such that for all $t\in U^c$, if $K^{-1}\leq\abs{\sqrt{n}\theta^n_j}\leq K$ and $\theta^n_j\neq 0$, then 
\[
 \abs{\frac{\Phi_p(\sqrt{n}\theta^n_j(\lambda_1+it_1),\lambda_2+it_2)}{\Phi_p(\sqrt{n}\theta^n_j\lambda_1,\lambda_2)}} < r.
\]
This implies 
  \[
 \abs{\frac{\Phi_p(\sqrt{n}\theta^n_j(\lambda_1+it_1),\lambda_2+it_2)}{\Phi_p(\sqrt{n}\theta^n_j\lambda_1,\lambda_2)}} < r^{1_{\left\{K^{-1}\leq\abs{\sqrt{n}\theta^n_j}\leq K,\theta^n_j\neq0\right\}}}.
 \]
Combining this with~\eqref{mu_hat} yields the inequality 
\begin{align*}
\sup_{t\in U^c}\abs{\hat{\mu}^n_{x,\theta}(t)}^{1/n} &\leq r^{\frac{1}{n}\sum_{j=1}^n1_{\left\{K^{-1}\leq\abs{\sqrt{n}\theta^n_j}\leq K,\theta^n_j\neq 0\right\}}}. 
\end{align*}
Since
      $\frac{1}{n}\sum_{j=1}^n1_{\{K^{-1}\leq\abs{\sqrt{n}\theta^n_j}\leq K\}}= L^n_\theta([K^{-1}, K]\setminus \{0\})$
      whose limit, as $n \rightarrow \infty$,
      is dominated by $c_K := \gamma_2 \left([K^{-1},K]\right)  > 0$ due to
    Lemma  \ref{slln}, we have for $\smeas$-a.e.\ $\theta$,
      \[
\limsup_{n \rightarrow \infty} \sup_{t\in U^c}\abs{\hat{\mu}^n_{x,\theta}(t)}^{1/n}\leq r^{c_K} < 1. 
\]
Thus, for $\smeas$-a.e.\ $\theta$, we have a uniform bound $0<C<1$ such that for all sufficiently large $n$, 
\begin{equation} \label{sup_hmu}
\sup_{t\in U^c}\abs{\hat{\mu}^n_{x,\theta}(t)}^{1/n} < C.
\end{equation}
Since $\Phi_p$ is uniformly continuous in $\lambda_x$ 
by definition and $\lambda_x$ is a infinitely differentiable function of $x$ by the inverse function theorem applied to~\eqref{lambda_x}, we may choose a neighborhood $\widetilde{U}$ of $x$ such that for $y\in\widetilde{U}$, 
\begin{equation*}
\sup_{t\in U^c}\abs{\hat{\mu}^n_{y,\theta}(t)}^{1/n} < C,
\end{equation*}
i.e., for $\smeas$-a.e.\ $\theta$ and all sufficiently large $n$ (possibly depending on $\theta$),  \eqref{sup_ct} holds.

Next, note that by \eqref{mu_hat} and \eqref{cn}, for $t \in \R^2$, 
    \[
    \hat{\mu}^n_{x,\theta}\left(\frac{t}{\sqrt{n}}\right)e^{-i\langle t,c^n_{x}(\theta^n)\rangle} =
    \widetilde{\mathbb{E}}^n_\theta\left[e^{i \left\langle t, \hV^n_{x} - \widetilde{\mathbb{E}}_{\theta}^n[\hV^n_{x}] \right\rangle}\right].  
    \]
    Thus, for $\theta\in\mathbb{S}$, by \eqref{Gammanx} and~\cite[Lemma 3.3.7]{Durrett10},
we have the following expansion:
\begin{align*}
\abs{\hat{\mu}^n_{x,\theta}\left(\frac{t}{\sqrt{n}}\right)e^{-i\langle t,c^n_{x}(\theta^n)\rangle}-1+\frac{1}{2}\langle \Hess^n_{x}(\theta^n) t,t\rangle}& \leq\widetilde{\mathbb{E}}^n_\theta\left[ \abs{\langle t,\hV^n_{x} - \widetilde{\mathbb{E}}_{\theta}^n[\hV^n_{x}]\rangle}^3\right]\\
&\leq\norm{t}^3 \widetilde{\mathbb{E}}^n_\theta\left[ \norm{\hV^n_{x} - \widetilde{\mathbb{E}}_{\theta}^n[\hV^n_{x}]}^3\right].
\end{align*}
For $\varepsilon >0$, by \eqref{moment-bnd2} of Lemma~\ref{moment-est}, we may choose a neighborhood $U\subset \R^2$ of the origin with small enough radius so that the right-hand-side of the last display is bounded by $\varepsilon||t||^2$  for $t\in U$.
On the other hand, by the convergence of $\Hess^n_{x}(\theta^n)$ to $\Hess_x$ established in Lemma~\ref{conv_c_gamma}, for $\smeas$-a.e.\ $\theta$, there exists $\varepsilon>0$ such that $\Hess^n_{x}(\theta^n)-\varepsilon I$ is positive definite for all sufficiently large $n$ (possibly depending on $\theta$) and for $t\in U$,
\begin{align*}
\abs{\hat{\mu}^n_{x,\theta}\left(\frac{t}{\sqrt{n}}\right)e^{-i\langle t,c_x^n(\theta^n)\rangle}}\leq 1-\frac{1}{2}\langle (\Hess^n_{x}(\theta^n)-\varepsilon I) t,t\rangle &\leq \exp\left(-\frac{1}{2}\langle (\Hess^n_{x}(\theta^n)-\varepsilon I)t,t\rangle \right).
\end{align*}
Note that the right-hand side of the last display converges  to the integrable function 
$\exp(-\frac{1}{2}\langle (\Hess_x -\varepsilon I)t,t\rangle)$ as $n$ tends to infinity. 
Similar to the proof of~\eqref{sup_ct}, the uniformity of the bound in~\eqref{l2_bound} follows from the definition in~\eqref{def-cnhessn},~\eqref{psin_p} and the aforementioned uniform continuity of $\Phi_p$ in $x$. 
\end{proof}

\subsection{Proof of the joint density estimate} \label{sec-proppf}

We now combine the lemmas established in Sections \ref{subs-asymptotics}--\ref{subs-integrand} 
to prove the estimate for the density $\bar{h}^n_\theta$ of $\Sn$ obtained
in Proposition \ref{prop-mainest}. 

\begin{proof}[Proof of Proposition \ref{prop-mainest}]
Fix $n\in\N$. Combining Lemma~\ref{g_est}, \eqref{Rn} and \eqref{mu_hat} of Lemma~\ref{conv_c_gamma}, we see that for $x\in\mathbb{J}_p$ and $\smeas$-a.e.\ $\theta$,
\begin{equation}\label{int_gn}
\bar{h}^n_\theta(x) = \frac{n}{2\pi}e^{-n\Psi^*_p(x)}e^{\sqrt{n}R^n_{x}(\theta^n)}\frac{n}{2\pi}\int_{\mathbb{R}^2}\hat{\mu}^n_{x,\theta}(t)dt.
\end{equation}
When compared with~\eqref{gn_est} and \eqref{def-gn}, to prove the proposition, it suffices to show that
\[
\frac{n}{2\pi}\int_{\mathbb{R}^2}\hat{\mu}^n_{x,\theta}(t)dt=\det\Hess_{x}^{-1/2}\exp\left(\norm{\Hess_{x}^{-1/2}c^n_{x}(\theta^n)}^2 \right)(1+o(1)),
\]
with the approximation uniformly for $x$ in any compact set of $\mathbb{J}_p$.

Let $U\subset \mathbb{R}^2$ be a neighborhood of the origin. We split the integral in the last display into two parts
\begin{equation}\label{int_split}
\int_{\mathbb{R}^2} \hat{\mu}^n_{x,\theta}(t)dt = \int_U \hat{\mu}^n_{x,\theta}(t)dt+\int_{U^c} \hat{\mu}^n_{x,\theta}(t)dt.
\end{equation}
Now, by the estimate~\eqref{sup_ct} in Lemma~\ref{mu_est}, we have for $C\in(0,1)$ and $s>1$,
\begin{equation}\label{int_Vc}
\abs{\int_{U^c}\hat{\mu}^n_{x,\theta}(t)dt} \leq \int_{U^c}\abs{\hat{\mu}^n_{x,\theta}(t)}dt\leq C^{n-s}\int_{U^c}\abs{\hat{\mu}^n_{x,\theta}(t)}^{s/n}dt.
\end{equation}
From the definition of $\hat{\mu}^n_{x,\theta}$ in~\eqref{mu_hat} and Lemma \ref{g_est}, we see that $|\hat{\mu}^n_{x,\theta}(t)|^{s/n}$ is integrable. Hence, the right hand side of \eqref{int_Vc}  tends to zero exponentially fast as $n$ tends to infinity. Moreover, the convergence is uniform in a neighborhood of $x$ by~\eqref{sup_ct} from Lemma~\ref{mu_est}. 

Recall the definition of $\hat{\mu}^n_{x,\theta}$ in~\eqref{mu_hat}. By~\eqref{Zbar} and~\eqref{cn}, the characteristic function of $\hV^n_x$ is given by $\hat{\mu}^n_{x,\theta}\left(\frac{t}{\sqrt{n}}\right)e^{-itc^n_{x}(\theta^n)}$. Since the sequence $(\hV^n_x)_{n\in\N}$ satisfies the Lyapunov-type condition stated in~\eqref{moment-bnd} of Lemma~\ref{moment-est}, the central limit theorem implies that it converges weakly to a centered Gaussian distribution with covariance matrix $\Hess_x$.
Thus, the corresponding characteristic functions satisfy
\begin{align} \label{limit_ct}
\hat{\mu}^n_{x,\theta}\left(\frac{t}{\sqrt{n}}\right)e^{-itc^n_{x}(\theta^n)} \to \exp\left(-\frac{1}{2}\langle \Hess_{x} t,t\rangle \right).
\end{align}

Now, by~\eqref{l2_bound} of Lemma~\ref{mu_est} and~\eqref{limit_ct}, we may apply the dominated convergence theorem, and use \eqref{limit_ct} to obtain for $\sigma$ a.e.\  $\theta$,
\begin{align}
\int_U\hat{\mu}^n_{x,\theta}(t)dt & = \frac{1}{n}\int_{\sqrt{n}U}\hat{\mu}^n_{x,\theta}\left(\frac{t}{\sqrt{n}}\right)dt\nonumber\\
& = \frac{1}{n}\int_{\sqrt{n}U}\exp\left(itc^n_x(\theta^n)-\frac{1}{2}\langle \Hess_{x} t,t\rangle \right)dt \nonumber\\
&\quad+ \frac{1}{n}\int_{\sqrt{n}U}e^{itc^n_x(\theta^n)}\left(\hat{\mu}^n_{x,\theta}\left(\frac{t}{\sqrt{n}}\right)e^{-itc^n_{x}(\theta^n)}-\exp\left(-\frac{1}{2}\langle \Hess_{x}(\theta^n) t,t\rangle \right)\right) dt\nonumber\\
&= \frac{1}{n}\int_{\mathbb{R}^2}\exp\left(itc^n_{x}(\theta^n)-\frac{1}{2}\langle \Hess_{x} t,t\rangle \right)dt(1+o(1)),\nonumber
\end{align}
with $\Hess_x$ as in~\eqref{gamma_x}. Using standard properties of Gaussian integrals, this implies that
\begin{align}
\int_U\hat{\mu}^n_{x,\theta}(t)dt = \frac{2\pi}{n}\det\Hess_{x}^{-1/2}\exp\left(\norm{\Hess_{x}^{-1/2}c^n_{x}(\theta^n)}^2 \right)(1+o(1)),\label{int_V}
\end{align}
Combining~\eqref{def-gn}, \eqref{int_gn}, \eqref{int_split},~\eqref{int_V} and the estimate of the integral over $U^c$ in~\eqref{int_Vc},
we conclude that the asymptotic expansion for   the density $\bar{h}^n_\theta(x)$ given in  \eqref{gn_est} 
holds  uniformly for $x$ in any compact subset of $\mathbb{J}_p$.

Finally, by the definition of $\lambda_x$ in \eqref{lambda_x} and the inverse function theorem, the mapping $x\mapsto\lambda_x$ is infinitely differentiable. Therefore, combining \eqref{def-gn}, \eqref{gamma_x}, \eqref{def-cnhessn} and \eqref{Rn}, we conclude $\gn$ is infinitely differentiable.
\end{proof}

\appendix
\section{Infimum of the rate function}
\label{apsec-infrfn}

In this section, we analyze the infimum of the rate function.
\begin{proof}[Proof of Lemma~\ref{inf_ratefunc}]
Recall from~\eqref{eq:Psi} and~\eqref{rate_func}, that we have the following expression for the rate function: for $t\in\mathbb{R}$, 
\begin{align}
\mathbb{I}_{p}(t) & = \inf_{\tau_1\in\mathbb{R},\tau_2>0:\tau_1\tau_2^{-1/p}=t} \Psi_p^*(\tau_1,\tau_2)\nonumber \\
& = \inf_{\tau_1\in\mathbb{R},\widetilde{\tau}_2>0:\tau_1\widetilde{\tau}_2^{-1}=t} \Psi_p^*(\tau_1,\widetilde{\tau}_2^p)\nonumber\\
& = \inf_{\widetilde{\tau}_2>0} \Psi_p^*(\widetilde{\tau}_2t,\widetilde{\tau}_2^p),\label{I_p-rep}
\end{align}
where
$
\Psi_p^*(\widetilde{\tau}_2t,\widetilde{\tau}_2^p) =\sup_{s_1,s_2\in\mathbb{R}}
\left\{ s_1\widetilde{\tau_2}t+s_2\widetilde{\tau}_2^p-\Psi_p(s_1,s_2)\right\}.
$

By Lemmas 5.8 and 5.9 of~\cite{GanKimRam17}, $\Psi_p$ is essentially smooth, convex and lower semi-continuous; see Definition 2.3.5 of~\cite{DemZeiBook} for the definition of essential smoothness. Thus, by convexity, for $t, \tau\in\mathbb{R}$, when $\Psi^*_p(\tau t,\tau^p)<\infty$, there exist $s_i=s_i(\tau t,\tau^p)$, $i=1,2$, that attain the supremum in the definition of $\Psi^*_p(\tau t,\tau^p)$, i.e.,
\begin{align}\label{psi-star-rep}
\Psi^*_p(\tau t, \tau^p) = s_1\tau t+s_2\tau^p-\Psi_p(s_1,s_2),
\end{align}
where, by~\eqref{psi},
$
\Psi_p(s_1,s_2) = \int\Lambda_p(us_1,s_2)\gamma_2(du),
$
with $\gamma_2$ being the standard Gaussian measure and $\Lambda_p$ defined as in~\eqref{logmgf_lp}.
Note that $s_1,s_2$ satisfy the following first order conditions:
\begin{align*}
\tau t  = \partial_1\Psi_p(s_1,s_2)\qquad \text{and} \qquad
\tau^p  = \partial_2\Psi_p(s_1,s_2),
\end{align*}  
where $\partial_i$ represents the partial derivative with respect to $s_i$, for $i=1,2$.
From~\cite[Lemma 5.9]{GanKimRam17}, we can exchange the order of differentiation and integration to obtain
\begin{equation}\label{diff-psi}
\begin{aligned}
\partial_1\Psi_p(s_1,s_2) & = \int_\mathbb{R} u\partial_1\Lambda_p(us_1,s_2)\gamma_2(du), \\
\partial_2\Psi_p(s_1,s_2) & = \int_\mathbb{R} \partial_2\Lambda_p(us_1,s_2)\gamma_2(du).
\end{aligned} 
\end{equation}
To calculate these integrals, we first recall the expression for $\Lambda_p$ established in~\cite[Lemma 5.7]{GanKimRam17},  
\begin{equation}
  \label{lambdap-simple}
\Lambda_p(s_1,s_2) = -\frac{1}{p}\log(1-ps_2)+\log M_{\gamma_p}\left(\frac{s_1}{(1-ps_2)^{1/p}}\right),
\end{equation}
where $M_{\gamma_p}$ denotes the moment generating function of the measure $\gamma_p$ with density defined in~\eqref{pNormal}. Differentiation yields
\begin{equation}\label{diff-lambda}
\begin{aligned}
\partial_1\Lambda_p(us_1,s_2)&=\frac{M'_{\gamma_p}\left(\frac{us_1}{(1-ps_2)^{1/p}}\right)}{M_{\gamma_p}\left(\frac{us_1}{(1-ps_2)^{1/p}}\right)}\frac{1}{(1-ps_2)^{1/p}},\\
\partial_2\Lambda_p(us_1,s_2)&=\frac{1}{1-ps_2}  + \frac{M'_{\gamma_p}\left(\frac{us_1}{(1-ps_2)^{1/p}}\right)}{M_{\gamma_p}\left(\frac{us_1}{(1-ps_2)^{1/p}}\right)}\frac{us_1}{(1-ps_2)^{(p+1)/p}}.
\end{aligned}
\end{equation}
Combining all the  above relations, we obtain
\begin{align}
\tau t &= \int_\mathbb{R} \frac{M'_{\gamma_p}\left(\frac{us_1}{(1-ps_2)^{1/p}}\right)}{M_{\gamma_p}\left(\frac{us_1}{(1-ps_2)^{1/p}}\right)}\frac{u}{(1-ps_2)^{1/p}}\gamma_2(du)\label{s1},\\
\tau^p &= \int_\mathbb{R}\left( \frac{1}{1-ps_2}  + \frac{M'_{\gamma_p}\left(\frac{us_1}{(1-ps_2)^{1/p}}\right)}{M_{\gamma_p}\left(\frac{us_1}{(1-ps_2)^{1/p}}\right)}\frac{us_1}{(1-ps_2)^{(p+1)/p}}\right)\gamma_2(du)\nonumber\\
&=\frac{1}{1-ps_2}+\frac{\tau ts_1}{1-ps_2},  \label{s2}
\end{align}
and note that~\eqref{s2} implies
\begin{equation}\label{tau-t}
\tau^pps_2+\tau ts_1=\tau^p-1.
\end{equation}

Now, in view of \eqref{I_p-rep}, to compute $\mathbb{I}_{p}(t)$ we have to first take the derivative of $\Psi^*_p(\tau t,\tau^p)$ with respect to $\tau$ and set it to $0$. Note that in the following, $s_1,s_2$ are functions of $\tau$ and $t$ satisfying~\eqref{s1} and~\eqref{s2}. Using~\eqref{psi} and \eqref{I_p-rep}, we first
rewrite $\Psi_p(s_1,s_2)$ as
\begin{align*}
\Psi_p(s_1,s_2)&=\int_\mathbb{R} \Lambda_p(us_1,s_2)\gamma_2(du)\\
&=-\frac{1}{p}\log(1-ps_2)+\int_\mathbb{R}\log M_{\gamma_p}\left(\frac{us_1}{(1-ps_2)^{1/p}}\right)\gamma_2(du).
\end{align*}
From equations~\eqref{psi-star-rep}-\eqref{tau-t}, we obtain
\begin{align*}
\frac{d}{d\tau} \Psi^*_p(\tau t,\tau^p) & =\frac{d}{d\tau}\left(s_1\tau t+s_2\tau^p-\Psi_p(s_1,s_2)\right)\\
&=\frac{\partial s_1}{\partial\tau}\tau t + s_1 t +\frac{\partial s_2}{\partial\tau}\tau^p+ps_2\tau^{p-1}-\frac{\partial s_2}{\partial\tau}\frac{1}{1-ps_2}\\
&\quad-\int_\mathbb{R}\frac{M'_{\gamma_p}\left(\frac{us_1}{(1-ps_2)^{1/p}}\right)}{M_{\gamma_p}\left(\frac{us_1}{(1-ps_2)^{1/p}}\right)}\left[\frac{\partial s_1}{\partial\tau}\frac{u}{(1-ps_2)^{1/p}}+\frac{\partial s_2}{\partial\tau}\frac{us_1}{(1-ps_2)^{1/p+1}} \right]\gamma_2(du)\\
& = \frac{\partial s_1}{\partial\tau}\tau t + s_1 t +\frac{\partial s_2}{\partial\tau}\tau^p+ps_2\tau^{p-1}-\frac{\partial s_2}{\partial\tau}\frac{1}{1-ps_2}-\tau t \frac{\partial s_1}{\partial\tau}-\frac{s_1 \tau t}{1-ps_2}\frac{\partial s_2}{\partial\tau}\\
& = s_1t+\frac{\partial s_2}{\partial\tau}\frac{\tau^p(1-ps_2)-s_1\tau t-1}{1-ps_2} +ps_2\tau^{p-1}\\
& = s_1t+ps_2\tau^{p-1} \\
&= \tau^{p-1}-\frac{1}{\tau}.
\end{align*}
Setting the derivative computed above to $0$, we conclude that 
the minimum over $\tau>0$ in~\eqref{I_p-rep} is attained at $\tau=1$. Substituting this back into the  definition of $\mathbb{I}_{p}$, we conclude that 
$\mathbb{I}_{p}(t)=\Psi^*_p(t,1)$ which, along with \eqref{eq:Psi}, proves Lemma \ref{inf_ratefunc}. 
\end{proof}

\section{Proof of the  Central Limit Theorem for the empirical measure}
\label{subs-clt}

\begin{proof}[Proof of Lemma~\ref{clt_expansion}]
Let $(Z^n_j,j=1,\ldots,n)_{n\in\mathbb{N}}$ be independent standard Gaussian random variables. Then note that (e.g. see Section~\ref{subs-reform} or~\cite[Lemma 1]{Schechtman90})
\begin{equation}\label{direc-rep}
 \Theta^n_j \buildrel (d) \over = \frac{Z^n_j}{\norm{Z^n}},
\end{equation}
where we use $\norm{Z^n}=\norm{Z^n}_{n,2}$ to denote the Euclidean norm of the vector $Z^n := (Z^n_1,\cdots,Z^n_n)$. 

Since $F$ is a thrice continuously differentiable function, we may apply Taylor's theorem, for $x\in\mathbb{R}$ and $h>0$ to obtain
\begin{align*}
F(x+h)& = F(x) +F'(x)h+\frac{F''(x)}{2}h^2+\frac{F'''(\widetilde{x})}{6}h^3,
\end{align*}
for some $\widetilde{x}\in (x,x+h)$. With the expansion above, we obtain
\begin{align}
&\frac{1}{\sqrt{n}}\sum_{j=1}^n\left[F\left(\sqrt{n}\frac{Z^n_j}{\norm{Z^n}}\right) - \mathbb{E}\left[F(Z)\right]\right]\nonumber\\
& = \frac{1}{\sqrt{n}}\sum_{j=1}^n \left[ F(Z^n_j) -\mathbb{E}[F(Z)] + F'(Z^n_j)\left( \frac{\sqrt{n}Z^n_j}{\norm{Z^n}}-Z^n_j\right)+\frac{F''(Z^n_j)}{2}\left( \frac{\sqrt{n}Z^n_j}{\norm{Z^n}}-Z^n_j\right)^2\right.\nonumber\\
&\quad\quad\quad\quad\quad \left.+ \frac{F'''(\widetilde{Z}^n_j)}{6}\left( \frac{\sqrt{n}Z^n_j}{\norm{Z^n}}-Z^n_j\right)^3\right],\nonumber\\
& = \hat{r}_n(F)+\frac{1}{\sqrt{n}}\hat{s}_n(F)+\frac{1}{\sqrt{n}}\sum_{j=1}^n\frac{F'''(\widetilde{Z}^n_j)}{6}\left( \frac{\sqrt{n}Z^n_j}{\norm{Z^n}}-Z^n_j\right)^3\label{exp}
\end{align}
where $\hat{r}_n(\cdot)$ and $\hat{s}_n(\cdot)$ are defined in~\eqref{def-rn} and~\eqref{def-sn}, respectively, and 
$\widetilde{Z}^n_i\in\R$ lies between $Z^n_j$ and $\sqrt{n}Z^n_j/\norm{Z^n}$. 

In the following, the notation $o(1)$ means having order $o(1)$ in probability $\mathbb{P}$. We first show that the last term in~\eqref{exp} is of order $o(1/n)$ in probability. By assumption, $\abs{F'''}$ has polynomial growth, so there exist $q>0$ and $C <\infty$ such that
\[
\abs{F'''(t)}<C(1+\abs{t}^q), \quad \forall t\in\mathbb{R}.
\]
Therefore, for each $n\in\mathbb{N}$,
\begin{align*}
\sum_{j=1}^n\frac{\abs{F'''(\widetilde{Z}^n_j)}}{6}\left( \frac{\sqrt{n}Z^n_j}{\norm{Z^n}}-Z^n_j\right)^3 
& \leq \frac{C}{6} \sum_{j=1}^n\left(1+\abs{\widetilde{Z}^n_j}^q\right)\abs{ \frac{\sqrt{n}Z^n_j}{\norm{Z^n}}-Z^n_j}^3.
\end{align*}

Since $\widetilde{Z}^n_j$ lies between $Z^n_j$ and $\sqrt{n}Z^n_j/\norm{Z^n}$, and $\sqrt{n}/\norm{Z^n}$ converges to $1$ almost surely. For each $0<\bar{C}<\infty$, there exists $N=N(w)$ such that a.s. for all $n>N$,
\[
\abs{\widetilde{Z}^n_j}<\abs{Z^n_j}(1+\bar{C}).
\]
Combining the last two inequalities above, we obtain for some constant $C'<\infty$, and all $n>N$,
\begin{align*}
\sum_{j=1}^n\frac{\abs{F'''(\widetilde{Z}^n_j)}}{6}\left( \frac{\sqrt{n}Z^n_j}{\norm{Z^n}}-Z^n_j\right)^3 
& \leq C' \sum_{j=1}^n\left(1+\abs{Z^n_j}^q\right)\abs{ \frac{\sqrt{n}Z^n_j}{\norm{Z^n}}-Z^n_j}^3 \\
& = C' \frac{\abs{\norm{Z^n}-\sqrt{n}}^3}{\sqrt{n}}\frac{n^{3/2}}{\norm{Z^n}^3}\left[\frac{1}{n}\sum_{j=1}^n\abs{Z^n_j}^3(1+\abs{Z^n_j}^q)\right].
\end{align*}

From the Gaussian concentration inequality (see~\cite[Theorem 3.1.1]{Vershynin18}), there exists a universal constant $c$ such that for $\delta >0$,
\[
\mathbb{P}\left ( \abs{\norm{Z^n}-\sqrt{n}}>\delta \right) \leq 2e^{-c\delta^2},
\]
Given $\epsilon > 0$, we have
\begin{align}
\mathbb{P}\left( \frac{1}{\sqrt{n}} \abs{\norm{Z^n}-\sqrt{n}}^3>\epsilon \right) 
& = \mathbb{P}\left( \abs{\norm{Z^n}-\sqrt{n}}> n^{1/6}\epsilon^{1/3} \right)\nonumber \\
& \leq 2 e^{-c\epsilon^{2/3}n^{1/3}}\nonumber\\
&\to 0, \quad \text{as}\quad n\to \infty.\label{op1}
\end{align}

On the other hand, since $(Z^n_j)_{j=1,\ldots,n}$ are independent, by the strong law of large numbers for triangular arrays, as $n$ tends to infinity, almost surely 
\begin{align} \label{op2}
\frac{1}{n}\sum_{j=1}^n\abs{Z^n_j}^3(1+\abs{Z^n_j}^q)\to\mathbb{E}\left[\abs{Z}^3(1+\abs{Z}^q)\right].
\end{align}
Similarly, the strong law of large numbers also ensures that as $n$ tends to infinity,
\begin{equation}\label{op3}
\frac{\norm{Z^n}}{\sqrt{n}}\to 1,\quad \text{a.s.}
\end{equation}
Together,~\eqref{op1},~\eqref{op2} and~\eqref{op3} show that
\[
\sum_{j=1}^n\frac{\abs{F'''(\widetilde{Z}^n_j)}}{6}\left( \frac{\sqrt{n}Z^n_j}{\norm{Z^n}}-Z^n_j\right)^3 =o(1).
\]
We may then rewrite~\eqref{exp} as follows:
\begin{align}\label{exp2}
&\frac{1}{\sqrt{n}}\left(\sum_{j=1}^nF\left(\sqrt{n}\frac{Z^n_j}{\norm{Z^n}}\right) - \mathbb{E}\left[F(Z)\right]\right)=\hat{r}_n(F) + \frac{1}{\sqrt{n}}\hat{s}_n(F)+o\left(\frac{1}{\sqrt{n}}\right).
\end{align}

Due to the assumption that $F'''$, $G_1''$ and $G_2''$ all have polynomial growth, the variances of $F(Z)$, $F'(Z)Z$, $F''(Z)Z^2$, $G_1(Z)$, $G_1'(Z)Z$, $G_2(Z)$ and $G_2'(Z)Z$ are all finite. Define sequences $(\mathfrak{A}_n)$, $(\mathfrak{B}_n)$, $(\mathfrak{C}_n)$, $(\mathfrak{D}_n)$, $(\mathfrak{E}_n)$, $(\mathfrak{F}_n)$, $(\mathfrak{G}_n)$ and $(\mathfrak{H}_n)$ as follows:
\begin{align*}
\mathfrak{A}_n & := \frac{1}{\sqrt{n}} \sum_{j=1}^n \left( F(Z^n_j) -\mathbb{E}[F(Z)]\right), 
\quad &
\mathfrak{B}_n  & := \frac{1}{\sqrt{n}} \sum_{j=1}^n \left( F'(Z^n_j)Z^n_j -\mathbb{E}[F'(Z)Z]\right),\\
\mathfrak{C}_n & := \frac{1}{\sqrt{n}} \sum_{j=1}^n \left( F''(Z^n_j)(Z^n_j)^2 -\mathbb{E}[F''(Z)Z^2]\right),
\quad &
\mathfrak{D}_n & := \frac{1}{\sqrt{n}} \sum_{j=1}^n \left(\abs{Z^n_j}^2-1\right),\\
\mathfrak{E}_n &:= \frac{1}{\sqrt{n}} \sum_{j=1}^n \left( G_1(Z^n_j) -\mathbb{E}[G_1(Z)]\right), 
\quad &
\mathfrak{F}_n  &: = \frac{1}{\sqrt{n}} \sum_{j=1}^n \left( G_1'(Z^n_j)Z^n_j -\mathbb{E}[G_1'(Z)Z]\right),\\
\mathfrak{G}_n &:= \frac{1}{\sqrt{n}} \sum_{j=1}^n \left( G_2(Z^n_j) -\mathbb{E}[G_2(Z)]\right), 
\quad &
\mathfrak{H}_n  &: = \frac{1}{\sqrt{n}} \sum_{j=1}^n \left( G_2'(Z^n_j)Z^n_j -\mathbb{E}[G_2'(Z)Z]\right).
\end{align*}
By the multivariate central limit theorem, $(\mathfrak{A}_n,\mathfrak{B}_n,\mathfrak{C}_n,\mathfrak{D}_n,\mathfrak{E}_n,\mathfrak{F}_n,\mathfrak{G}_n,\mathfrak{H}_n)$ converges in distribution to a jointly Gaussian random vector $M:=(\mathfrak{A},\mathfrak{B},\mathfrak{C},\mathfrak{D},\mathfrak{E},\mathfrak{F},\mathfrak{G},\mathfrak{H})$ in $\mathbb{R}^8$ with mean $0$ and covariance matrix 
\begin{align}
\left(\widetilde{\Sigma}\right)_{ij} := \mathrm{Cov}(M_i,M_j),\quad \text{for}\quad i,j=1,\ldots,6,
\end{align}
where
\[
\left(M_1,M_2,M_3,M_4,M_5,M_6,M_7,M_8\right) := \left(F(Z),F'(Z)Z,F''(Z)Z^2,Z^2,G_1(Z),G_1'(Z)Z,G_2(Z),G_2'(Z)Z \right).
\]

By the Skorokhod representation theorem, we can find $(\widetilde{\mathfrak{A}}_n,\widetilde{\mathfrak{B}}_n,\widetilde{\mathfrak{C}}_n,\widetilde{\mathfrak{D}}_n,\widetilde{\mathfrak{E}}_n,\widetilde{\mathfrak{F}}_n,\widetilde{\mathfrak{G}}_n,\widetilde{\mathfrak{H}}_n)$ and $\widetilde{M}:=(\widetilde{\mathfrak{A}},\widetilde{\mathfrak{B}},\widetilde{\mathfrak{C}},\widetilde{\mathfrak{D}},\widetilde{\mathfrak{E}},\widetilde{\mathfrak{F}},\widetilde{\mathfrak{G}},\widetilde{\mathfrak{H}})$ all defined on some common probability space, such that
\[
(\mathfrak{A}_n,\mathfrak{B}_n,\mathfrak{C}_n,\mathfrak{D}_n,\mathfrak{E}_n,\mathfrak{F}_n,\mathfrak{G}_n,\mathfrak{H}_n,M) \buildrel (d) \over =(\widetilde{\mathfrak{A}}_n,\widetilde{\mathfrak{B}}_n,\widetilde{\mathfrak{C}}_n,\widetilde{\mathfrak{D}}_n,\widetilde{\mathfrak{E}}_n,\widetilde{\mathfrak{F}}_n,\widetilde{\mathfrak{G}}_n,\widetilde{\mathfrak{H}}_n,\widetilde{M}),
\]
and
\begin{equation}
  \label{sk-conv}
(\widetilde{\mathfrak{A}}_n,\widetilde{\mathfrak{B}}_n,\widetilde{\mathfrak{C}}_n,\widetilde{\mathfrak{D}}_n,\widetilde{\mathfrak{E}}_n,\widetilde{\mathfrak{F}}_n,\widetilde{\mathfrak{G}}_n,\widetilde{\mathfrak{H}}_n)\to \widetilde{M}\text{ a.s.}
\end{equation}

Now, we substitute $(\widetilde{\mathfrak{A}}_n,\widetilde{\mathfrak{B}}_n,\widetilde{\mathfrak{C}}_n,\widetilde{\mathfrak{D}}_n,\widetilde{\mathfrak{E}}_n,\widetilde{\mathfrak{F}}_n,\widetilde{\mathfrak{G}}_n,\widetilde{\mathfrak{H}}_n)$ into~\eqref{exp2}, and we first take care of $r_n$
\begin{align*}
\hat{r}_n(F) & =\frac{1}{\sqrt{n}}\sum_{j=1}^n \left[ F(Z^n_j) -\mathbb{E}[F(Z)] + F'(Z^n_j)\left( \frac{\sqrt{n}Z^n_j}{\norm{Z^n}}-Z^n_j\right)\right]\\
&\buildrel (d) \over = \frac{1}{\sqrt{n}} \left( \sqrt{n}\widetilde{\mathfrak{A}}_n +(\sqrt{n}\widetilde{\mathfrak{B}}_n+n\mathbb{E}[F'(Z)Z])\frac{\sqrt{n}-(\sqrt{n}\widetilde{\mathfrak{D}}_n+n)^{1/2}}{(\sqrt{n}\widetilde{\mathfrak{D}}_n+n)^{1/2}}\right) \\
& = \widetilde{\mathfrak{A}}_n+\sqrt{n}\left( \mathbb{E}[F'(Z)Z]+\frac{\widetilde{\mathfrak{B}}_n}{\sqrt{n}}\right)\left( \frac{1-(1+\widetilde{\mathfrak{D}}_n/\sqrt{n})^{1/2}}{(1+\widetilde{\mathfrak{D}}_n/\sqrt{n})^{1/2}}\right)\\
& = \widetilde{\mathfrak{A}}_n+\sqrt{n}H_1\left( \frac{\widetilde{\mathfrak{B}}_n}{\sqrt{n}},\frac{\widetilde{\mathfrak{D}}_n}{\sqrt{n}}\right),
\end{align*}
where $H_1:\mathbb{R}^2\to\mathbb{R}$ is the mapping
\[
H_1(x,y) := \left(\mathbb{E}[F'(Z)Z]+x\right)\frac{1-(1+y)^{1/2}}{(1+y)^{1/2}}.
\]
Since $\widetilde{\mathfrak{B}}_n/{\sqrt{n}}$ and $\widetilde{\mathfrak{D}}_n/{\sqrt{n}}$ converge to $0$ almost surely by \eqref{sk-conv}, we consider the Taylor expansion of $H_1$ at $(0,0)$:
\begin{align*}
H_1(x,y) &= \left.\frac{1-(1+y)^{1/2}}{(1+y)^{1/2}}\right|_{(x,y)=(0,0)}x\\
&\quad+\left.\left(\mathbb{E}[F'(Z)Z]+x\right)\frac{-1}{2(1+y)^{3/2}}\right|_{(x,y)=(0,0)}y \\
&\quad+O(x^2+y^2) \\
& = -\frac{y}{2}\mathbb{E}[F'(Z)Z]+O(x^2+y^2).
\end{align*}

Combining the last three displays, we obtain
\begin{align*}
\hat{r}_n(F)
& \buildrel (d) \over= \widetilde{\mathfrak{A}}_n+\sqrt{n}\left(-\frac{\widetilde{\mathfrak{D}}_n}{2\sqrt{n}}\mathbb{E}[F'(Z)Z]+O\left( \frac{\widetilde{\mathfrak{B}}_n^2}{n}+\frac{\widetilde{\mathfrak{D}}_n^2}{n}\right )\right)\\
& = \widetilde{\mathfrak{A}}_n-\frac{1}{2}\mathbb{E}[F'(Z)Z]\widetilde{\mathfrak{D}}_n+\mathbb{E}[F'(Z)Z]O\left( \frac{\widetilde{\mathfrak{B}}_n^2}{\sqrt{n}}+\frac{\widetilde{\mathfrak{D}}_n^2}{\sqrt{n}}\right ).
\end{align*}
By the a.s. convergence, $(\widetilde{\mathfrak{A}}_n,\widetilde{\mathfrak{D}}_n)\to(\widetilde{\mathfrak{A}},\widetilde{\mathfrak{D}})$, we see that as $n$ tends to infinity,
\[
\frac{\widetilde{\mathfrak{B}}_n^2}{\sqrt{n}}+\frac{\widetilde{\mathfrak{D}}_n^2}{\sqrt{n}}\to 0, \quad \text{a.s.}
\]
Applying Slutsky's lemma and the almost sure convergence above, we obtain
\begin{align}\label{rn-limit}
\hat{r}_n(F)& \Rightarrow \widetilde{\mathfrak{A}}-\frac{1}{2}\mathbb{E}[F'(Z)Z]\widetilde{\mathfrak{D}},
\end{align}
as $n\to\infty$.

Similarly, for $s_n$ we have
\begin{align*}
\hat{s}_n(F) & = \frac{1}{2}\sum_{j=1}^nF''(Z^n_j)(Z^n_j)^2\left( \frac{\sqrt{n}}{\norm{Z^n}}-1\right)^2 \\
& = \frac{1}{2}n\left(\mathbb{E}[F''(Z)Z^2]+\frac{\mathfrak{C}_n}{\sqrt{n}}\right)\left(\frac{1}{(1+\mathfrak{D}_n/\sqrt{n})^{1/2}}-1\right)^2 \\
& \buildrel (d) \over = \frac{1}{2}nH_2\left( \frac{\widetilde{\mathfrak{C}}_n}{\sqrt{n}}, \frac{\widetilde{\mathfrak{D}}_n}{\sqrt{n}}\right),
\end{align*}
where $H_2:\mathbb{R}^2\to\mathbb{R}$ is the mapping
\[
H_2(x,y) := \left(\mathbb{E}[F''(Z)Z^2]+x\right)\left(\frac{1}{(1+y)^{1/2}}-1\right)^2,\quad (x,y)\in\mathbb{R}^2.
\]
Note that $\widetilde{\mathfrak{C}}_n/{\sqrt{n}}$ and $\widetilde{\mathfrak{D}}_n/{\sqrt{n}}$ converge to $0$ almost surely by \eqref{sk-conv}.
 We now apply the Taylor expansion to $H_2$ at $(0,0)$ and obtain
\begin{align*}
H_2(x,y) = \frac{1}{4}\mathbb{E}[F''(Z)Z^2]y^2+O(x^3+y^3).
\end{align*}
With the above expansion for $H_2$, we write
\begin{align}
\hat{s}_n(F)
&\buildrel (d) \over = \frac{1}{8}\mathbb{E}[F''(Z)Z^2]\widetilde{\mathfrak{D}}_n^2+O\left(\frac{\widetilde{\mathfrak{C}}_n^3}{\sqrt{n}}+\frac{\widetilde{\mathfrak{D}}_n^3}{\sqrt{n}}\right)\nonumber\\
&\Rightarrow \frac{1}{8}\mathbb{E}[F''(Z)Z^2]\widetilde{\mathfrak{D}}^2,\label{sn-limit}
\end{align}
as $n$ tends to infinity, which holds since $\widetilde{\mathfrak{D}}_n\to\widetilde{\mathfrak{D}}$ almost surely.
This completes the analysis of the expansion for $F$. Fix $i =1, 2$, we next consider the expansion for $G_i$.
Following the same method, we can write
\begin{align*}
&\sqrt{n}\left[\frac{1}{n}\sum_{j=1}^n G_i\left(\sqrt{n}\frac{Z^n_j}{\norm{Z^n}}\right)-\mathbb{E}[G_i(Z)]\right] \\
&\quad = \frac{1}{\sqrt{n}}\sum_{j=1}^n\left[G_i(Z^n_j)-\mathbb{E}[G_i(Z)]+G_i'(Z^n_j)\left(\frac{\sqrt{n}Z^n_j}{\norm{Z^n}}-Z^n_j\right)+\frac{1}{2}G_i''(\widetilde{Z}^n_i)\left(\frac{\sqrt{n}Z^n_j}{\norm{Z^n}}-Z^n_j\right)^2\right].
\end{align*}

Again by assumption, $G_i''$ has polynomial growth, and thus the last term is of order $o(1)$. Hence, we may rewrite the terms above as follows:
\begin{align}
&\sqrt{n}\left[\frac{1}{n}\sum_{j=1}^n G_i\left(\sqrt{n}\frac{Z^n_j}{\norm{Z^n}}\right)-\mathbb{E}[G_i(Z)]\right] \nonumber\\
&\quad = \frac{1}{\sqrt{n}}\sum_{j=1}^n\left[G_i(Z^n_j)-\mathbb{E}[G_i(Z)]+G_i'(Z^n_j)\left(\frac{\sqrt{n}Z^n_j}{\norm{Z^n}}-Z^n_j\right)\right] +o(1) \nonumber\\
&\quad =\hat{r}_n(G_i)+o(1).\label{rn-limit2}
\end{align}

Thus, the expansion in Lemma~\ref{clt_expansion} follows from~\eqref{direc-rep},~\eqref{rn-limit}~\eqref{sn-limit} and~\eqref{rn-limit2}. The second assertion of the lemma is a consequence of~\eqref{rn-limit},~\eqref{sn-limit}, the analog of~\eqref{rn-limit} with $F$ replaced with $G_i$ and the joint convergence of $(\widetilde{\mathfrak{A}}_n,\widetilde{\mathfrak{D}}_n,\widetilde{\mathfrak{E}}_n)\Rightarrow(\widetilde{\mathfrak{A}},\widetilde{\mathfrak{D}},\widetilde{\mathfrak{E}})$ and $(\widetilde{\mathfrak{A}}_n,\widetilde{\mathfrak{D}}_n,\widetilde{\mathfrak{G}}_n)\Rightarrow(\widetilde{\mathfrak{A}},\widetilde{\mathfrak{D}},\widetilde{\mathfrak{G}})$.
\end{proof}

\section{Proof of Lemma~\ref{growth_lmgf}}\label{pf_growth}
  
\begin{proof}
For $p=2$,  $\gamma_2$ is the standard Gaussian, $\log M_{\gamma_2}(t)=t^2/2$ and so the lemma follows.

Next, we consider the case $p>2$.
Let $Y$ be a generalized $p$-Gaussian random variable with density as in~\eqref{pNormal}. The moments of $Y$ are given in~\cite{Nadarajah05} by
\begin{equation}\label{moments}
\mathbb{E}\left[Y^m \right] = 
\begin{cases}
0, \quad $m$ \text{ odd,}\\
\frac{p^{m/p}\Gamma\left(\frac{m+1}{p}\right)}{\Gamma\left( \frac{1}{p}\right)}, \quad $m$ \text{ even}.
\end{cases}
\end{equation}

Note that $d\log M_{\gamma_p}(t)/dt=\mathbb{E}[Ye^{tY}]/\mathbb{E}[e^{tY}]$, and for each $k>1$, $d^k\log M_{\gamma_p}(t)/dt^k$ is a linear combination of products of functions the form
\[
t\mapsto\frac{\mathbb{E}\left[Y^ne^{tY} \right]}{\mathbb{E}\left[e^{tY} \right]} ,\quad \text{for} \quad n=1,\ldots,k.
\]
Therefore, we only need to show that these functions have at most polynomial growth. The case when $k=0$ then follows since the derivative of $\log M_{\gamma_p}$ has polynomial growth, thus, $\log M_{\gamma_p}$ also has polynomial growth.

We first consider the case when $n$ is odd and the case when $n$ is even can be deduced analogously. Note that for $t\in\R$ and $n$ odd,
\begin{align*}
\frac{\mathbb{E}\left[Y^ne^{tY} \right]}{\mathbb{E}\left[e^{tY} \right]} 
&= \frac{\sum_{m=0}^\infty t^{2m+1}\frac{(p^{1/p})^{2m+1+n}\Gamma\left(\frac{2m+2+n}{p} \right)}{\Gamma\left(2m+1+n\right)}}{\sum_{m=0}^\infty t^{2m}\frac{(p^{1/p})^{2m}\Gamma\left(\frac{2m+1}{p} \right)}{\Gamma\left(2m+1\right)}}  \\
&\leq  \frac{\sum_{m=0}^\infty t^{2m+1}\frac{(p^{1/p})^{2m+1+n}\Gamma\left(\frac{2m+2+n}{p} \right)}{\Gamma\left(2m+1+n\right)}}{\sum_{m=n'}^\infty t^{2m}\frac{(p^{1/p})^{2m}\Gamma\left(\frac{2m+1}{p} \right)}{\Gamma\left(2m+1\right)}}  \\
& =  \frac{\sum_{m=0}^\infty t^{2m+1}\frac{(p^{1/p})^{2m+1+n}\Gamma\left(\frac{2m+2+n}{p} \right)}{\Gamma\left(2m+1+n\right)}}{\sum_{m=0}^\infty t^{2m+2n'}\frac{(p^{1/p})^{2m+2n'}\Gamma\left(\frac{2m+2n'+1}{p} \right)}{\Gamma\left(2m+2n'+1\right)}}.
\end{align*}
Pick $n' = (n-1)/2$ to obtain
\begin{align*}
\frac{\mathbb{E}\left[Y^ne^{tY} \right]}{\mathbb{E}\left[e^{tY} \right]}  
&\leq \frac{t}{t^{n-1}}\times\frac{\sum_{m=0}^\infty t^{2m}\frac{(p^{1/p})^{2m+1+n}\Gamma\left(\frac{2m+2+n}{p} \right)}{\Gamma\left(2m+1+n\right)}}{\sum_{m=0}^\infty t^{2m}\frac{(p^{1/p})^{2m+n-1}\Gamma\left(\frac{2m+n}{p} \right)}{\Gamma\left(2m+n\right)}}.
\end{align*}

Now, note that for each $m\in\mathbb{N}\cup \{0\}$,
\begin{align*}
\frac{\frac{(p^{1/p})^{2m+1+n}\Gamma\left(\frac{2m+2+n}{p} \right)}{\Gamma\left(2m+2\right)}}{\frac{(p^{1/p})^{2m+n-1}\Gamma\left(\frac{2m+n}{p} \right)}{\Gamma\left(2m+n\right)}} = \frac{p^{2/p}}{2m+1+n}\frac{\Gamma\left(\frac{2m+2+n}{p}\right)}{\Gamma\left(\frac{2m+1+n}{p}\right)}\leq (2m+1+n)^{\frac{2}{p}-1} \leq 1,
\end{align*}
where the second to last inequality is due to Wendel~\cite[Equation 7]{Wendel48}. Thus, we have shown that $\mathbb{E}[Y^ne^{tY}]/\mathbb{E}[e^{tY}]\leq t^{2-n}$, which has at most linear growth for any $n\in\N$.

Lastly, we turn to the case when $1<p<2$. We simply demonstrate the case $k=1$, the general result can be deduced using similarly. Again, we start with $k=1$, and for general $k\in\mathbb{N}\cup \{0\}$, the result can be deduced using the same technique as in the case $p>2$. 

In view of~\eqref{moments} we have for $t\in\R$,
\begin{align*}
\frac{d}{dt} \log M_{\gamma_p}(t)
&= \frac{\mathbb{E}\left[Ye^{tY} \right]}{\mathbb{E}\left[e^{tY} \right]} \\
&= \frac{\sum_{m=0}^\infty t^{2m+1}\frac{\mathbb{E}[Y^{2m+2}]}{(2m+1)!}}{\sum_{m=0}^\infty t^{2m}\frac{\mathbb{E}[Y^{2m}]}{(2m)!}}\\
&= \frac{\sum_{m=0}^{n-1} t^{2m+1}\frac{(p^{1/p})^{2m+2}\Gamma\left(\frac{2m+3}{p} \right)}{\Gamma\left(2m+2\right)}+\sum_{m=n}^\infty t^{2m+1}\frac{(p^{1/p})^{2m+2}\Gamma\left(\frac{2m+3}{p} \right)}{\Gamma\left(2m+2\right)}}{\sum_{m=0}^\infty t^{2m}\frac{(p^{1/p})^{2m}\Gamma\left(\frac{2m+1}{p} \right)}{\Gamma\left(2m+1\right)}} \\
& \leq \sum_{m=1}^{n-1} t^{2m+1}\frac{(p^{1/p})^{2m+2}\Gamma\left(\frac{2m+3}{p} \right)}{\Gamma\left(2m+2\right)} +\frac{\sum_{m=n}^\infty t^{2m+1}\frac{(p^{1/p})^{2m+2}\Gamma\left(\frac{2m+3}{p} \right)}{\Gamma\left(2m+2\right)}}{\sum_{m=0}^\infty t^{2m}\frac{(p^{1/p})^{2m}\Gamma\left(\frac{2m+1}{p} \right)}{\Gamma\left(2m+1\right)}}\\
&= \sum_{m=1}^{n-1} t^{2m+1}\frac{(p^{1/p})^{2m+2}\Gamma\left(\frac{2m+3}{p} \right)}{\Gamma\left(2m+2\right)} +t^{2n+1}\frac{\sum_{m=0}^\infty t^{2m}\frac{(p^{1/p})^{2m+2n+2}\Gamma\left(\frac{2m+2n+3}{p} \right)}{\Gamma\left(2m+2n+2\right)}}{\sum_{m=0}^\infty t^{2m}\frac{(p^{1/p})^{2m}\Gamma\left(\frac{2m+1}{p} \right)}{\Gamma\left(2m+1\right)}},
\end{align*}
where the inequality follows from $\mathbb{E}[e^{tY}]\geq 1$, which is due to Jensen's inequality. 

To conclude the proof of the lemma, it suffices to show that there exists $n\in\mathbb{N}$ such that for all $m\in\mathbb{N}\cup\{0\}$
\[
\frac{\frac{(p^{1/p})^{2m+2n+2}\Gamma\left(\frac{2m+2n+3}{p} \right)}{\Gamma\left(2m+2n+2\right)}}{\frac{(p^{1/p})^{2m}\Gamma\left(\frac{2m+1}{p} \right)}{\Gamma\left(2m+1\right)}} = p^{(2n+2)/p}\frac{\Gamma\left((2m+2n+3)/p \right)\Gamma\left(2m+1 \right)}{\Gamma\left((2m+1)/p \right) \Gamma\left(2m+2n+2 \right)} \leq 1.
\]
To this end,
pick $a,b\in\mathbb{N}$ such that the following inequalities hold:
\begin{align}
(2m+1)\left(1-\frac{1}{p}\right)-1&<a<(2m+1)\left(1-\frac{1}{p}\right);\nonumber
\\
\quad (2m+2n)\left(1-\frac{1}{p}\right)-\frac{3}{p}&<b<(2m+2n)\left(1-\frac{1}{p}\right)-\frac{3}{p}+1. \label{prop-ab}
\end{align}
Then we have the inequality
\begin{align} \label{prop-b-a}
2n-\frac{2n}{p}-1-\frac{2}{p} < b-a < 2n-\frac{2n}{p}+1-\frac{2}{p}.
\end{align}
Now we use the identity $\Gamma(z+1)=z\Gamma(z)$ and the chosen $a,b$ above to obtain
\begin{align*}
& p^{\frac{2n+2}{p}}\frac{\Gamma\left((2m+2n+3)/p \right)\Gamma\left(2m+1 \right)}{\Gamma\left((2m+1)/p \right) \Gamma\left(2m+2n+2 \right)} \\
 & \quad=p^{\frac{2n+2}{p}}\frac{\Gamma(2m+1)\left(\frac{2m+1}{p}\right)\cdots\left(\frac{2m+1}{p}+a\right)}{\Gamma\left(\frac{2m+1}{p}+a+1 \right)}\frac{\Gamma\left(\frac{2m+2n+3}{p}+b+1 \right)}{\Gamma(2m+2n+2)\left(\frac{2m+2n+3}{p}\right)\cdots\left(\frac{2m+2n+3}{p}+b\right)} \\
 & \quad\leq p^{2n+1}\frac{(2m+1)\cdots(2m+1+ap)}{(2m+2n+3)\cdots(2m+2n+3+bp)},
 \end{align*}
 where the inequality follows from \eqref{prop-ab} and \eqref{prop-b-a}. We further see that 
  \begin{align*}
 &p^{\frac{2n+2}{p}}\frac{\Gamma\left((2m+2n+3)/p \right)\Gamma\left(2m+1 \right)}{\Gamma\left((2m+1)/p \right) \Gamma\left(2m+2n+2 \right)}\\
 & \quad \leq p^{2n+1}\frac{2m+1}{2m+2n+3}\cdots\frac{2m+1+ap}{2m+2n+3+ap}\frac{1}{(2m+2n+3+(a+1)p)\cdots(2m+2n+3+bp)}\\
 &\quad \leq  p^{2n+1}\frac{1}{(2n+3+(a+1)p)\cdots(2n+3+bp)}\\
  &\quad \leq  p^{2n+1}\left(\frac{1}{(2n+3+(a+1)p)}\right)^{b-a}\\
  &\quad \leq p\left(\frac{p}{(2n+3+(a+1)p)^{1-\frac{1}{p}}} \right)^{2n}\left(\frac{1}{2n+3+(a+1)p)} \right)^{1-\frac{2}{p}}
\end{align*}
which tends to zero as $n$ tends to infinity, uniformly in $m$. This concludes the proof of the lemma.
\end{proof}

\section{Geometric information in sharp large deviation estimates}\label{subs-geoinfo}
Fix $p\in (1,\infty)$ and $n\in\mathbb{N}$. We now demonstrate how sharp large deviation estimates encode geometric properties of the underlying high-dimensional measure. First observe from the estimate in~\eqref{tail_prob} that the leading order term that depends on $\theta$ is $R^n_a(\theta^n)$, which, in turn, depends on $\theta$ only through $\Psi^n_{p,\theta}(\lambda_a)$, as evident from its definition in ~\eqref{Rn}. From the definitions in~\eqref{logmgf_lp},~\eqref{lambda_x}and~\eqref{psin_p}, we have
\begin{equation}\label{first-term-Rn}
\Psi^n_{p,\theta}(\lambda_a) = \frac{1}{n}\sum_{j=1}^n\Lambda_p\left(\sqrt{n}\theta^n_j\lambda_{a,1},\lambda_{a,2} \right),
\end{equation}
where we suppress the $\theta^n$ dependence in $\Psi^n_{p,\theta}$.
We first state a lemma regarding the properties of $\Lambda_p$ in~\cite{GanKimRam17}.
\begin{lemma}~\cite[Lemma 7.5]{GanKimRam17} \label{con-Lambdap}
Let $p\in(1,\infty)$ and $t_2<1/p$. The map $\mathbb{R}_+\ni t_1\mapsto\Lambda_p(\sqrt{t_1},t_2)$ is concave but not linear for $p>2$, linear for $p=2$ and convex but not linear for $p<2$.
\end{lemma}

\begin{proof}[Proof of Proposition \ref{p-dependence}]
From the definition of $R^n_a(\theta^n)$ in \eqref{Rn}, it suffices to understand the behavior of \eqref{first-term-Rn}.
Since by \eqref{lambda_x}, $\lambda_{a,2}<1/p$. We may apply Lemma~\ref{con-Lambdap} in the following proof.

First, for $p=2$, from~\eqref{logmgf_lp} and \eqref{lambda-est}, a simple calculation yields
$$\Lambda_2(\lambda_{a,1},\lambda_{a,2})=-\frac{1}{2}\log (1-2\lambda_{a,2})+\frac{1}{2}\frac{\lambda_{a,1}^2}{1-2\lambda_{a,2}}.$$
Hence, by~\eqref{first-term-Rn} and the last display,
$\Psi^n_{p,\theta}(\lambda_a)$ does not depend on $\theta$ and thus is a constant.

Next, consider $p>2$. By~\ref{con-Lambdap}, $\Lambda_p\left(\sqrt{\cdot},\lambda_{a,2}\right)$ is concave but not linear. By the definition of $\Lambda_p$ in~\eqref{logmgf_lp} and the symmetry of the $p$-Gaussian distribution~\eqref{pNormal}, $\Lambda_p(\cdot,\lambda_{a,2})$ is an even function. Therefore, for $\theta^n\in\mathbb{S}^{n-1}$,
\begin{align*}
\frac{1}{n}\sum_{j=1}^n\Lambda_p\left(\sqrt{n}\theta^n_j\lambda_{a,1},\lambda_{a,2} \right) & = 
 \frac{1}{n}\sum_{j=1}^n\Lambda_p\left(\sqrt{n(\theta^n_j)^2(\lambda_{a,1})^2},\lambda_{a,2} \right)\\
 & \leq \Lambda_p\left(\sqrt{\sum_{j=1}^n(\theta^n_j)^2(\lambda_{a,1})^2},\lambda_{a,2}\right)\\
 & = \Lambda_p\left(\sqrt{n}\lambda_{a,1},\lambda_{a,2} \right).
\end{align*}
Moreover, since $\Lambda_p\left(\sqrt{\cdot},\lambda_{a,2}\right)$ is not linear, the equality in the last display holds only when
\[
(\theta^n_1)^2=(\theta^n_2)^2=\cdots=(\theta^n_n)^2=\frac{1}{n}.
\]
Thus, we conclude that the maximum of $\Psi^n_{p,\theta}(\lambda_a)$ is attained at $(\pm 1,\pm1,\ldots,\pm1)/\sqrt{n}$.

On the other hand, to identify the minimizers of $\theta^n\mapsto\Psi^n_{p,\theta}(\lambda_a)$, note from Lemma~\ref{con-Lambdap} and the fact that $\Lambda_p(\cdot,\lambda_{a,2})$ is even, we can write
$
\Psi^n_{p,\theta}(\lambda_a)=\mathcal{F}(\theta^n_1,\ldots,\theta^n_n),
$
where $\mathcal{F}=\mathcal{F}_a$ is defined to be
\begin{align}\label{psi-F}
\mathcal{F}(t_1,t_2,\ldots,t_n) : = \frac{1}{n}\sum_{j=1}^n\Lambda_p\left(\sqrt{nt_j(\lambda_{a,1})^2},\lambda_{a,2}\right).
\end{align}
for $(t_1,\ldots,t_n)$ lies in the compact domain $$\mathcal{A}:=\left\{(t_1,t_2,\ldots,t_n)\in\mathbb{R}^n_+:\sum_{j=1}^nt_j=1\right\}.$$
Since $\mathcal{F}$ is strictly concave by Lemma~\ref{con-Lambdap}, the minimum of $\mathcal{F}$ is obtained at the extreme points of $\mathcal{A}$, namely, the vectors, $\pm e_j, j=1,\ldots,n$.
Thus, by \eqref{psi-F}, the minimum of $\Psi^n_{p,\theta}(\lambda_a)$ is also attained at
\[
\theta^n=\pm e_j,\quad \text{for}\quad j=1,\ldots,n.
\]

The case $p<2$ follows from the same argument on interchanging maxima and minima, and invoking now the convexity of $t_1\mapsto\Lambda_p(\sqrt{t_1},t_2)$ from Lemma~\ref{con-Lambdap}.
\end{proof}

\section{Proof of Proposition~\ref{asymptotic-exp}}\label{app-asymptotic}
\begin{proof}
Fix $m$, $d$, $D$, $h^n$ and $(f,x^*,\alpha,g^n)$ as in the proposition. By setting $\tilde{f}(x) = f(x)-f(x^*)$,
without loss of generality, we assume $f(x^*)=0$.
For $k\in\N$ and any multi-index $\beta=(\beta_1,\ldots,\beta_{k})\in\N^k$, we define $f_\beta:=\partial_{\beta_1,\ldots,\beta_{k}}^{\abs{\beta}}f(x^*)$. 
Since $D\subset \R^m\times \R^d_+$ and $x^*=(0,0,\ldots,0)$, $x^*$ lies in the boundary of $D$. Moreover,
since $f$ is twice continuously differentiable in $D$ and its minimum over $\cl(D)$ is attained at $x^*$, we have $\nabla f(x^*)\cdot (a_1,\ldots,a_m,0,\ldots,0)=0$ for all $a_i\in\R$, $i=1,\ldots,m$, which implies that $f_i= 0$ for $i=1,\ldots,m$. Moreover, since $f$ achieves its minimum uniquely at $x^*$, $f_{ii}>0$ for $i=1,\ldots,m$. 
By Taylor's theorem, we may then write $f$ as 
\[
f(x) = \sum_{i=1}^m\frac{f_{i,i}}{2}x_i^2(1+P_i(x))+\sum_{i=m+1}^df_{i}x_{i}(1+P_{i}(x)),  \quad \text{for $x\in D$}
\]
where $(P_i)_{i=1,\ldots,m+d}$ are continuously differentiable real-valued functions on $D$.
We will proceed by making several changes of variables. We start with the transformation $T_1:\R^m\times\R^d_+\to \R^m\times\R^d_+$ where $u=T_1(x)$ is defined to be
\[
u_{i}=x_{i}(1+P_{i}(x))^{1/2}, \quad i=1,\ldots,m \quad \text{and}\quad u_i = x_i(1+P_i(x)),\quad i=m+1,\ldots,m
+d.\]
Note that the Jacobian $\mathcal{J}_1$ of $T_1$ is $1$ for $x\in D$. Let $u^*:=T_1(x^*)=(0,\ldots,0)$. 
Let $D':=T_1(D)$ and note that $D'\subset \R^m\times\R^d_+$ contains a neighborhood of the origin. For $u \in D'$, define
\begin{equation}\label{eq-F}
F(u) := f(T_1^{-1}(u))=\sum_{i=1}^{m}\frac{f_{i,i}}{2}u_{i}^2+\sum_{i=m+1}^{m+d}f_iu_i.
\end{equation}

Next, let $\mstar$ be the maximum of $F$ in $D'$ or equivalently, of  $f$ in $D$ and define $G^n(u):=g^n(T_1^{-1}(u))$. Using the change of variables  $T_1$, we see that 
\begin{align}
   I^n & := \int_Dh^n(x)dx  = \int_{D'} G^n(u)e^{-nF(u)}du. \label{change-uvw}
\end{align} 
  Since $D'\subset\R^m\times\R^d_+$, we see that $D'$ is covered by the family of surfaces $F(u)=t$ for $t\in[0,\mstar]$, that is, $D'\subset\cup_{t\in[0,\mstar]}\{u\in\R^{k+1}:F(u)=t \}$.
 Note also that $D'$ is bounded and $\nabla F$ is nonzero in $D'$. Using the method of resolution of
multiple integrals~\cite[Theorem 9, Chapter V]{Wong01}, we then have 
\begin{equation}\label{resolution}
I^n = \int_0^\mstar \mathfrak{r}^n(t)e^{-nt}dt, 
\end{equation}
where
\begin{align}\label{r}
\mathfrak{r}^n(t) & := \int_{\{F(u)=t\}}\frac{G^n(u)}{\sqrt{\sum_{i=1}^{k+1}F_{u_i}^2}}dA, \quad t \in [0,\mstar],
\end{align}
with $dA$ denoting the surface element of the surface $F(u)=t$.

To further simplify the integral in \eqref{resolution}-\eqref{r}, we introduce an additional change of variables $T_2: \R^m\times\R^d_+\to \R_+\times[\pi/2,\pi,2]^m\times[0,\pi/2]^{d-1}$ by
letting $(\xi,\theta,\phi)=T_2(u)$ be such that
\begin{align}\label{change-polar}
u_i &= \left(\frac{2\xi}{f_{i,i}}\right)^{1/2}\cos\theta_1\cdots\cos\theta_{i-1}\sin\theta_i,\quad& i = 1,\ldots,m, \nonumber\\
u_{m+i} & = \frac{\xi}{f_{m+i}}\cos^2\theta_1\cdots\cos^2\theta_m\cos^2\phi_{1}\cdots\cos^2\phi_{i-1}\sin^2\phi_{i},\quad& i = 1,\ldots,d-1,\nonumber\\
u_{m+d} & = \frac{\xi}{f_{m+d}}\cos^2\theta_1\cdots\cos^2\theta_m\cos^2\phi_{1}\cdots\cos^2\phi_{d-2}\cos^2\phi_{d-1},
\end{align}
for $\theta_i\in[-\pi/2,\pi/2]$, $i=1,\ldots,m$, $\phi_i\in[0,\pi/2]$, $i=1,\ldots,d-1$ and $\xi\in[0,\mstar]$. Since for $i=m+1,\ldots,m+d$, $u_i\in\R_+$, the domain of $\phi_i$ is restricted to $[0,\pi/2]$, and thus we may take the square in cosines and sines for $i=m+1,\ldots,m+d$. Therefore, $T_2$ is a modified version of polar coordinates and is well defined. 
From \eqref{eq-F} and \eqref{change-polar}, $F(u) = \xi$ for $u\in\R^{m+d}$ and  the Jacobian $\mathcal{J}_2$ of $T_2$ is
\begin{align}
\mathcal{J}_2&:=\frac{\partial(u_1,\cdots,u_{m+d})}{\partial(\xi,\theta_1,\cdots,\theta_{m},\phi_1,\cdots,\phi_{d-1})}\nonumber\\
&=\frac{(2\xi)^{m/2+d-1}}{f_{m+1}\cdots f_{m+d}\sqrt{f_{1,1}\cdots f_{m,m}}}\prod_{i=1}^m\cos^{2d+m-1-i}\theta_i\prod_{i=1}^{d-1}\cos^{2d-1-2i}\phi_i\sin\phi_i.\label{jacobian}
\end{align}
From the second change of variables in~\eqref{change-polar}, we have 
\begin{align}\label{eq-J2}
\frac{dA}{\norm{\nabla F}}=\mathcal{J}_2d\theta d\phi.
\end{align}
Recall $g^n(x) = e^{r^n(x)}$ and $G^n(u) = g^n(T_1^{-1}(u))$. Let $\hat{r}^n(\xi,\theta,\phi)$ denotes the transformation of $r^n$ under $T_2\circ T_1$. Then, \eqref{r}, \eqref{jacobian} and \eqref{eq-J2} imply that
\begin{align}
\mathfrak{r}^n(t) & = \frac{(2t)^{m/2+d-1}}{f_{m+1}\cdots f_{m+d}\sqrt{f_{1,1}\cdots f_{m,m}}}\\
&\qquad \times\int_{\theta\in[-\pi/2,\pi/2]^m}\int_{\phi\in[0,\pi/2]^{d-1}}e^{\hat{r}^n(t,\theta,\phi)}\prod_{i=1}^m\cos^{2d+m-1-i}\theta_i\prod_{i=1}^{d-1}\cos^{2d-1-2i}\phi_i\sin\phi_id\phi d\theta\nonumber\\
&= \frac{(2t)^{m/2+d-1}}{f_{m+1}\cdots f_{m+d}\sqrt{f_{1,1}\cdots f_{m,m}}} e^{\tilde{r}^n(t)},\label{rr}
\end{align}
where 
\begin{align}\label{tilder}
\tilde{r}^n(t) = \log \int_{\theta\in[-\pi/2,\pi/2]^m}\int_{\phi\in[0,\pi/2]^{d-1}}e^{\hat{r}^n(t,\theta,\phi)}\prod_{i=1}^m\cos^{2d+m-1-i}\theta_i\prod_{i=1}^{d-1}\cos^{2d-1-2i}\phi_i\sin\phi_id\phi d\theta.
\end{align}
Since $\abs{r^n(x)}\leq Cn^\alpha\norm{x}$ for $n$ large and $x$ in a neighborhood of the origin, there exist $\varepsilon >0$ and $\tilde{r}^n(t)$ such that  $\abs{\tilde{r}^n(t)}\leq Cn^\alpha t$ for $n$ large and $t\in(0,\varepsilon)$.

From \eqref{resolution} and \eqref{rr}, we observe that 
\begin{align*}
I^n &= \frac{2^{m/2+d-1}}{f_{m+1}\cdots f_{m+d}\sqrt{f_{1,1}\cdots f_{m,m}}}\int_0^\mstar t^{m/2+d-1}e^{-nt+\tilde{r}^n(t)}dt.
\end{align*}
Now, applying \cite[Chapter 9, Theorem 2.1]{Olv97} with $p$, $r$, $q$, $\lambda$, $\mu$, and $\nu$ being $t$, $\tilde{r}^n$, $t^{m/2+d-1}$, $m/2+d$, $1$ and $1$, we obtain
\begin{align*}
I^n 
&= \frac{2^{m/2+d-1}}{f_{m+1}\cdots f_{m+d}\sqrt{f_{1,1}\cdots f_{m,m}}}\prod_{i=1}^m\int_{-\pi/2}^{\pi/2}\cos^{2d+m-1-i}\theta d\theta \prod_{i=1}^{d-1}\int_0^{2\pi}\cos^{2d-1-2i}\phi \sin\phi d\phi \frac{\Gamma\left(m/2+d\right)}{n^{m/d+2}}\\
&=\frac{(2\pi)^{m/2}g^n(x^*)}{n^{d+m/2}f_{m+1}\cdots f_{m+d}\sqrt{f_{1,1}\cdots f_{m,m}}}(1+o(1)).
\end{align*}
\end{proof}

\section{A uniform deviation estimate} \label{app-GC-class}
We now establish Lemma \ref{lem-conti-K}. Key ingredients of  the proof  include the Gaussian concentration inequality and certain deviation estimates that are uniform with respect to a class of functions, much in the  spirit of uniform Glivenko-Cantelli or Donsker classes.
\begin{proof}[Proof of Lemma \ref{lem-conti-K}]

Fix $\varepsilon>0$.
Also, consider $x\in\R$ and $(t_1,t_2)\in\mathbb{D}$. We will repeatedly use the fact that $(xt_1,t_2)$ lies in $\mathbb{D}$.
By the assumed differentiability properties of $\mathcal{K}$, an application of Taylor's theorem shows that
\begin{align}\label{eq-taylor}
\mathcal{K}(xt_1,t_2) = \mathcal{K}(0,t_2)+\partial_1\mathcal{K}(\rho(xt_1,t_2),t_2)xt_1,
\end{align}
where $\rho:\mathbb{D}\to\R$ is a function that satisfies
\begin{align}\label{eq-rho_ineq}
\abs{\rho(y,t)}\leq\abs{y} \quad \text{for} \quad (y,t)\in \mathbb{D}.
\end{align}
 $\abs{\rho(xt_1,t_2)}\leq\abs{xt_1}$. By the polynomial growth assumption on the partial derivatives of $\mathcal{K}$, there exist $q$, $\tilde{C}\in(0,\infty)$ such that
\begin{align}
\sup_{t\in\mathbb{D}, \norm{t}<\varepsilon}\abs{\partial_1\mathcal{K}(xt_1,t_2)}&\leq \tilde{C}(1+\varepsilon^q\abs{x}^q),\label{eq-polybound}\\
\sup_{t\in\mathbb{D}, \norm{t}<\varepsilon}\abs{\partial_{1j}\mathcal{K}(xt_1,t_2)}&\leq  \tilde{C}(1+\varepsilon^q\abs{x}^q),\quad \text{for} \quad j=1,2.\label{eq-polybound1}
\end{align}
By \eqref{eq-taylor}, for $t\in\mathbb{D}$, $\norm{t}<\varepsilon$, we see that
\begin{align}
\sum_{j=1}^n\left(\mathcal{K}\left(\frac{\sqrt{n}Z_j}{\norm{Z^{(n)}}}t_1,t_2\right)-\mathbb{E}\left[\mathcal{K}(Zt_1,t_2)\right]  \right)=t_1\left(I^n_1(t)+I^n_2(t)\right),\label{eq-111}
\end{align}
where 
\begin{align}
I^n_1(t)&:=\sum_{j=1}^n\left(\partial_1\mathcal{K}\left(\rho\left(\frac{\sqrt{n}Z_j}{\norm{Z^{(n)}}}t_1,t_2\right),t_2\right)\frac{\sqrt{n}Z_j}{\norm{Z^{(n)}}}-\partial_1\mathcal{K}\left(\rho\left(Z_jt_1,t_2\right),t_2\right)Z_j\right),\label{eq-In1.5}\\
I^n_2(t)&:=\sum_{j=1}^n\left(\partial_1\mathcal{K}\left(\rho\left(Z_jt_1,t_2\right),t_2\right)Z_j-\mathbb{E}\left[\partial_1\mathcal{K}(\rho(Zt_1,t_2),t_2)Z\right]\right).\label{eq-In2}
\end{align}

The proof follows in several steps.

\noindent{\bf Step 1.} We claim that for $\alpha\in(1/2,1)$,
 there exist $C_1\in(0,\infty)$ and a random integer $N_1$ such that $\mathbb{P}'$-almost surely,
\begin{align}\label{eq-wts1}
\sup_{t\in\mathbb{D}, \norm{t}<\varepsilon}\abs{n^{-\alpha}I^n_1(t)}\leq C_1 \quad \text{for} \quad n\geq N_1.
\end{align}

\noindent{\bf Proof of claim of Step 1.} Note from \eqref{eq-taylor} that for $x\in\R$, and $t\in\mathbb{D}$, $\|t\|<\varepsilon$, 
\[
\frac{d}{dx}\left(\partial_1\mathcal{K}(\rho(xt_1,t_2),t_2)x\right) =
\begin{dcases}
\frac{d}{dx}\left(\frac{\mathcal{K}(xt_1,t_2) - \mathcal{K}(0,t_2)}{t_1}\right)=\partial_1\mathcal{K}(xt_1,t_2),\quad &\text{if $t_1\neq 0$},\\
\partial_1\mathcal{K}(0,t_2),\quad &\text{if $t_1=0$.}
\end{dcases}
\]
Together with \eqref{eq-polybound}, this implies that for $x\in\R$,
\begin{align*}
\sup_{t\in\mathbb{D},\norm{t}<\varepsilon}\abs{\frac{d}{dx}\left(\partial_1\mathcal{K}(\rho(xt_1,t_2),t_2)x\right)}=\sup_{t\in\mathbb{D},\norm{t}<\varepsilon}\abs{\partial_1\mathcal{K}(xt_1,t_2)}&\leq \tilde{C}(1+\varepsilon^q\abs{x}^q).
\end{align*}
Combining the last two displays, we see that
\begin{align}
\sup_{t\in\mathbb{D}, \norm{t}<\varepsilon}\abs{n^{-\alpha}I^n_1(t)}
&\leq n^{-\alpha}\sup_{t\in\mathbb{D}, \norm{t}<\varepsilon} \sum_{j=1}^n \sup_{x\in \left[\frac{\sqrt{n}Z_j}{\norm{Z}^{(n)}},Z_j\right]}\partial_1\abs{\mathcal{K}(xt_1,t_2)}\abs{\frac{\sqrt{n}Z_j}{\norm{Z^{(n)}}}-Z_j}\nonumber\\
& \leq \abs{\frac{(\sqrt{n}-\norm{Z^{(n)}})n^{1/2-\alpha}}{\norm{Z^{(n)}}/\sqrt{n}}} \frac{\tilde{C}}{n}\sum_{j=1}^n\left(1+\varepsilon^q\max\left(\frac{\sqrt{n}}{\norm{Z^{(n)}}},1\right)^q\abs{Z_j}^q\right)\abs{Z_j}.\label{eq-up0}
\end{align}

Since $(Z_j)_{j\in\N}$ are independent, by the strong law of large numbers, $\mathbb{P}'$-almost surely, as $n\to\infty$,
\begin{align}\label{eq-up2}
\frac{1}{n}\sum_{j=1}^n(1+(\tilde{C}\varepsilon)^q\abs{Z_j}^q)\abs{Z_j}\to\mathbb{E}\left[(1+(\tilde{C}\varepsilon)^q\abs{Z}^q)\abs{Z}\right] \quad\text{and} \quad  \frac{\norm{Z^{(n)}}}{\sqrt{n}}\to 1.
\end{align}
Furthermore, the Gaussian concentration inequality~\cite[Theorem 3.1.1]{Vershynin18}, implies that
there exists $c\in(0,\infty)$ such that
\begin{align*}
\mathbb{P}\left( n^{1/2-\alpha} \abs{\norm{Z^{(n)}}-\sqrt{n}}>\tilde{C} \right) 
= \mathbb{P}\left( \abs{\norm{Z^{(n)}}-\sqrt{n}}> n^{\alpha-1/2}\tilde{C} \right)
 \leq 2 e^{-c(\tilde{C})^{2}n^{2\alpha-1}},
\end{align*}
which is summable because $\alpha>1/2$. Hence, by the Borel-Cantelli lemma and the second limit in \eqref{eq-up2}, there exists a random integer $N'_1\in\N$ such that $\mathbb{P}'$-almost surely,
\begin{align}\label{eq-up1}
n^{1/2-\alpha} \abs{\norm{Z^{(n)}}-\sqrt{n}}\leq \tilde{C} \quad \text{and} \quad \max\left(\frac{\sqrt{n}}{\norm{Z^{(n)}}},1 \right)\leq \tilde{C},\quad n\geq N'_1.
\end{align} 
The claim of Step 1 then follows from \eqref{eq-up0}, \eqref{eq-up2}, and \eqref{eq-up1}.\\

\noindent{\bf Step 2.} We now establish a bound on $I_2^n(t)$ defined in \eqref{eq-In2}.
Specifically we show that there exist $C_2\in(0,\infty)$ and a random integer $N_2\in\N$ such that $\mathbb{P}'$-almost surely, for $n\geq N_2$,
\begin{align}\label{eq-wts2}
\sup_{\norm{t}<\varepsilon}\abs{n^{-\alpha}I^n_2(t)}\leq C_2, \quad \text{for} \quad n\geq N_2.
\end{align}

Before proving this bound, we first show how when combined with Step 1, this proves the lemma.
Indeed, \eqref{eq-111}, \eqref{eq-wts1} and \eqref{eq-wts2} together show that there  exist $C_1,C_2\in(0,\infty)$ such that almost surely, for $t\in\mathbb{D}$, $\norm{t}\leq \varepsilon$ and $n\geq \max\{N_1,N_2\}$,
\begin{align*}
\abs{\sum_{j=1}^n\mathcal{K}\left(\frac{\sqrt{n}Z_j}{\norm{Z^{(n)}}}t_1,t_2\right)-\mathbb{E}\left[\mathcal{K}(Zt_1,t_2)\right] }\leq n^\alpha\abs{t_1}(C_1+C_2)\leq (C_1+C_2)n^\alpha\norm{t}. 
\end{align*}
This implies that \eqref{eq-wts} holds $\mathbb{P}'$-almost surely and concludes the proof of the lemma.

To complete the proof of the lemma, it only remains to prove the bound in Step 2.

\noindent{\bf Proof of bound in Step 2.}
To prove \eqref{eq-wts2}, we introduce a suitable truncation of $(x,(t_1,t_2))\mapsto\partial_1\mathcal{K}(\rho(xt_1,t_2),t_2)x$.
To this end, recall the definition of $q$ in \eqref{eq-polybound}. Since $\alpha>1/2$, we may choose $\beta>0$ so that
\begin{align}\label{eq-r}
2\alpha-1-2\beta q>0 \quad \text{and define} \quad r_n:= n^\beta.
\end{align}
Then for $x\in\R$ and $t\in\mathbb{D}$, $\norm{t}<\varepsilon$, define
\begin{align}\label{eq-M}
\mathbb{T}_n(x,t):= 
\begin{cases}
\partial_1\mathcal{K}\left(\rho\left(xt_1,t_2\right),t_2\right)x, & \text{if }  \abs{x}<r_n,\\
\partial_1\mathcal{K}\left(\rho\left(\text{sgn}(x)r_n t_1,t_2\right),t_2\right)\text{sgn}(x)r_n, & \text{if }  \abs{x}\geq r_n,
\end{cases}
\end{align}
where $\text{sgn}:\R\to\{-1,1\}$ is defined by $\text{sgn}(x)=1$ if $x\geq 0$ and $\text{sgn}(x)=-1$ if $x<0$.
We bound \eqref{eq-wts2} above by the sum of three terms:
\begin{align}\label{eq-123}
\sup_{\norm{t}<\varepsilon}\abs{n^{-\alpha}I^n_2(t)}\leq  I^n_{21}+I^n_{22}+I^n_{23},
\end{align}
where
\begin{align}
I^n_{21}&:=\sup_{t\in\mathbb{D}, \norm{t}<\varepsilon}\abs{\frac{1}{n^\alpha}\sum_{j=1}^n\left(\mathbb{T}_n(Z_j,t)-\mathbb{E}[\mathbb{T}_n(Z,t)] \right)]},\label{eq-I1}\\
I^n_{22}&:=\sup_{t\in\mathbb{D},\norm{t}<\varepsilon}\abs{\frac{1}{n^\alpha}\sum_{j=1}^n\left(\partial_1\mathcal{K}(\rho(Z_jt_1,t_2),t_2)Z_j- \mathbb{T}_n(Z_j,t)\right)},\label{eq-I2}\\
I^n_{23}&:=\sup_{t\in\mathbb{D},\norm{t}<\varepsilon}\abs{\mathbb{E}\left[\frac{1}{n^\alpha}\sum_{j=1}^n\left(\partial_1\mathcal{K}(\rho(Z_jt_1,t_2),t_2)Z_j- \mathbb{T}_n(Z_j,t)\right)\right]}\label{eq-I3}.
\end{align} 
We now treat each of these terms individually.

\noindent{\bf Step 2A.}
For the first term $I^n_{21}$ in \eqref{eq-I1}, we start by proving that there exist $C_3<\infty$ and a random integer $N_3$ such that $\mathbb{P}'$-almost surely,
\begin{align}\label{eq-b1}
I^n_{21}\leq C_3+2\varepsilon \quad \text{for} \quad n\geq N_3.
\end{align}

\noindent{\bf Proof of Step 2A bound.} The proof of \eqref{eq-b1} starts with the following claim.

\noindent{\bf Claim.}
{\it For $x\in\R$ and $t\in\mathbb{D}$, $\norm{t}<\varepsilon$, $t\mapsto\mathbb{T}_n(x,t)$ is Lipschitz continuous with constant $\tilde{C}(1+\varepsilon^qr_n^q)r_n^2$.
}\\
{\bf Proof of the claim.}
Define 
\begin{align*}
x_n:=
\begin{cases}
x, & \text{if} \quad x<r_n,\\
\text{sgn}(x)r_n, &\text{if} \quad x\geq r_n.
\end{cases}
\end{align*}
First, note that since $\mathcal{K}$ is twice continuously differentiable, by \eqref{eq-taylor} and \eqref{eq-M}, for $t\in\mathbb{D}$, $\norm{t}<\varepsilon$ and $t_1\neq 0$,
\begin{align*}
\partial_{t_1}\mathbb{T}_n(x,t) =
\partial_{t_1}\left(\frac{\mathcal{K}(xt_1,t_2) - \mathcal{K}(0,t_2)}{t_1}\right)&=\frac{xt_1\partial_{1}\mathcal{K}(xt_1,t_2)-\mathcal{K}(xt_1,t_2)+\mathcal{K}(0,t_2)}{t_1^2}\\
&= -\frac{1}{2}\xi^2\partial_{11}\mathcal{K}(\xi,t_2),
\end{align*}
where the second equality follows from Taylor's theorem with $\xi$ being a constant such that $\abs{\xi}\leq \abs{x_nt_1}\leq \varepsilon r_n$. Likewise, by \eqref{eq-rho_ineq}, $\rho(0,t_2)=0$, and for  $\in\mathbb{D}$, $\norm{t}<\varepsilon$ and $t_1= 0$,
\begin{align*}
\left.\partial_{t_1}\mathbb{T}_n(x,t)\right|_{t_1=0} &=
\lim_{t_1\to0}\left(\frac{\partial_1\mathcal{K}\left(\rho\left(xt_1,t_2\right),t_2\right)x -\partial_1\mathcal{K}\left(0,t_2\right)x}{t_1}\right)\\&=\lim_{t_1\to0}\frac{\mathcal{K}(xt_1,t_2)-\mathcal{K}(0,t_2)-\partial_1\mathcal{K}(0,t_2)xt_1}{t_1^2}\\
&= \frac{1}{2}x^2\partial_{11}\mathcal{K}(0,t_2).
\end{align*}
By \eqref{eq-polybound1} and the last two displays, for $t\in\mathbb{D}$, $\norm{t}<\varepsilon$, we have
\begin{align}\label{eq-lip1}
\abs{\partial_{t_1}\mathbb{T}_n(x,t)} \leq \tilde{C}(1+\varepsilon^qr_n^q)r_n^2.
\end{align}
Similarly, for $x\in\R$ and $t\in\mathbb{D}$, $\norm{t}<\varepsilon$ by \eqref{eq-taylor}, \eqref{eq-M} and Taylor's theorem, when $t_1\neq 0$,
\begin{align*}
\partial_{t_2}\mathbb{T}_n(x,t) =
\partial_{t_2}\left(\frac{\mathcal{K}(xt_1,t_2) - \mathcal{K}(0,t_2)}{t_1}\right)&=\frac{\partial_{2}\mathcal{K}(xt_1,t_2)-\partial_2\mathcal{K}(0,t_2)}{t_1}\\
&= \xi\partial_{12}\mathcal{K}(\xi,t_2),
\end{align*}
where $\xi$ is a constant such that $\abs{\xi}\leq \abs{x_nt_1}$ and when $t_1= 0$,
\begin{align*}
\partial_{t_2}\mathbb{T}_n(x,t) =
\partial_{t_2}\left(\partial_1\mathcal{K}\left(0,t_2\right)x \right)&=x\partial_{12}\mathcal{K}(0,t_2).
\end{align*}
The last two displays and \eqref{eq-polybound1} imply that
\begin{align}\label{eq-lip2}
\abs{\partial_{t_2}\mathbb{T}_n(x,t)} \leq \tilde{C}(1+\varepsilon^qr_n^q)r_n, \quad \text{for $x\in\R$ and $\norm{t}\leq \varepsilon$}.
\end{align}
Thus, the claim follows from \eqref{eq-lip1} and \eqref{eq-lip2}. \hfill\qedsymbol{}

We now continue with the proof of Step 2A.
For $n\in\N$, let $\delta_n$ and $k_n$ be finite positive constants given by 
\begin{align}\label{eq-delta}
\delta_n:=\frac{\varepsilon}{\tilde{C}(1+\varepsilon^qr_n^q)r_n^2n^{1-\alpha}}\quad  \text{and}\quad  k_n:=\left\lceil \left(\tilde{C}(1+\varepsilon^qr_n^q)r_n^2n^{1-\alpha}\right)^2 \right\rceil.
\end{align}
Given the claim, there exist $(l_k)_{k=1,\ldots,k_n}\subset \{t\in\mathbb{D}:\norm{t}\leq \varepsilon\}$ such that $\cup_{k=1}^{k_n}B_{\delta_n}(l_k)\supset \{t\in\mathbb{D}:\norm{t}\leq \varepsilon\}$ and for $x\in\R$, 
\begin{align}\label{eq-lip}
\abs{\mathbb{T}_n(x,u)-\mathbb{T}_n(x,v)}\leq \norm{u-v}\tilde{C}(1+\varepsilon^qr_n^q)r_n^2\leq \frac{2\varepsilon}{n^{1-\alpha}} \quad \text{for} \quad u,v\in B_{\delta_n}(l_k),
\end{align}
where the last inequality uses $\norm{u-v}\leq 2\delta_n$ and \eqref{eq-delta}.
Together with the expression for $I^n_{21}$ in \eqref{eq-I1}, this shows that
\begin{align}
I^n_{21} &\leq \sup_{k=1,\ldots,k_n}\sup_{t\in B_{\delta_n}(l_k)}\abs{\frac{1}{n^\alpha}\sum_{j=1}^n\left(\mathbb{T}_n(Z_j,t)-\mathbb{E}[\mathbb{T}_n(Z,t)] \right)}\nonumber\\
&\leq \sup_{k=1,\ldots,k_n}\abs{\frac{1}{n^\alpha}\sum_{j=1}^n\left(\sup_{t\in B_{\delta_n}(l_k)}\mathbb{T}_n(Z_j,t)-\mathbb{E}\left[\sup_{t\in B_{\delta_n}(l_k)}\mathbb{T}_n(Z_j,t)\right]\right)}\nonumber\\
&\qquad +\sup_{k=1,\ldots,k_n}n^{1-\alpha}\abs{\mathbb{E}\left[\sup_{t\in B_{\delta_n}(l_k)}\mathbb{T}_n(Z,t)-\inf_{t\in B_{\delta_n}(l_k)}\mathbb{T}_n(Z,t)\right]}\nonumber\\
&\qquad + \sup_{k=1,\ldots,k_n} n^{1-\alpha} \abs{\sup_{t\in B_{\delta_n}(l_k)}\mathbb{T}_n(Z,t)-\inf_{t\in B_{\delta_n}(l_k)}\mathbb{T}_n(Z,t)}\nonumber\\
&\leq \sup_{k=1,\ldots,k_n}\abs{\frac{1}{n^\alpha}\sum_{j=1}^n\left(\sup_{t\in B_{\delta_n}(l_k)}\mathbb{T}_n(Z_j,t)-\mathbb{E}\left[\sup_{t\in B_{\delta_n}(l_k)}\mathbb{T}_n(Z_j,t)\right]\right)}+4\varepsilon,\label{eq-I11}
\end{align}
where the last inequality follows from \eqref{eq-lip}.
In addition, note that, together, \eqref{eq-polybound} and \eqref{eq-M} imply
\[
\abs{\sup_{t\in B_{\delta_n}(l_k)}\mathbb{T}_n(Z_j,t)} \leq \tilde{C}(1+\varepsilon^qr_n^q).
\]
Hence, by the union bound and Hoeffding's inequality~\cite[Theorem 2.2.6]{Vershynin18}, which is applicable since $(Z_j)_{j\in\N}$ are i.i.d., for any $n\in\N$ and $C_3\in(0,\infty)$ we have
\begin{align*}
&\mathbb{P}\left( \sup_{k=1,\ldots,k_n}\abs{\frac{1}{n^\alpha}\sum_{j=1}^n\left(\sup_{t\in B_{\delta_n}(l_k)}\mathbb{T}_n(Z_j,t)-\mathbb{E}\left[\sup_{t\in B_{\delta_n}(l_k)}\mathbb{T}_n(Z_j,t)\right]\right)}>C_3\right)\\
&\qquad \leq \sum_{k=1}^{k_n}\mathbb{P}\left( \abs{\frac{1}{n^\alpha}\sum_{j=1}^n\left(\sup_{t\in B_{\delta_n}(l_k)}\mathbb{T}_n(Z_j,t)-\mathbb{E}\left[\sup_{t\in B_{\delta_n}(l_k)}\mathbb{T}_n(Z_j,t)\right]\right)}>C_3\right)\\
&\qquad \leq 2k_n\exp\left(-\frac{2C_3^2n^{2\alpha}}{n\tilde{C}^2(1+\varepsilon^qr_n^q)^2} \right),
\end{align*}
which is summable in $n$ by \eqref{eq-delta} and \eqref{eq-r}. Step 2A then follows from the Borel-Cantelli lemma and \eqref{eq-I11}.

\noindent{\bf Step 2B.}
Next, we deal with the quantity $I^n_{22}$ in \eqref{eq-I2} and show that 
there exists $C_4<\infty$ and a random integer $N_5$ such that $\mathbb{P}'$-almost surely, for $n\geq N_5$,
\begin{align}\label{eq-b2}
I^n_{22}\leq C_4, \quad \text{and} \quad I^n_{23}\leq C_4.
\end{align}

\noindent{\bf Proof of Step 2B bounds.}
Note that by \eqref{eq-I2}, \eqref{eq-M}, \eqref{eq-polybound} and \eqref{eq-rho_ineq},
\begin{align*}
I^n_{22}\leq \frac{1}{n^\alpha}\sum_{i=1}^n \left(\abs{Z_j}\tilde{C}(1+\varepsilon^q\abs{Z_j}^q)+r_n\tilde{C}(1+\varepsilon^qr_n^q)\right)1_{\{\abs{Z_j}>r_n\}}.
\end{align*}
Hence, Markov's inequality and the fact that $(Z_j)_{j\in\N}$ are i.i.d. imply that for any $C_4\in(0,\infty)$,
\begin{align*}
\mathbb{P}(I^n_{22}>C_4)&\leq \mathbb{P}\left(  \frac{1}{n^\alpha}\sum_{i=1}^n \left(\abs{Z_j}\tilde{C}(1+\varepsilon^q\abs{Z_j}^q)+r_n\tilde{C}(1+\varepsilon^qr_n^q)\right)1_{\{\abs{Z_j}>r_n\}} >C_4\right)\\
&\leq \frac{1}{C_4}n^{1-\alpha}\mathbb{E}\left[\ \left(\abs{Z_j}\tilde{C}(1+\varepsilon^q\abs{Z_j}^q)+r_n\tilde{C}(1+\varepsilon^qr_n^q)\right)1_{\{\abs{Z_j}>r_n\}} \right].
\end{align*}
Now, for any $k\in\N\cup\{0\}$, the Laplace approximation (see e.g.\ \cite[Chapter 2]{Wong01}) implies 
\begin{align*}
\mathbb{E}\left[\abs{Z}^k1_{\{\abs{Z}>r_n\}}\right] = 2\int_{\{x>r_n\}}x^k\frac{1}{\sqrt{2\pi}}e^{-\frac{x^2}{2}}dx = 2r_n^{k-\frac{1}{2}}e^{-\frac{r_n^2}{2}}(1+o(1)).
\end{align*}
Hence, there exist $C_4'\in(0,\infty)$ and $N_4\in\N$ such that for $n\geq N_4$ and $k\in\N\cup\{0\}$,
\[
\mathbb{E}\left[\abs{Z}^k1_{\{\abs{Z}>r_n\}}\right]\leq C_4'r_n^{k}e^{-\frac{r_n^2}{2}}.
\]
The last three displays together yield the following bound on the tail probability of $I^n_{22}$:
\begin{align*}
\mathbb{P}(I^n_{22}>C_4)&\leq\frac{C_4'}{C_4}n^{1-\alpha} \left(r_n\tilde{C}(1+\varepsilon^qr_n^q)\right)e^{-r_n^2/2}.
\end{align*}
Since this is summable in $n$ due to the definition of $r_n$ in \eqref{eq-r},
the first inequality in \eqref{eq-b2} follows from the Borel-Cantelli lemma.

Since \eqref{eq-I2} and \eqref{eq-I3} imply $I^n_{23}\leq \mathbb{E}[I^n_{22}]$, the second inequality in \eqref{eq-b2} follows from the first. This concludes Step 2B. Moreover, when combined with \eqref{eq-123}, Step 2A and Step B prove the claim of Step 2.

\end{proof}

\bibliographystyle{imsart-nameyear}
\bibliography{LDP_prefactor,refs11}

\end{document}